\documentclass{amsart}

\usepackage{amsmath, graphicx, amsthm, amssymb, booktabs}
\usepackage{yhmath} 
\usepackage{enumitem}
\usepackage[initials]{amsrefs}
\usepackage{mathrsfs} 

\usepackage{tikz} 
\usetikzlibrary{matrix, decorations.pathmorphing, positioning, patterns}
\usetikzlibrary{arrows,automata,backgrounds, petri}
\usetikzlibrary{shapes.multipart}
\usetikzlibrary{lindenmayersystems}
\usetikzlibrary{cd}
 
\theoremstyle{plain}
\newtheorem{theorem}{Theorem}[section]
\newtheorem{corollary}[theorem]{Corollary}
\newtheorem{proposition}[theorem]{Proposition}
\newtheorem{lemma}[theorem]{Lemma}

\theoremstyle{definition}
\newtheorem{definition}[theorem]{Definition}

\newtheorem{remark}[theorem]{Remark}

\newtheorem{notation}[theorem]{Notation}

\newcommand{\LLL}{\mathscr{L}}
\newcommand{\XXX}{\mathcal{X}}
\newcommand{\ZZZ}{\mathcal{Z}}
\newcommand{\TTT}{\mathcal{T}}

\newcommand{\dd}{\mathrm{d}}
\newcommand{\sign}{\mathrm{sign}}

\DeclareMathOperator{\Tr}{Tr}
\DeclareMathOperator{\Ht}{Ht}

\newenvironment{psmallmatrix}
  {\left(\begin{smallmatrix}}
  {\end{smallmatrix}\right)}

\begin{document}

\title{Lagrange spectrum of a circle over the Eisensteinian field}

\author{Byungchul Cha}
\email{cha@muhlenberg.edu}
\author{Heather Chapman}
\email{heathercchapman@muhlenberg.edu}
\author{Brittany Gelb}
\email{brittany.gelb@rutgers.edu}
\author{Chooka Weiss}
\email{ChookaWeiss@muhlenberg.edu}
\address{Muhlenberg College, 2400 W.~Chew st., Allentown, PA 18104, USA}
\address{Department of Mathematics, Rutgers University, 110 Frelinghuysen rd., Piscataway, NJ 08854, USA}

\date{Oct 28, 2021}

\begin{abstract}
We study an intrinsic Lagrange spectrum of the unit circle $|z|=1$ in the complex plane with respect to the Eisensteinian field $\mathbb{Q}(\sqrt{-3})$. 
We prove that the minimum of the Lagrange spectrum is $2$ and that its smallest accumulation point is $4/\sqrt{3}$. 
In addition, we characterize the set of all values in the spectrum between $2$ and $4/\sqrt3$.
\end{abstract}
\maketitle

\section{Introduction}
\subsection{Motivation}\label{SecMotivation}
Call $(a, b)$ an \emph{Eisenstein pair} if $a$ and $b$ are positive integers such that
$ a^2 + ab + b^2 $ is a perfect square.
Plot all Eisenstein pairs $(a, b)$ on the complex plane by associating $(a, b)$ to $a + b\omega$ where $\omega = \frac{1 + \sqrt{-3}}2$.
Suppose that we draw a half-line $\ell$ from the origin into 
the subset
$
\{ z \in \mathbb{C} \mid 0 \le \arg(z)  \le \pi/3 \}
$
of the complex plane and we aim to make $\ell$ stay as far away as possible from all but finitely many Eisenstein pairs (see Figure~\ref{FigEisensteinPairs}).
\begin{figure}[b]
    \centering
	\begin{tikzpicture}
  \pgfdeclarelindenmayersystem{triangular grid}{\rule{F->F-F+++F--F}}
  \path[draw=black,
        dotted,
  l-system={triangular grid,step=1cm,
    angle=-60,axiom=F--F--F,order=3,
  }]
  lindenmayer system -- cycle;

        \node[below] at (0, 0) {$0$};

        \filldraw (1, 0) circle (1pt);
        \node[below] at (1, 0) {$1$};

        \filldraw (0.5, 0.87) circle (1pt);
        \node[left] at (0.5, 0.85) {$\omega$};

        \node[fill=white,above] at (6.5, 2.6) {$5 + 3\omega$};
        \filldraw (6.5, 2.6) circle (1pt);

        \node[fill=white,above] at (5.5, 4.33) {$3 + 5\omega$};
        \filldraw (5.5, 4.33) circle (1pt);

        \draw (1, 0) arc (0:60:1);
        \draw[->] (0, 0) -- (6.92, 4) node[right] {$\ell(z)$};

        \node[right] at (0.87, 0.5) {$z$};
        \filldraw (0.87, 0.5) circle (1pt);
		\node[right] at (8, 0) {$\dots$};
		\node[above right] at (4, 6.928) {$\adots$};
\end{tikzpicture}
	\caption{Eisenstein pairs and a half-line $\ell(z)$}
    \label{FigEisensteinPairs}
\end{figure}
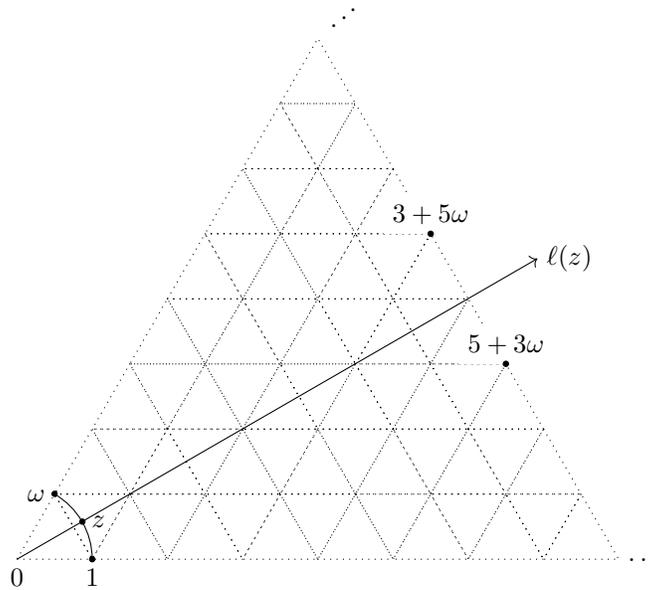
What is the greatest possible margin by which $\ell$ misses all but finitely many Eisenstein pairs?
What is the second greatest?

To formulate this question more precisely, 
let $z = \alpha + \beta\omega$
for any nonnegative real numbers $\alpha$ and $\beta$ with  $\alpha^2 + \alpha\beta + \beta^2 = 1$
and let $\ell(z)$ be the half-line in the complex plane which begins at 0 and passes through $z$.
Denote by $\delta'(z, (a, b))$ the shortest (Euclidean) distance from $a+b\omega$ to the half-line $\ell(z)$. 
Then we are interested in maximizing 
\[
	\delta(z) = \liminf_{(a, b)} \delta'(z, (a,b))
\]
where the Eisenstein pairs $(a,b)$ are ordered by 
the absolute value $\lvert a+b\omega \rvert = \sqrt{a^2 + ab + b^2}$.
Equivalently, we are interested in minimizing
\[
L(z) := \frac1{\delta(z)}
	= \limsup_{(a, b)} \frac1{\delta'(z, (a,b))}.
\]
Our first theorem provides answers for the questions we asked in the beginning.

\begin{theorem}\label{ThmMainTheoremI}
	The 5 smallest values of $L(z)$, together with corresponding $z$, are 
	as in Table~\ref{tab:Table1}.
\end{theorem}
\begin{table}
    \caption{The 5 smallest values of $L(z)$.}
\[
	   \begin{array}{@{} ll @{}}
			\toprule
			L(z) & z\\
			\midrule
			2 &  \frac1{\sqrt3} +  \frac1{\sqrt3}\omega
			 \\
			 \sqrt{\frac{13}{3}} = 
			 2.08166599946613 \dots
			&
			 \frac{1}{2}+\frac{-1+\sqrt{13}}{4} \omega
			\\
			\frac{\sqrt{133}}5 =
			2.30651251893416 \dots
			&
		   \frac{10}{\sqrt{399}} + \frac{13}{\sqrt{399}} \omega
			\\
			\frac{2\sqrt{15841}}{109}
			=
			2.30937677942215 \dots
			&
		   \frac{109}{\sqrt{47523}} + \frac{142}{\sqrt{47523}} \omega
			 \\
            \frac{2\sqrt{1884961}}{1189}
            =
            2.30940087256323\dots
            &
		   \frac{1189}{\sqrt{5654883}} + \frac{1549}{\sqrt{5654883}} \omega
			 \\
			\bottomrule
		\end{array}
\]
    \label{tab:Table1}
\end{table}

In fact, for each value of $L(z)$, there are (infinitely) many $z$ which share the same $L(z)$. 
Table~\ref{tab:Table1} lists only one such $z$.

\begin{theorem}\label{ThmMainTheoremII}
The smallest accumulation point of 
$
	\{ L(z) \mid \lvert z\rvert = 1, 
	L(z)>0 \}
$ is 
\[\frac4{\sqrt{3}} =
2.30940107675850\dots.
\]
\end{theorem}

\subsection{Intrinsic Diophantine approximation} \label{SecIntrinsicDiophantineApproximation}
To place our results in a general context, we introduce some notions about 
\emph{intrinsic Diophantine approximation}, following the exposition in \cite{FKMS}.
Suppose that $(\XXX, \dd)$ is a complete metric space and that  $\mathcal{Z}$ is a countable dense subset of $\XXX$ equipped with a \emph{height function} 
\[
\Ht: \mathcal{Z} \longrightarrow (0, \infty),
\]
namely, a function whose inverse image of any bounded subset of $(0, \infty)$ is finite.
We will call $(\XXX, \dd, \ZZZ, \Ht)$ (or more simply $(\XXX,\ZZZ)$) a \emph{Diophantine space}.
With respect to a Diophantine space $(\XXX, \ZZZ)$,
we say that a function $\psi:(0, \infty) \longrightarrow (0, \infty) $ is a \emph{Dirichlet function} if it has the property that, for each $P\in \XXX$, there exists a constant $C(P)$ and a sequence $\{Z_n\}_{n=1}^{\infty}$ in $\ZZZ$ such that
$
\dd(P, Z_n) \le C(P) \psi(\Ht(Z_n))
$
for all $n\ge 0$.
Once the data 
$(\XXX, \dd, \ZZZ, \Ht, \psi)$ 
are fixed, 
we define an \emph{approximation constant} $\delta(P)$ of each $P \in \XXX - \ZZZ$ to be 
\begin{equation}\label{EqDefApproxConstantWithDirichlet}
\delta(P) = 
\liminf_{\substack{Z\in\mathcal{Z}\\ \Ht(Z)\to\infty}} 
\frac{\dd(P, Z)}{\psi(\Ht(Z))}.
\end{equation}
The approximation constant $\delta(P)$ is thought to measure the approximability of $P$; this means that
the smaller $\delta(P)$ is the better $P$ is approximated by points in $\ZZZ$.
In particular, $P$ is said to be \emph{badly approximable} if $\delta(P)>0$.
There is a notion of an \emph{optimal} Dirichlet function, which is 
equivalent to the existence of badly approximable points
under some technical conditions.
We refer the reader to \cite{FSU14} for more in-depth discussion on this.
In our paper, we will simply pick a function $\psi$ which is known to be an optimal Dirichlet function (thus guaranteeing that badly approximable points exist), and we study the resulting approximation constants.
For instance, if $\XXX = \mathbb{R}^n$ and $\ZZZ = \mathbb{Q}^n$
(with the distance in $\XXX$ being given by the supremum norm and $\Ht(\mathbf{p}/q) = q$ with primitive $\mathbf{p}\in \mathbb{Z}^n$ and $q>0$),
it is well-known that the function $\psi: H \mapsto H^{-(1 + \frac1n)}$ is an optimal Dirichlet function.
In addition to the approximation constant $\delta(P)$, we define
the \emph{Lagrange number} $L(P)$ of $P \in \XXX- \ZZZ$ to be

\begin{equation}\label{EqDefLagrange}
L(P) =
\frac1{\delta(P)} = 
\limsup_{\substack{Z\in\mathcal{Z}\\ \Ht(Z)\to\infty}} 
\frac{\psi(\Ht(Z))}{\dd(P, Z)}.
\end{equation}
If $\delta(P) = 0$, then we let $L(P) = \infty$.
We also define the \emph{Lagrange spectrum} $\LLL(\XXX,\ZZZ)$ to be
\[
\LLL(\XXX,\ZZZ) =
\{ L(P) \mid P\in \XXX-\ZZZ, L(P)<\infty \}.
\]

Suppose that $(\XXX_1, \ZZZ_1, \dd_1, \Ht_1)$
and $(\XXX_2, \ZZZ_2, \dd_2, \Ht_2)$ are Diophantine spaces.
Assume that there exists an isometry $\iota: (\XXX_1, \dd_1) \longrightarrow (\XXX_2, \dd_2)$ such that
\begin{itemize}
    \item $\iota$ maps $\ZZZ_1$ bijectively onto $\ZZZ_2$, and
    \item $\iota$ preserves heights, that is, $\Ht_2(\iota(Z)) = \Ht_1(Z)$ for all $Z\in \ZZZ_1$.
\end{itemize}
In other words, the isometry $\iota$ preserves the structure of ``rational points''.
Let us call such an $\iota$ a \emph{Diophantine isometry}. 
The Diophantine spaces $(\XXX_1,\ZZZ_1)$ and $(\XXX_2, \ZZZ_2)$ then share a common Dirichlet function $\psi$
and we have $\delta(P) = \delta(\iota(P))$ for all $P \in \XXX_1 - \ZZZ_1$.
As a result, $\LLL(\XXX_1,\ZZZ_1) = \LLL(\XXX_2,\ZZZ_2)$.

\subsection{Main theorem}
Let $K=\mathbb{Q}(\sqrt{-3})$, which we call the \emph{Eisensteinian field}.
In the present paper, we are concerned with the following Diophantine space:
\[
\begin{cases}
\XXX_{\mathbb{C}} = \{ z\in \mathbb{C} \mid \lvert z\rvert  = 1 \}, \\
\ZZZ_K =  \XXX_{\mathbb{C}} \cap K. \\
\end{cases}
\]
Here, $\XXX_{\mathbb{C}}$ is equipped with the usual Euclidean metric in the complex plane. 
To define a height function on $\ZZZ_K$, notice that
any element in $y\in K$ is written uniquely in the form
\[
y =
\frac{a + b\omega}{c}
\]
with $a, b, c\in \mathbb{Z}$ having no common factor and $c>0$.
(Recall $\omega= (1 + \sqrt{-3})/2$.)
The height function $\Ht_K(y)$ is defined by
\[
\Ht_K(y) = \Ht_K
\left(
\frac{a + b\omega}{c}
\right)
 = c.
\]
Finally, we let $\psi: H\mapsto H^{-1}$, which is a Dirichlet optimal function in this case.
With respect to this data $(\XXX_{\mathbb{C}}, \ZZZ_K,\Ht_K, \psi)$,
the definition \eqref{EqDefLagrange} becomes
\[
L(z) = 
\limsup_{\substack{y\in\mathcal{Z}\\ \Ht(y)\to\infty}} 
\frac1{
\Ht_K(y) \dd(z, y)}.
\]
It is easy to see that this $L(z)$ coincides with the $L(z)$ defined in \S\ref{SecMotivation}. 
\begin{theorem}\label{ThmMainTheoremIntro}
Define 
\[
\lambda =
\frac{3 + \sqrt{13}}2
\quad\text{and}\quad
\overline{\lambda} =
\frac{3 - \sqrt{13}}2.
\]
Then we have 
\begin{align*}
\LLL&(\XXX_{\mathbb{C}},\ZZZ_K) \cap \left(0, \frac{4}{\sqrt3}
\right)
    \\
    &=
    \left\{
    \sqrt{\frac{13}{3}}
    \right\}
    \cup
    \left\{
    \frac2{\sqrt{3}} \cdot
\sqrt{4 - \frac{13}%
{\left(\lambda^{2k+1} - \overline{\lambda}^{2k+1}\right)^2}}
\Biggm\vert
k = 0, 1, \dots
\right\} \\
&=
\left\{
2,
    \sqrt{\frac{13}{3}},
    \frac{\sqrt{133}}5,
    \frac{2}{109}\sqrt{15841},
    \frac{2}{1189}\sqrt{1884961},
    \dots
\right\}.
\end{align*}
\end{theorem}
This theorem is an easy consequence of
Theorem~\ref{ThmMainTheorem}, 
which is our main theorem. 
Also, 
given any element $L$ in 
$\LLL(\XXX_{\mathbb{C}},\ZZZ_K) \cap \left(0, \frac{4}{\sqrt3}\right)$, 
Theorem~\ref{ThmMainTheorem} provides a recipe to produce a point $z\in \XXX_{\mathbb{C}}$ with 
$L = L(z)$.
All assertions in Theorems~\ref{ThmMainTheoremI} and \ref{ThmMainTheoremII} easily follow from this.

\subsection{Literature review}
In the study of Lagrange spectra of various Diophantine spaces, perhaps the most prototypical example is the case $(\XXX, \ZZZ) = (\mathbb{R}, \mathbb{Q})$ with the usual Euclidean metric on $\mathbb{R}$ and the height function being $\Ht(\frac ab) = \lvert b \rvert $.
The function  $\psi: H\mapsto H^{-2}$ is an optimal Dirichlet function here.
With this set-up, a celebrated theorem of Markoff in \cite{Mar79} and \cite{Mar80} 
characterizes every (irrational) number $P$ whose Lagrange number $L(P) < 3$ in terms of its continued fraction expansion.
In particular, Markoff's theorem proves that $3$ is the smallest accumulation point of $\LLL(\mathbb{R},\mathbb{Q})$.
For detailed discussion of this remarkable theorem, we refer the reader to \cite{CF89} and \cite{Mal77}, as well as to \cite{Aig13} and \cite{Reu19} for more recent expositions.

There are numerous generalizations of Markoff's theorem in various contexts.
However, structures of Lagrange spectra of \emph{intrinsic} Diophantine approximation seem to be less known.
See the introduction in \cite{FKMS} for an overview on this topic.
In \cite{KM15} Kleinbock and Merrill studied intrinsic Diophantine approximation for $(\XXX, \ZZZ) = (S^n, S^n\cap \mathbb{Q}^{n+1})$, where $S^n$ is the unit $n$-sphere in $\mathbb{R}^{n+1}$.
Here, the metric on $S^n$ is the sup norm in $\mathbb{R}^{n+1}$
and the height of $\mathbf{z}/q\in S^{n+1}$ is $q$ whenever $\mathbf{z}\in\mathbb{Z}^{n+1}$ is a primitive integral vector and $q$ is a positive integer.
They show in \cite{KM15} that there exist badly approximable points in $S^n$ with respect to the function $\psi: H \mapsto H^{-1}$, therefore, establishing that $\psi$ is an optimal Dirichlet function in this setting.
As for the Lagrange spectrum of 
$\LLL(S^n, S^n\cap \mathbb{Q}^{n+1})$,
Kopetzky \cite{Kop80} and, independently, Moshchevitin \cite{Mos16} found that the minimum (or \emph{the Hurwitz bound}) of $\LLL(S^1, S^1\cap \mathbb{Q}^2)$ is $1/\sqrt{2}$.

In 2008, Romik introduced in \cite{Rom08} a certain dynamical system on the unit quarter circle $\{ (x, y) \in \mathbb{R}^2 \mid x, y\ge0, \, x^2 + y^2 =1\}$ and used it to define a certain digit expansion for points in the quarter circle. 
Romik's dynamical system is based on an old theorem of Berggren \cite{Ber34}, which provides a tree-like structure for the set of all \emph{Pythagorean triples} $(a, b, c)$, namely, a triple of coprime positive integers with $a^2 + b^2 = c^2$.
Jointly with Dong Han Kim in \cite{CK19}, the first-named author of the present paper utilized Romik's digit expansions to reveal the structure of the initial discrete part of 
$\LLL(S^1, S^1\cap \mathbb{Q}^2)$.
(In \cite{CK19}, the metric on $S^1$ is the usual Euclidean metric on $\mathbb{R}^2$, not the sup norm.) 
In particular, they prove that $2$ is the smallest accumulation point of 
$\LLL(S^1, S^1\cap \mathbb{Q}^2)$ 
and characterize those $P$ in $S^1$ whose Lagrange numbers $L(P)<2$ 
in terms of Romik's digit expansions of $P$.
This provides an analogue of the aforementioned theorem of Markoff in the context of intrinsic Diophantine approximation of $S^1$.
A similar but less direct result had been previously proven by Kopetzky in \cite{Kop85}.

\subsection{Comparison with A.~Schmidt's Markoff spectrum}
Let $f(x, y) = \alpha x^2 + \beta xy + \gamma y^2$ be a quadratic form with $\alpha, \beta, \gamma\in \mathbb{R}$ and $\alpha \neq 0$ and discriminant $\Delta(f) = \beta^2 - 4\alpha\gamma >0$.
In \cite{Sch77}, A.~Schmidt defines the \emph{3-minimum} of $f$ to be
\[
  \mu_3(f) = \inf_{(x, y) \in \mathbb{Z}^2 - \{ (0, 0) \}}
  \left\vert
  \frac{f(x, y)}{\gcd(x, 3)}
  \right\vert,
\]
and the \emph{3-spectrum}
$\mathfrak{M}_3$
to be
\[
  \mathfrak{M}_3 = 
  \{
    \sqrt{\Delta(f)}/\mu_3(f) \mid f \text{ as above}
  \}.
\]
His main result (Theorem 2.1 in \cite{Sch77}) gives
\begin{align*}
  \mathfrak{M}_3 \cap (0, 4) &= 
  \{
    \sqrt{13}
  \}
  \cup
  \left\{
    4\sqrt{1 - \frac4{M_n^2}}
    \biggm\vert
    \begin{array}[]{l}
      M_0 = 4, M_1 = 40, \\
      M_n = 11M_{n-1} - M_{n-2} \text{ for } n\ge 2
    \end{array}
  \right\} \\
  &= 
  \left\{
    2\sqrt3, 
    \sqrt{13}, 
     \frac{\sqrt3\sqrt{133}}5,
     \frac{2\sqrt3\sqrt{15841}}{109},
     \frac{2\sqrt3\sqrt{1884961}}{1189},
     \dots 
  \right\}.
\end{align*}
To connect this with our results, let $\lambda, \overline{\lambda}$ be as in Theorem~\ref{ThmMainTheoremIntro}.
Then one can use induction to show that, for $n\ge0$,
\[
  M_n = \frac4{\sqrt{13}}
  \left(
  \lambda^{2n+1} - \overline{\lambda}^{2n+1}
\right).
\]
Combining this and Theorem~\ref{ThmMainTheoremIntro}, we have
\[
  \frac{\mathfrak{M}_3 \cap (0, 4)}{\sqrt3} = 
  \mathscr{L}(\mathcal{X}_{\mathbb{C}}, \mathcal{Z}_K) \cap (0, \tfrac{4}{\sqrt3} ).
\]
This is analogous to the well-known fact that the original Markoff and Lagrange spectra coincide below 3, their common smallest accumulation point.
It would be desirable to provide a more detailed explanation on the connection between our Lagrange spectrum and A.~Schmidt's Markoff spectrum. 
Indeed, Kopetzky's result in \cite{Kop85} can be interpreted as providing this connection in the \emph{Pythagorean} case.
We intend to return to this theme in the near future. 

\subsection{Diophantine isometry and structure of the paper}
Instead of dealing with $(\XXX_{\mathbb{C}}, \ZZZ_K)$ directly,
we consider another Diophantine space $(\XXX_0,\ZZZ_0)$ with
\begin{equation}\label{EqDefinitionXZ2}
\begin{cases}
\XXX_0 = \{ (x, y) \in\mathbb{R}^2 \mid x^2 + xy + y^2 = 1 \}\\
\mathcal{Z}_0=  \XXX_0 \cap\mathbb{Q}^2. 
\end{cases}
\end{equation}
To define a metric on $\XXX_0$, we equip the ambient vector space $\mathbb{R}^2$ of $\XXX_0$ with an inner product 
\begin{equation}\label{EqDefInnerProduct}
P_1\cdot P_2 = x_1x_2 + \frac{x_1y_2 + x_2y_1}2 + y_1y_2
\end{equation}
for $P_1 = (x_1, y_1)$ and $P_2 = (x_2, y_2)$ in $\mathbb{R}^2$
and let $\XXX_0$ inherit the metric from the inner product space $(\mathbb{R}^2, \cdot)$.
With respect to this metric, $\XXX_0$ is the ``unit circle'' centered at the origin 
and $\ZZZ_0$ is the set of its rational points.
Also, we define a height function $\Ht$ on $\ZZZ_0$ to be
\[
\Ht(\tfrac ac, \tfrac bc) = c,
\]
whenever $(\frac ac, \frac bc)\in \ZZZ_0$ is written in lowest terms with $c>0$.
Finally we choose our Dirichlet function $\psi$ to be $\psi: H \longmapsto H^{-1}$.

One can easily check  that the map
\begin{equation}\label{EqDefZmap}
z: 
\mathbb{R}^2 \longrightarrow  
\mathbb{C},
\qquad
P = (\alpha, \beta) \mapsto
z(P) = \alpha+ \beta\omega
\end{equation}
is an isometry (see Figure~\ref{FigTheMapFromCtoEllipse}).
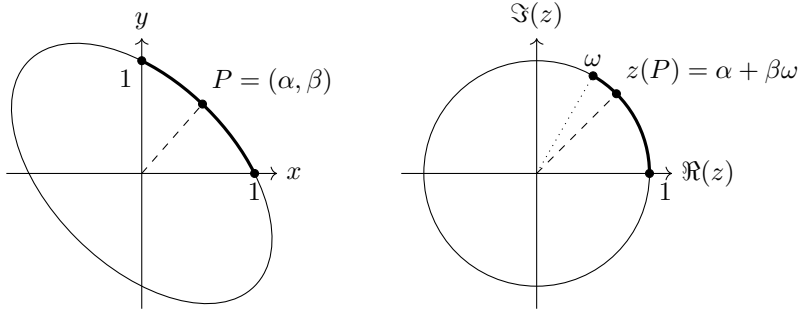
\begin{figure}
\centering
\begin{tikzpicture}[scale=1.5]
        %
        %
        %
        %
      
        \tikzset{
    partial ellipse/.style args={#1:#2:#3}{
        insert path={+ (#1:#3) arc (#1:#2:#3)}
    }
}
        \draw[->] (-1.2, 0) -- (1.2, 0) node[right]{$x$};
        \draw[->] (0, -1.2) -- (0, 1.2) node[above]{$y$};

        \draw[rotate=-45] (0,0) ellipse (1.414cm and 0.816cm); 
        \draw[very thick, rotate=-45] (0,0) [partial ellipse=60:120:1.414cm and 0.816cm];
        \draw[dashed] (0, 0) -- (7/13, 8/13) node[above right]{$P = (\alpha, \beta)$};
        \filldraw (7/13, 8/13) circle (1.0pt);
        \filldraw (1, 0) circle (1.0pt) node[below]{$1$};
        \filldraw (0, 1) circle (1.0pt) node[below left]{$1$};

    \draw[->] (2.3, 0) -- (4.7, 0) node[right]{$\Re(z)$};
    \draw[->] (3.5, -1.2) -- (3.5, 1.2) node[above]{$\Im(z)$};
    
      \draw (3.5, 0) circle (1cm);
    \draw[very thick] (4.5, 0) arc (0:60:1cm);
    \draw[dotted] (3.5, 0) -- (4, 0.866);
    \draw[dashed] (3.5, 0) -- (4.207, 0.707) node[above right]{$z(P) = \alpha+ \beta\omega$};
    
        \filldraw (4.5, 0) circle (1.0pt) node[below right]{$1$};
        \filldraw (4, 0.866) circle (1.0pt) node[above]{${\omega}$};
        \filldraw (4.207, 0.707) circle (1.0pt);
\end{tikzpicture}
    \caption{The Diophantine isometry $P\mapsto z(P)$}
    \label{FigTheMapFromCtoEllipse}
\end{figure}
Furthermore, this map sends $\XXX_0$ bijectively onto $\XXX_{\mathbb{C}}$ and $\ZZZ_0$ bijectively onto $\ZZZ_K$, preserving
heights.
In other words, $z: (\XXX_0, \ZZZ_0) \longrightarrow (\XXX_{\mathbb{C}}, \ZZZ_K)$ is a \emph{Diophantine isometry} (see \S\ref{SecIntrinsicDiophantineApproximation}).
By symmetry, it is enough for us to consider a ``one-sixth'' $(\XXX, \ZZZ)$ of $(\XXX_0, \ZZZ_0)$, which is defined to be
\[
\XXX = \{ (x, y) \in \XXX_0 \mid x, y\ge 0\}
\]
and $\ZZZ = \ZZZ_0 \cap \XXX$.
Clearly $\LLL(\XXX, \ZZZ) = \LLL(\XXX_0, \ZZZ_0)$.

Let us call $(a, b, c)$ an \emph{Eisenstein triple} if $(a, b, c)$ is a coprime positive integer triple satisfying $a^2 + ab + b^2 = c^2$.
The set of all Eisenstein triples are in one-to-one correspondence with points in $\ZZZ$. 
On the other hand, it is proven by Wayne in \cite{Way82} that the set of all Eisenstein triples forms a certain tree-like structure, just as in the Pythagorean case.
This enables us to apply the same strategy developed in \cite{CK19} to our situation.

In fact, the current paper should be thought of as a companion paper to \cite{CK19} in the \emph{Eisensteinian} case.
We begin in \S\ref{SecRomik} by outlining constructions of Romik's dynamical system and digit expansions, making necessary modifications for the Eisensteinian case.
We will give a self-contained exposition for the sake of completeness at the expense of some duplication.

A central theme in the techniques developed in \cite{CK19} is the fact that Romik's digit expansion plays an analogous role as continued fraction expansions in the classical case.
The same is true for the Eisensteinian case. 
One of the main technical results in \S\ref{SecRomik} is to establish the fact that, for a fixed $P \in \XXX-\ZZZ$, all the best approximants of $P$ are contained in boundary points of the cylinder sets containing $P$ (Theorem~\ref{ThmBestApproximants}).
Another key step is to prove a version of \emph{Perron's formula} in the Eisensteinian case (Theorem~\ref{ThmPerronFormula}).

Once Perron's formula is established, it becomes possible to define a \emph{doubly infinite Romik sequence} $T$ and its Lagrange number $L(T)$ 
(see Definition~\ref{DefLagrangeSequence}),
so that to each $P\in \XXX - \ZZZ$ 
we can associate a doubly infinite Romik sequence $T$ such that $L(T) = L(P)$ (see Lemma~\ref{LemBombieri}).
Therefore, in order to characterize every Lagrange number $L(P) < 4/\sqrt3$, it is enough to characterize every doubly infinite Romik sequence $T$ with $L(T) < 4/\sqrt3$. 
This is the goal of \S\ref{SecAdmissible} and we give such a characterization in Theorem~\ref{ThmMainTheorem}, which is the main theorem of the paper.

Finally we mention that there is another Diophantine space $(\XXX', \ZZZ')$, whose intrinsic Diophantine approximation is equivalent to $(\XXX_0, \ZZZ_0)$.
Define
\[
\XXX' = \{ (x, y, z) \in \mathbb{R}^3 \mid x^2 + y^2 + z^2 = 2, \quad x + y + z = 0 \}
\]
and $\ZZZ' =  \XXX' \cap \mathbb{Q}^3$.
To define a metric on $\XXX'$, we simply rescale the Euclidean metric on $\mathbb{R}^3$ and let $\dd' = \dd_{\mathrm{Euclidean}}/\sqrt2$, so that $\XXX'$ becomes a ``unit circle'' in $\mathbb{R}^3$.
In addition, we define the height of $\mathbf{p}/q \in \ZZZ'$ to be $q$ whenever $\mathbf{p}\in \mathbb{Z}^3$ is primitive and $q>0$ is a positive integer.
Let $\mathbf{i} = (1, -1, 0)$ and $\mathbf{j} = (1, 0, -1)$. 
Then it is easy to check that the map
\[
\iota: \mathbb{R}^2 \longrightarrow \{ (x, y, z) \in \mathbb{R}^3 \mid x+ y+ z = 0\}
\]
with $\iota(\alpha, \beta) = \alpha \mathbf{i} + \beta\mathbf{j}$ 
is a Diophantine isometry between 
$(\XXX_0, \ZZZ_0)$ and $(\XXX', \ZZZ')$.
Recall that $\mathbb{R}^2$ is equipped with the metric defined by the inner product in \eqref{EqDefInnerProduct}.
In particular, $\LLL(\XXX, \ZZZ) = \LLL(\XXX', \ZZZ')$.

\subsection{Acknowledgments}

We are thankful to an anonymous referee, who brought to our attention a paper \cite{Sch77} by A.~Schmidt and prompted us to consider the connection between our Lagrange spectrum and Schmidt's Markoff spectrum.
We are also thankful to another referee whose comments helped us to improve the paper. Dong Han Kim discovered the Diophantine isometry $\iota$ between $(\XXX_0, \ZZZ_0)$ and $(\XXX', \ZZZ')$.
We are grateful to him for having us include it here and for other useful discussions while this paper was being written. 
\section{Romik's dynamical system and Berggren trees of Eisenstein triples}
\label{SecRomik}
\subsection{Preliminary results on linear algebra}\label{SubsecPrelim}
Let $(\mathbb{R}^3, Q(\mathbf{x}))$ be a \emph{quadratic space}, namely, a real vector space $\mathbb{R}^3$ equipped with a quadratic form
\[ 
Q(\mathbf{x}) = x_1^2 + x_1x_2 + x_2^2 - x_3^2
\]
for $\mathbf{x} = (x_1,x_2, x_3)\in \mathbb{R}^3$.
Associated to the form $Q(\mathbf{x})$, there is a symmetric bilinear pairing $\langle \cdot, \cdot \rangle$
\begin{equation}\label{EqDefPairing}
\begin{aligned}
\langle \mathbf{x}, \mathbf{y} \rangle
&= 
\frac12
\left(
Q(\mathbf{x}+\mathbf{y}) - Q(\mathbf{x}) - Q(\mathbf{y})
\right) \\
&= x_1y_1 + \frac{x_1y_2 + x_2y_1}2 + x_2y_2 - x_3y_3
\end{aligned}
\end{equation}
for $\mathbf{x} = (x_1, x_2, x_3)$ and $\mathbf{y} = (y_1, y_2, y_3)$.

We also consider an inner product space $(\mathbb{R}^2, \cdot)$, where the inner product $P_1\cdot P_2$ is defined by
\begin{equation}\label{EqDefInnerProduct2}
P_1\cdot P_2 = x_1x_2 + \frac{x_1y_2 + x_2y_1}2 + y_1y_2
\end{equation}
whenever $P_1 = (x_1, x_2)$ and $P_2 = (y_1, y_2)$.
This dot product, being positive definite, defines a metric on $\mathbb{R}^2$ and we will henceforth regard $\mathbb{R}^2$ as a metric space using this metric.
Let $\XXX_0$ be a ``unit circle'', that is, 
\begin{equation}\label{EqUnitCircle}
\XXX_0 = \{(x,y) \in \mathbb{R}^2 \mid x^2+xy+y^2=1 \}.
\end{equation}
For $P = (\alpha, \beta)\in \XXX_0$, denote by $\theta(P)$ the angle with $0\le \theta(P) < 2\pi$ satisfying
\begin{equation}\label{EqArgumentRelations}
\begin{cases}
\cos(\theta(P)) = \alpha + \frac12\beta, \\
\sin(\theta(P)) = \frac{\sqrt3}2\beta.
\end{cases}
\end{equation}
Also, we write 
\begin{equation}\label{EqDefPrecEq}
    P_1 \preceq P_2
\end{equation}
whenever $\theta(P_1) \le \theta(P_2)$.
Geometrically speaking, $\theta(P)$ is the angle measured from $(1, 0)$ to $P$ counterclockwise using the inner product \eqref{EqDefInnerProduct2}.
If $P_1, P_2\in \XXX_0$, we write
\begin{equation}\label{EqDefAngleP1P2}
\theta(P_1, P_2) = \theta(P_2)- \theta(P_1).
\end{equation}

\begin{definition}\label{DefVectors}
	Let $\mathbf{x} = (x_1, x_2, x_3) \in \mathbb{R}^3$. We say that $\mathbf{x}$ 
\begin{enumerate}[font=\upshape, label=(\alph*)]
	\item is \emph{normalized} (or $x_3$-\emph{normalized}) if $x_3 = 1$,
	\item is \emph{positive} if $x_3 > 0$, and
	\item \emph{represents} a point $(x, y)\in \mathbb{R}^2$ if $x = x_1/x_3$ and $y = x_2/x_3$. 
\end{enumerate}
When $P = (x, y)\in \mathbb{R}^2$, we will denote by $(P, 1)$ a normalized vector $ (x, y, 1)$.
\end{definition}

The bilinear pairing $\langle \cdot, \cdot \rangle$ and the inner product in \eqref{EqDefInnerProduct2} are related in the following obvious way.
If $\mathbf{x}_1 = (P_1, 1)$ and $\mathbf{x}_2 = (P_2, 1)$, then 
\begin{equation}\label{EqPairingDotProduct}
\langle \mathbf{x}_1, \mathbf{x}_2 \rangle
=
P_1\cdot P_2 - 1,
\end{equation}
where $P_1\cdot P_2$ is the inner product defined in \eqref{EqDefInnerProduct}.
Moreover, we have the following proposition.
\begin{proposition}\label{PropBilinearPairingInnerProductRelation}
If $P_1, P_2\in \XXX_0$ and if $\mathbf{x}_1 = (P_1, 1)$ and $\mathbf{x}_2 = (P_2, 1)$, then 
\begin{equation*}
\langle \mathbf{x}_1, \mathbf{x}_2 \rangle
=
-2\sin^2\left(
\frac{\theta(P_1, P_2)}2
\right).
\end{equation*}
\end{proposition}
\begin{proof}
Writing $P_1 = (x_1, y_1)$ and $P_2 = (x_2, y_2)$,  we have
\begin{align*}
    -2\sin^2\left(
    \frac{\theta(P_1, P_2)}2
    \right)
    &= 2\cos^2\left(
    \frac{\theta(P_1, P_2)}2
    \right) - 2\\
    &= \cos(\theta(P_2) - \theta(P_1)) - 1\\
    &= \cos\theta(P_2)\cos\theta(P_1) + \sin\theta(P_2)\sin\theta(P_1) - 1 \\
    &= \left(
    x_1 + \frac{y_1}{2}\right)\left(x_2 + \frac{y_2}{2}\right) + \frac{\sqrt{3}}{2}y_1\cdot\frac{\sqrt{3}}{2}y_2 - 1\\
    &= \langle \mathbf{x}_1, \mathbf{x}_2 \rangle.
\end{align*}
\end{proof}
\begin{lemma}\label{LemNegativeOnCQ}
Let $\mathbf{p}_1$ and $\mathbf{p}_2$ be vectors with positive $x_3$-coordinates such that $Q(\mathbf{p}_1) = Q(\mathbf{p}_2) = 0$.
Then we have $\langle \mathbf{p}_1, \mathbf{p}_2 \rangle \le 0$,
with the equality holing only when $\mathbf{p}_1$ and $\mathbf{p}_2$ are (positive) scalar multiples of one another.
\end{lemma}
\begin{proof}
Without loss of generality, we may assume that both $\mathbf{p}_1$ and $\mathbf{p}_2$ are normalized. 
The conclusion of the lemma then follows immediately from Proposition~\ref{PropBilinearPairingInnerProductRelation}.
\end{proof}

We say that an invertible linear map $A$ from the quadratic space $(\mathbb{R}^3, Q(\mathbf{x}))$ to itself is \emph{orthogonal} with respect to $Q(\mathbf{x})$ 
if $Q(A\mathbf{x}) = Q(\mathbf{x})$ for any $\mathbf{x} \in \mathbb{R}^3$.
Equivalently, $A$ is orthogonal if and only if 
$
\langle A\mathbf{x}, A\mathbf{y} \rangle = 
\langle \mathbf{x}, \mathbf{y} \rangle
$
for any $\mathbf{x}, \mathbf{y} \in \mathbb{R}^3$.


\begin{definition}[See \S3.2.4 of \cite{CNT}]\label{DefHandUd}
Define $H$ to be the linear map of $\mathbb{R}^3$ onto itself given by the matrix
\[
H =
\begin{pmatrix}
  -4 & -3 & 4 \\
  -3 & -4 & 4 \\
  -6 & -6 & 7 \\
\end{pmatrix}
\]
with respect to the standard basis of $\mathbb{R}^3$.
Additionally, define $U_1, \dots, U_5$ by the following matrices:
\[
\begin{gathered}
U_1 = 
\begin{pmatrix}
  0 & -1 & 0 \\
  1 & 1 & 0 \\
  0 & 0 & 1 \\
\end{pmatrix},
\quad
U_2 = 
\begin{pmatrix}
  -1 & -1 & 0 \\
  0 & 1 & 0 \\
  0 & 0 & 1 \\
\end{pmatrix},
\quad
U_3 = 
\begin{pmatrix}
  -1 & 0 & 0 \\
  0 & -1 & 0 \\
  0 & 0 & 1 \\
\end{pmatrix},
\\
U_4 = 
\begin{pmatrix}
  1 & 0 & 0 \\
  -1 & -1 & 0 \\
  0 & 0 & 1 \\
\end{pmatrix},
\quad
U_5 = 
\begin{pmatrix}
  1 & 1 & 0 \\
  -1 & 0 & 0 \\
  0 & 0 & 1 \\
\end{pmatrix}.
\end{gathered}
\]

Finally, we define 
\[ M_d = HU_d\]
for $d = 1,\dots, 5$.
Explicitly, we have
\begin{equation}\label{EqDefintionM}
\begin{gathered}
M_1=\begin{pmatrix}
-3&1&4\\
-4&-1&4\\
-6&0&7
\end{pmatrix},
\quad
M_2=\begin{pmatrix}
4&1&4\\
3&-1&4\\
6&0&7
\end{pmatrix},
\quad
M_3=\begin{pmatrix}
4&3&4\\
3&4&4\\
6&6&7
\end{pmatrix}, \\
M_4=\begin{pmatrix}
-1&3&4\\
1&4&4\\
0&6&7
\end{pmatrix},
\qquad
M_5=\begin{pmatrix}
-1&-4&4\\
1&-3&4\\
0&-6&7
\end{pmatrix}.
\end{gathered}
\end{equation}
\end{definition}
\begin{lemma}[\S3.2.4 in \cite{CNT}]\label{Lem324}
The maps $U_1, \dots, U_5$ and $H$ are orthogonal with respect to $Q(\mathbf{x})$.
(As a result, $M_1, \dots, M_5$ are also orthogonal.)
Furthermore, 
define
\[
\hat{d}
=
\begin{cases}
5 & \text{ if } d= 1, \\
2 & \text{ if } d= 2, \\
3 & \text{ if } d= 3, \\
4 & \text{ if } d= 4, \\
1 & \text{ if } d= 5. \\
\end{cases}
\]
Then we have
\begin{enumerate}[font=\upshape, label=(\alph*)]
    \item $H = H^{-1}$, 
    \item $U_{\hat{d}} = U_d^{-1},$
    \item $M_d = HU_d,$ and $M_d^{-1} = U_{\hat{d}}H.$
\end{enumerate}
\end{lemma}
\begin{proof}
All these can be verified by straightforward calculation.
\end{proof}

Before we finish this subsection, we will prove a geometric lemma (Lemma~\ref{LemPreservationOfShade}), which will be used later.
First, we define a subset $\Omega$ of $\mathbb{R}^2$ to be
\begin{equation}\label{EqDefOmega}
\Omega = \{ (x, y)\in \mathbb{R}^2 \mid x + y \ge 1, \quad x^2 + xy + y^2 \le 1 \},
\end{equation}
which is shown as a gray region in Figure~\ref{FigOmega}.
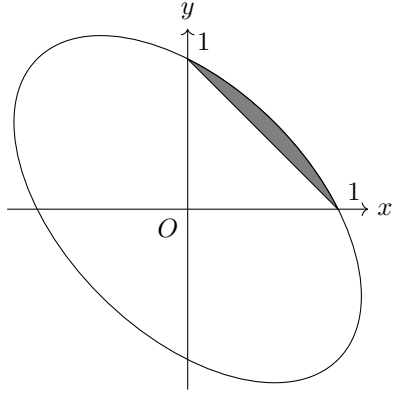
\begin{figure}
\begin{center}
\begin{tikzpicture}[scale=2]
        \tikzset{
    partial ellipse/.style args={#1:#2:#3}{
        insert path={+ (#1:#3) arc (#1:#2:#3)}
    }
}
        \draw[->] (-1.2, 0) -- (1.2, 0) node[right]{$x$};
        \draw[->] (0, -1.2) -- (0, 1.2) node[above]{$y$};
        \node[below left] at (0, 0) {$O$};
        \node[above right] at (1, 0) {$1$};
        \node[above right] at (0, 1) {$1$};

\draw[rotate=-45] (0, 0) [partial ellipse=0:360:1.414cm and 0.8165cm];


\draw[fill=gray, rotate=-45] (0, 0) [partial ellipse=60:120: 1.414cm and 0.8156cm] --  cycle;


\end{tikzpicture}
\end{center}
\caption{The subset $\Omega$ of $\mathbb{R}^2$ is shown as a gray region.} 
\label{FigOmega}
\end{figure}
It is easy to see that $\mathbf{y}$ represents a point in $\Omega$ whenever
\begin{itemize}
	\item[($\Omega$-I)] $\mathbf{y}$ is positive,
	\item[($\Omega$-II)] $Q(\mathbf{y})\le 0$, and 
	\item[($\Omega$-III)] $\langle \mathbf{y}, (2, 2, 3)^T \rangle \ge 0$.
\end{itemize}
Also, if $\mathbf{y} = (y_1, y_2, y_3)$ is a positive vector representing a point in $\Omega$, then
\begin{itemize}
	\item[($\Omega$-IV)] $y_1, y_2\ge0$, $y_3 > 0$, and
	\item[($\Omega$-V)] $y_3 \ge y_1$ and $y_3 \ge y_2$.
\end{itemize}
\begin{lemma}\label{LemPreservationOfShade}
Suppose that $\mathbf{y}$ is a positive vector representing a point in $\Omega$.
Then the vectors $\mathbf{y}_d' = M_d \mathbf{y}$ for $d = 1, \dots, 5$ are positive and they represent points in $\Omega$.
\end{lemma}
\begin{proof}
Write $\mathbf{y} = (y_1, y_2, y_3)$.
A straightforward calculation shows that the $x_3$-coordinates of $\mathbf{y}_d'$ are 
\[
	7y_3 \pm 6y_1, \quad
	7y_3 \pm 6y_2, \quad \text{or} \quad
	6y_1 + 6y_2 + 7y_3.
\]
From ($\Omega$-IV) and ($\Omega$-V) above, we see that all these quantities are positive and therefore $\mathbf{y}_d'$ is positive.
Also, $Q(\mathbf{y}'_d) = Q(\mathbf{y})\le0$ because of the orthogonality of $M_d$ (Lemma~\ref{Lem324}).
It remains to prove that $\mathbf{y}_d'$ satisfies the condition ($\Omega$-III) above.

From the orthogonality of $M_d$ again, we have
\[
	\langle \mathbf{y}_d', (2, 2, 3)^T \rangle
	=
	\langle M_d\mathbf{y}, (2, 2, 3)^T \rangle
	=
	\langle \mathbf{y}, M_d^{-1}(2, 2, 3)^T \rangle.
\]

An easy calculation shows
\[
	M_d^{-1}
\begin{pmatrix} 2 \\  2 \\ 3 \end{pmatrix} 
	=
	\begin{cases}
		(2, -4, -3)^T & \text{ if } d = 1, \\
		(4, -2, -3)^T & \text{ if } d = 2, \\
		(2, 2, -3)^T & \text{ if } d = 3, \\
		(-2, 4, -3)^T & \text{ if } d = 4, \\
		(-4, 2, -3)^T & \text{ if } d = 5. \\
	\end{cases}
\]
From this, we get
\[
	\langle \mathbf{y}, M_d^{-1}(2, 2, 3)^T \rangle
	=
	\begin{cases}
	3y_3 - 3y_2 & \text{ if } d = 1, \\
	3y_3 + 3y_1 & \text{ if } d = 2, \\
	3y_1 + 3y_2 + 3y_3& \text{ if } d = 3, \\
	3y_3 + 3y_2 & \text{ if } d = 4, \\
	3y_3 - 3y_1 & \text{ if } d = 5. \\
	\end{cases}
\]
From ($\Omega$-IV) and ($\Omega$-V), we see that all these quantities are nonnegative.
This completes the proof of the lemma. 
\end{proof}

\subsection{Berggren trees of Eisenstein triples and Romik's dynamical system}
We review a result of Wayne in \cite{Way82} on Eisenstein triples. 
Wayne's theorem states that every Eisenstein triple is obtained from one of the four Eisenstein triples---$(8, 7, 13)$,
$(3, 5, 7)$,
$(5, 3, 7)$,
$(7, 8, 13)$---by successively multiplying $M_{d_1}$, $\dots$, $M_{d_k}$ for some $[d_1, \dots, d_k]\in \{ 1, 2, 3, 4, 5\}^k$ in a unique way (see Definition~\ref{DefHandUd}).
Another way of stating this theorem is that the set of all Eisenstein triples forms four quinary trees (that is, directed trees with each vertex having one incoming edge and 5 outgoing edges, except for the base vertices which have no incoming edges), where each edge in the tree denotes left-multiplication.
Such trees are called \emph{Berggren trees} in \cite{CNT} and Wayne's theorem is a direct analogue of a much older theorem of Berggren regarding Pythagorean triples. 
It will be convenient for us to treat $(1, 0, 1)$ and $(0, 1, 1)$ as Eisenstein triples and add them to the Berggren trees, as drawn in Figure~\ref{FigTrees}.
\begin{figure}
\begin{tikzpicture}[->, auto, xscale=1.1]
\tikzset{triarrow/.pic={
    \draw (0, 0) -- (0, 0.5) node[above]{$\vdots$};
    \draw (-0.1, 0) -- (-0.2, 0.5) node[above]{$\vdots$};
    \draw (0.1, 0) -- (0.2, 0.5) node[above]{$\vdots$};
    \draw (-0.2, 0) -- (-0.4, 0.5) node[above]{$\vdots$};
    \draw (0.2, 0) -- (0.4, 0.5) node[above]{$\vdots$};
  }
}

\node (subbase) at (0, -1) {$(1, 0, 1)$};
\node (base1) at (-1.5, -0.2) {$(8, 7, 13)$};
\node (base2) at (1.5, -0.2) {$(3, 5, 7)$};

\node (1f1) at (-4.4, -0.1) {$(35, 13, 43)$};
\node (1f2) at (-4.65, 1.9) {$(91, 69, 139)$};
\node (1f3) at (-3.35, 2.4) {$(105, 104, 181)$};
\node (1f4) at (-2, 1.9) {$(65, 88, 133)$};
\node (1f5) at (-.8, 1.3) {$(16, 39, 49)$};

\node (1f10) at (4.4, -0.1) {$(5, 16, 19)$};
\node (1f9) at (4.6, 1.9) {$(40, 51, 79)$};
\node (1f8) at (3.4, 2.4) {$(55, 57, 97)$};
\node (1f7) at (2, 1.9) {$(45, 32, 67)$};
\node (1f6) at (.8, 1.3) {$(24, 11, 31)$};

\pic at (4.4, .2)  {triarrow};
\pic at (4.6, 2.2) {triarrow};
\pic at (3.4, 2.7) {triarrow};
\pic at (2, 2.2) {triarrow};
\pic at (.8, 1.6) {triarrow};

\pic at (-4.4, .2)  {triarrow};
\pic at (-4.7, 2.2) {triarrow};
\pic at (-3.4, 2.7) {triarrow};
\pic at (-2, 2.2) {triarrow};
\pic at (-.8, 1.6) {triarrow};

\draw (subbase) to[bend left=20] node[left, font=\footnotesize]{${M_2}$} (base1);
\draw (subbase) to[bend right=20] node[above, font=\footnotesize]{${M_3}$} (base1);

\draw (subbase) to[bend left=20] node[above, font=\footnotesize]{${M_4}$} (base2);
\draw (subbase) to[bend right=20] node[right, font=\footnotesize]{${M_5}$} (base2);

\draw (subbase) to [out=-120,in=-60,loop,looseness=4] 
	node[below, font=\footnotesize]{${M_1}$} (subbase);
	
\draw (base1) to node[font=\footnotesize, below]{${M_1}$} (1f1);
\draw (base1) to node[font=\footnotesize, below]{${M_2}$} (1f2);
\draw (base1) to node[font=\footnotesize, left]{${M_3}$} (1f3);
\draw (base1) to node[font=\footnotesize, right]{${M_4}$} (1f4);
\draw (base1) to node[font=\footnotesize, right]{${M_5}$} (1f5);

\draw (base2) to node[font=\footnotesize, left]{${M_1}$} (1f6);
\draw (base2) to node[font=\footnotesize, left]{${M_2}$} (1f7);
\draw (base2) to node[font=\footnotesize, right]{${M_3}$} (1f8);
\draw (base2) to node[font=\footnotesize, below]{${M_4}$} (1f9);
\draw (base2) to node[font=\footnotesize, below]{${M_5}$} (1f10);
\end{tikzpicture}

\begin{tikzpicture}[->, auto, xscale=1.1]
\centering
\tikzset{triarrow/.pic={
    \draw (0, 0) -- (0, 0.5) node[above]{$\vdots$};
    \draw (-0.1, 0) -- (-0.2, 0.5) node[above]{$\vdots$};
    \draw (0.1, 0) -- (0.2, 0.5) node[above]{$\vdots$};
    \draw (-0.2, 0) -- (-0.4, 0.5) node[above]{$\vdots$};
    \draw (0.2, 0) -- (0.4, 0.5) node[above]{$\vdots$};
  }
}
\node (subbase) at (0, -1) {$(0, 1, 1)$};
\node (base1) at (-1.5, -0.2) {$(5, 3, 7)$};
\node (base2) at (1.5, -0.2) {$(7, 8, 13)$};

\node (1f1) at (-4.4, -0.1) {$(16, 5, 19)$};
\node (1f2) at (-4.6, 1.9) {$(51, 40, 79)$};
\node (1f3) at (-3.4, 2.4) {$(57, 55, 97)$};
\node (1f4) at (-2, 1.9) {$(32, 45, 67)$};
\node (1f5) at (-.8, 1.3) {$(11, 24, 31)$};

\node (1f10) at (4.4, -0.1) {$(13, 35, 43)$};
\node (1f9) at (4.7, 1.9) {$(69, 91, 139)$};
\node (1f8) at (3.4, 2.4) {$(104, 105, 181)$};
\node (1f7) at (2, 1.9) {$(88, 65, 133)$};
\node (1f6) at (.8, 1.3) {$(39, 16, 49)$};

\pic at (4.4, 0.2)  {triarrow};
\pic at (4.7, 2.2) {triarrow};
\pic at (3.4, 2.7) {triarrow};
\pic at (2, 2.2) {triarrow};
\pic at (.8, 1.6) {triarrow};

\pic at (-4.4, 0.2)  {triarrow};
\pic at (-4.6, 2.2) {triarrow};
\pic at (-3.4, 2.7) {triarrow};
\pic at (-2, 2.2) {triarrow};
\pic at (-.8, 1.6) {triarrow};

\draw (subbase) to[bend left=20] node[left, font=\footnotesize]{${M_1}$} (base1);
\draw (subbase) to[bend right=20] node[above, font=\footnotesize]{${M_2}$} (base1);

\draw (subbase) to[bend left=20] node[above, font=\footnotesize]{${M_3}$} (base2);
\draw (subbase) to[bend right=20] node[right, font=\footnotesize]{${M_4}$} (base2);

\draw (subbase) to [out=-120,in=-60,loop,looseness=4] 
	node[below, font=\footnotesize]{${M_5}$} (subbase);

\draw (base1) to node[font=\footnotesize, below]{${M_1}$} (1f1);
\draw (base1) to node[font=\footnotesize, below]{${M_2}$} (1f2);
\draw (base1) to node[font=\footnotesize, left]{${M_3}$} (1f3);
\draw (base1) to node[font=\footnotesize, right]{${M_4}$} (1f4);
\draw (base1) to node[font=\footnotesize, right]{${M_5}$} (1f5);

\draw (base2) to node[font=\footnotesize, left]{${M_1}$} (1f6);
\draw (base2) to node[font=\footnotesize, left]{${M_2}$} (1f7);
\draw (base2) to node[font=\footnotesize, right]{${M_3}$} (1f8);
\draw (base2) to node[font=\footnotesize, below]{${M_4}$} (1f9);
\draw (base2) to node[font=\footnotesize, below]{${M_5}$} (1f10);
\end{tikzpicture}
\caption{Berggren trees of Eisenstein triples}\label{FigTrees}
\end{figure}
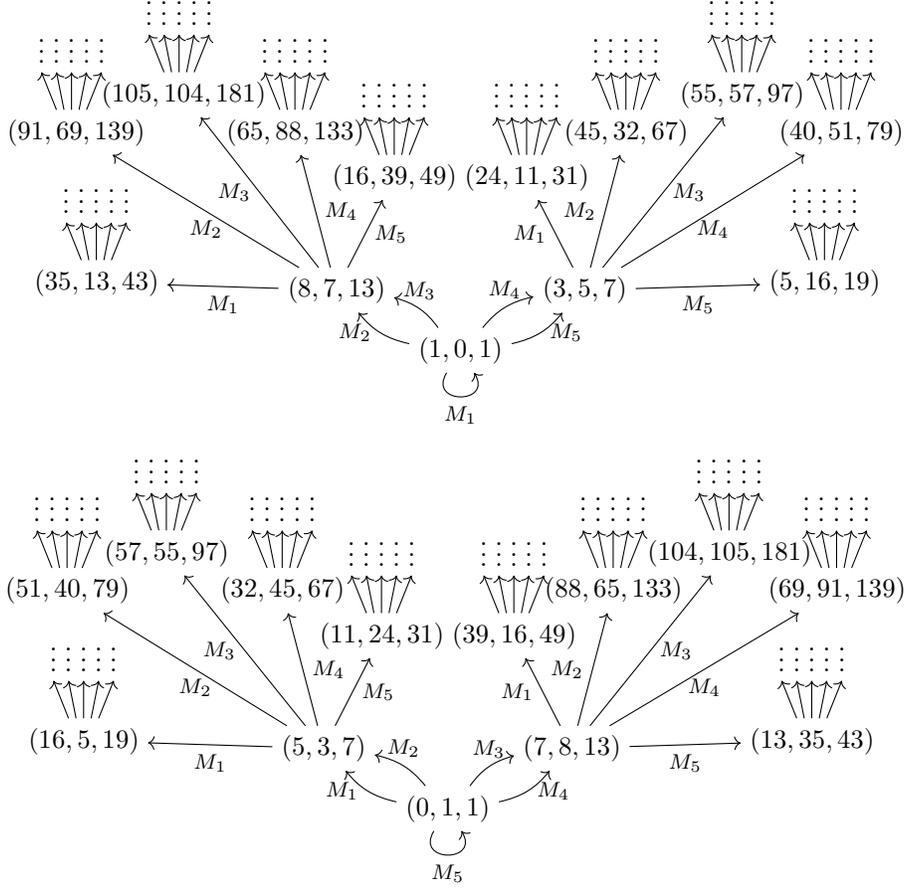

\begin{definition}[Romik's dynamical system]\label{DefRomikSystem}
Let
\[
\XXX = \{(x,y) \in \mathbb{R}^2 \mid x^2+xy+y^2=1, x, y \geq 0 \}
\]
and define $\TTT: \XXX \longrightarrow \XXX$ to be
\[
\TTT(x, y) = 
\begin{cases}
\left( \dfrac{3x+4y-4}{-6x-6y+7}, \dfrac{-7x-7y+8}{-6x-6y+7} \right) & \text{if } 0 \leq x \leq \frac{3}{7}, \\
\left( \dfrac{-4x-3y+4}{-6x-6y+7}, \dfrac{7x+7y-8}{-6x-6y+7} \right) & \text{if } \frac{3}{7} < x \leq \frac{7}{13}, \\
\left( \dfrac{4x+3y-4}{-6x-6y+7}, \dfrac{3x+4y-4}{-6x-6y+7} \right) &
\text{if } \frac{7}{13} < x \leq \frac{8}{13}, \\
\left( \dfrac{7x+7y-8}{-6x-6y+7}, \dfrac{-3x-4y+4}{-6x-6y+7} \right) & \text{if } \frac{8}{13} < x \leq \frac{5}{7}, \\
\left( \dfrac{-7x-7y+8}{-6x-6y+7}, \dfrac{4x+3y-4}{-6x-6y+7} \right) &
\text{if } \frac{5}{7} < x \leq 1.
\end{cases}
\]
We will call the dynamical system $(\XXX, \TTT)$ \emph{the Romik system}. 
\end{definition}

For each $P = (x, y) \in \XXX$, we define the \emph{Romik digit} $d(P)$ of $P$ to be
\begin{equation}\label{EqDefRomikDigits}
d(P) =
\begin{cases}
1 & \text{if } \frac{5}{7} \leq x \leq 1, \\
2 & \text{if } \frac{8}{13} \leq x \leq \frac{5}{7}, \\
3 & \text{if } \frac{7}{13} \leq x \leq \frac{8}{13}, \\
4 & \text{if } \frac{3}{7} \leq x \leq \frac{7}{13}, \\
5 & \text{if } 0 \leq x \leq \frac{3}{7},
\end{cases}
\end{equation}
as pictured in Figure~\ref{FigRomikDigits}.
Notice in the definition \eqref{EqDefRomikDigits} that there are four (rational) 
points---$(\frac{5}{7}, \frac{3}{7})$, $(\frac{8}{13}, \frac{7}{13})$, $(\frac{7}{13}, \frac{8}{13}),$ and $(\frac{3}{7}, \frac{5}{7})$---whose 
Romik digits are not uniquely defined.  
It will be convenient for us to regard those points 
as having \textit{two} valid Romik digits. 
\begin{figure}
\centering
\begin{tikzpicture}[scale=2.5]
      
        \tikzset{
    partial ellipse/.style args={#1:#2:#3}{
        insert path={+ (#1:#3) arc (#1:#2:#3)}
    }
}
        \draw[->] (-.2, 0) -- (2.2, 0) node[right]{$x$};
        \draw[->] (0, -.2) -- (0, 2.2) node[above]{$y$};
        \node[below left] at (0, 0) {$O$};
                                                     
        \draw[very thick, rotate=-45] (0,0) [partial ellipse=60:120:2.828cm and 1.632cm];
        
        \draw[fill] (6/7, 10/7) circle (0.02) node[left] {};
        \draw (6/7, 1.50857142857) node[below left] {$(\frac37, \frac57)$};
        \draw[fill] (14/13, 16/13) circle (0.02) node[left] {};
        \draw (14/13, 1.16076923077) node[left] {$(\frac{7}{13}, \frac{8}{13})$};
        \draw[fill] (16/13, 14/13) circle (0.02) node[left] {};
        \draw (1.28076923077, 1.09692307692) node[below left] {$(\frac{8}{13}, \frac{7}{13})$};
        \draw[fill] (10/7, 6/7) circle (0.02);
        \draw (1.45857142857, 0.89714285714) node[below left] {($\frac{5}{7}, \frac{3}{7}$)};
        \draw[fill] (0, 2) circle (0.02);
        \draw[fill] (2, 0) circle (0.02);
        
        \draw[dotted, thick, <->, rotate=-45] (0, 0) [partial ellipse=60:82:2.9694cm and 1.7136cm]
        node[midway, above right] {$\footnotesize{d=1}$};
        \draw[dotted, thick, <->, rotate=-45] (0, 0) [partial ellipse=82:87.9:2.9694cm and 1.7136cm]
        node[midway, above right] {$\footnotesize{d=2}$};
        \draw[dotted, thick, <->, rotate=-45] (0, 0) [partial ellipse=87.9:92.3:2.9694cm and 1.7136cm]
        node[midway, above right] {$\footnotesize{d=3}$};
        \draw[dotted, thick, <->, rotate=-45] (0, 0) [partial ellipse=92.3:98:2.9694cm and 1.7136cm]
        node[midway, above right] {$\footnotesize{d=4}$};
        \draw[dotted, thick, <->, rotate=-45] (0, 0) [partial ellipse=98:120:2.9694cm and 1.7136cm]
        node[midway, above right] {$\footnotesize{d=5}$};
        
\end{tikzpicture}
\caption{Romik digits of $P$}
\label{FigRomikDigits}
\end{figure}
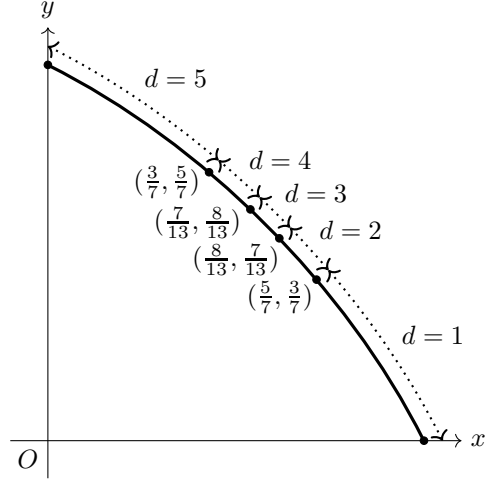
The following proposition reveals how the Romik system $(\XXX, \TTT)$ is related to the Berggren trees of Eisenstein triples.
\begin{proposition}\label{PropRomiksystemBerggrenTrees}
Let $\mathbf{p}$ be a positive vector representing a point $P$ in $\XXX$.
Then the vector ${M_{d(P)}}^{-1} \mathbf{p}$ is a positive vector representing $\TTT(P)$.
\end{proposition}
\begin{proof}
To prove this, we observe from \eqref{EqDefintionM} that
\[
\begin{gathered}
M_1^{-1} =
\begin{pmatrix}
-7 & -7 & 8 \\
4 & 3 & -4 \\
-6 & -6 & 7
\end{pmatrix},
M_2^{-1} =
\begin{pmatrix}
7 & 7 & -8 \\
-3 & -4 & 4 \\
-6 & -6 & 7
\end{pmatrix},
M_3^{-1} =
\begin{pmatrix}
4 & 3 & -4 \\
3 & 4 & -4 \\
-6 & -6 & 7
\end{pmatrix}, 
\\
M_4^{-1} =
\begin{pmatrix}
-4 & -3 & 4 \\
7 & 7 & -8 \\
-6 & -6 & 7
\end{pmatrix},
\quad
M_5^{-1} =
\begin{pmatrix}
3 & 4 & -4 \\
-7 & -7 & 8 \\
-6 & -6 & 7
\end{pmatrix}
\end{gathered}
\]
and compare this with Definition~\ref{DefRomikSystem} and \eqref{EqDefRomikDigits}.
The proof follows from this immediately.
\end{proof}

\subsection{Romik sequences and cylinder sets}
To each $P\in \XXX$, we associate an infinite sequence $\{ d_j \}_{j=1}^{\infty}$ in $\{1, 2, 3, 4, 5\}$, which is defined by
\[
d_j=d(\TTT^{j-1}(P))
\]

for $j\in \mathbb{N}$.
Such a sequence is called a \emph{Romik digit expansion} of $P$
and we write 
\[
P = [d_1, d_2, \dots]_{\XXX}.
\]

The next proposition says that $\TTT$ is an analogue of the Gauss map for ordinary continued fractions.
\begin{proposition}
\label{prop:TshiftMap}
If $P = [d_1, d_2, \dots]_{\XXX} \in \XXX$
then $\TTT(P) = [d_2, d_3, \dots]_{\XXX}$.
In other words, $\TTT$ is a shift-map on the digit expansion of $P$.
\end{proposition}

\begin{proof}
This is clear from the definition $d_j = d(\TTT^{j-1}(P))$.
\end{proof}

Note that both points $(1, 0)$ are $(0, 1)$ are fixed by $\TTT$ and therefore
\[
(1, 0) = [1, 1, 1, \dots]_{\XXX} = [1^{\infty}]_{\XXX}
\text{ and }
(0, 1) = [5, 5, 5, \dots]_{\XXX} = [5^{\infty}]_{\XXX}.
\]
Because of the ambiguity in Romik digits for the four points---$(\frac{5}{7}, \frac{3}{7})$, $(\frac{8}{13}, \frac{7}{13})$, $(\frac{7}{13}, \frac{8}{13})$, and $(\frac{3}{7}, \frac{5}{7})$,
they admit \emph{two} Romik digit expansions:
\begin{equation}\label{EqFourRationalPts}
\begin{cases}
   \left(\frac{5}{7}, \frac{3}{7}
   \right) = 
   [1,  5^{\infty}]_{\XXX} \text{ or }
   [2, 5^{\infty}]_{\XXX}, \\
   \left(
   \frac{8}{13}, \frac{7}{13}
   \right) =
   [2, 1^{\infty}]_{\XXX} \text{ or }
   [3, 1^{\infty}]_{\XXX} ,\\
   \left(
   \frac{7}{13}, \frac{8}{13}
   \right) =
   [3, 5^{\infty}]_{\XXX} \text{ or }
   [4, 5^{\infty}]_{\XXX} ,\\
   \left(
   \frac{3}{7}, \frac{5}{7}
   \right) = 
   [4, 1^{\infty}]_{\XXX} \text{ or }
   [5, 1^{\infty}]_{\XXX}.
\end{cases}
\end{equation}
If $P$ is a rational point on $\XXX$ not equal to $(1, 0)$ and $(0, 1)$,
Wayne's theorem implies that $\TTT^j(P)$ is equal to one of the four rational points in \eqref{EqFourRationalPts} for some $j\ge 0$.
Therefore, we conclude that a Romik digit expansion of every rational point, except for $(1, 0)$ and $(0, 1)$, terminates with two alternate tails as given in \eqref{EqFourRationalPts}.
However, every irrational point $P\in \XXX-\ZZZ$ has a unique Romik digit expansion.

For any finite sequence $d_1, \dots, d_k$ in $\{ 1, \dots, 5 \}$,
we define its \emph{cylinder set} 
\[ 
C(d_1, \dots, d_k) = 
\{P \in \XXX \mid d_j=d(\TTT^{j-1}(P)) \text{ for } j=1,\dots, k\}.
\]
\begin{figure}
\centering
\begin{tikzpicture}[scale=2.5]
        \tikzset{
    partial ellipse/.style args={#1:#2:#3}{
        insert path={+ (#1:#3) arc (#1:#2:#3)}
    }
}
        \draw[->] (-.2, 0) -- (2.2, 0) node[right]{$x$};
        \draw[->] (0, -.2) -- (0, 2.2) node[above]{$y$};
        \node[below left] at (0, 0) {$O$};
                                                     
        \draw[very thick, rotate=-45] (0,0) [partial ellipse=60:120:2.828cm and 1.632cm];
        
        \draw[fill] (6/7, 10/7) circle (0.02) node[left] {};
        \draw (6/7, 1.50857142857) node[below left] {$(\frac37, \frac57)$};
        \draw[fill] (14/13, 16/13) circle (0.02) node[left] {};
        \draw (14/13, 1.16076923077) node[left] {$(\frac{7}{13}, \frac{8}{13})$};
        \draw[fill] (16/13, 14/13) circle (0.02) node[left] {};
        \draw (1.28076923077, 1.09692307692) node[below left] {$(\frac{8}{13}, \frac{7}{13})$};
        \draw[fill] (10/7, 6/7) circle (0.02);
        \draw (1.45857142857, 0.89714285714) node[below left] {($\frac{5}{7}, \frac{3}{7}$)};
        \draw[fill] (0, 2) circle (0.02);
        \draw[fill] (2, 0) circle (0.02);
        
        \draw[dotted, thick, <->, rotate=-45] (0, 0) [partial ellipse=60:82:2.9694cm and 1.7136cm]
        node[midway, above right] {$\footnotesize{C(1)}$};
        \draw[dotted, thick, <->, rotate=-45] (0, 0) [partial ellipse=82:87.9:2.9694cm and 1.7136cm]
        node[midway, above right] {$\footnotesize{C(2)}$};
        \draw[dotted, thick, <->, rotate=-45] (0, 0) [partial ellipse=87.9:92.3:2.9694cm and 1.7136cm]
        node[midway, above right] {$\footnotesize{C(3)}$};
        \draw[dotted, thick, <->, rotate=-45] (0, 0) [partial ellipse=92.3:98:2.9694cm and 1.7136cm]
        node[midway, above right] {$\footnotesize{C(4)}$};
        \draw[dotted, thick, <->, rotate=-45] (0, 0) [partial ellipse=98:120:2.9694cm and 1.7136cm]
        node[midway, above right] {$\footnotesize{C(5)}$};
        
\end{tikzpicture}
\caption{Cylinder sets of length one}
\label{FigCylinderSetsLengthOne}
\end{figure}
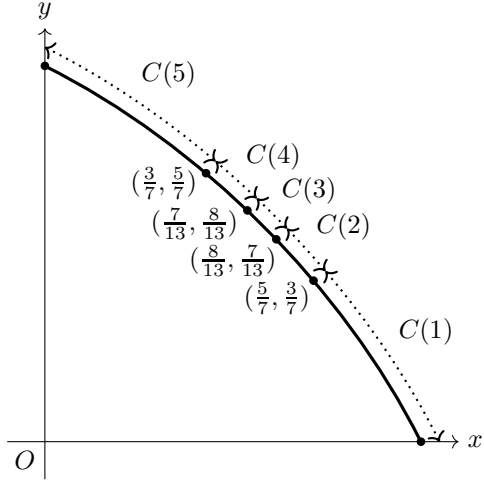
The cylinder sets of length one are pictured in Figure~\ref{FigCylinderSetsLengthOne}.
Topologically speaking, cylinder sets are closed sub-arcs of $\XXX$ with rational boundary points.
To describe this more explicitly, 
we let
\begin{equation}\label{EqDefU10andU01}
\mathbf{u}^{(1, 0)}
=
\begin{pmatrix}
  1 \\ 0 \\ 1
\end{pmatrix}
\quad
\text{and}
\quad
\mathbf{u}^{(0, 1)}
=
\begin{pmatrix}
  0 \\ 1 \\ 1
\end{pmatrix}.
\end{equation}
Once a finite sequence $d_1, \dots, d_k$ of Romik digits is fixed, we write
\begin{equation}\label{EqDefZ10andZ01}
\mathbf{z}^{(1, 0)}
=
M_{d_1} \cdots M_{d_k}
\mathbf{u}^{(1, 0)}
\quad
\text{and}
\quad
\mathbf{z}^{(0, 1)}
=
M_{d_1} \cdots M_{d_k}
\mathbf{u}^{(0, 1)}.
\end{equation}

Write 
$\mathbf{z}^{(1, 0)} = (a^{(1, 0)}, b^{(1, 0)}, c^{(1, 0)})$ 
and 
$\mathbf{z}^{(0, 1)} = (a^{(0, 1)}, b^{(0, 1)}, c^{(0, 1)})$,
and define rational points 
$Z^{(1, 0)}$ 
and
$Z^{(0, 1)}$ 
to be
\begin{equation}\label{EqDefZk}
Z^{(1,0 )} = 
\left(
\frac{a^{(1,0 )}}{c^{(1,0 )}},
\frac{b^{(1,0 )}}{c^{(1,0 )}}
\right)
\quad\text{and}\quad
Z^{(0, 1)} = 
\left(
\frac{a^{(0, 1)}}{c^{(0, 1)}},
\frac{b^{(0, 1)}}{c^{(0, 1)}}
\right).
\end{equation}
Then the cylinder set $C(d_1, \dots, d_k)$ is a closed sub-arc of $\XXX$ whose boundary points are $Z^{(1, 0)}$ and $Z^{(0, 1)}$.

\begin{definition}\label{DefZk}
Let $P\in \XXX - \ZZZ$ and write $P = [d_1, \dots, d_k, \dots]_{\XXX}$.
To each $k\ge0$ we define
\[
\mathbf{z}_k^{(1, 0)}(P)
=
\begin{pmatrix}
  a_k^{(1, 0)}(P) \\ b_k^{(1, 0)}(P) \\ c_k^{(1, 0)}(P)
\end{pmatrix}
=
M_{d_1} \cdots M_{d_k}
\mathbf{u}^{(1, 0)}
\]
and
\[
\mathbf{z}_k^{(0, 1)}(P)
=
\begin{pmatrix}
  a_k^{(0, 1)}(P) \\ b_k^{(0, 1)}(P) \\ c_k^{(0, 1)}(P)
\end{pmatrix}
=
M_{d_1} \cdots M_{d_k}
\mathbf{u}^{(0, 1)}.
\]
Also, define 
\[
Z_k^{(1, 0)}(P)
=
\left(
\frac{a_k^{(1,0 )}(P)}{c_k^{(1,0 )}(P)},
\frac{b_k^{(1,0 )}(P)}{c_k^{(1,0 )}(P)}
\right)
\quad\text{and}\quad
Z_k^{(0, 1)}(P) = 
\left(
\frac{a_k^{(0, 1)}(P)}{c_k^{(0, 1)}(P)},
\frac{b_k^{(0, 1)}(P)}{c_k^{(0, 1)}(P)}
\right).
\]
\end{definition}
For a finite Romik sequence $\{d_1, \dots, d_k\}$, define $\sign(d_1, \dots, d_k)$ to be 
\begin{equation}\label{EqDefSign}
\sign(d_1, \dots, d_k) 
=
\det(M_{d_1} \cdots M_{d_k}).
\end{equation}
An easy calculation shows 
\[
\det(M_d) = 
\begin{cases}
1 & \text{ if } d = 1, 3, 5, \\
-1 & \text{ if } d = 2, 4. \\
\end{cases}
\]
So we have
\[
\sign(d_1, \dots, d_k) = (-1)^{%
(\text{the number of occurrences of $2$ and $4$ in } \{ d_1, \dots, d_k \})
}.
\]
\begin{proposition}\label{PropMonotonicity}
Let $P_1, P_2\in C(d_1, \dots d_k)$, and 
let $P_1' = \TTT^k(P_1)$ and 
$P_2' = \TTT^k(P_2)$.
	\begin{enumerate}[font=\upshape, label=(\alph*)]
		\item Assume that $\sign(d_1, \dots, d_k) = 1$. Then $P_1 \preceq P_2$ if and only if $P_1' \preceq P_2'$.
		\item Assume that $\sign(d_1, \dots, d_k) = -1$. Then $P_1 \preceq P_2$ if and only if $P_2' \preceq P_1'$.
	\end{enumerate}
\end{proposition}
\begin{proof}

Recall from Proposition~\ref{prop:TshiftMap} that $\TTT$ is a shift-map on the digit expansion. In particular, this implies that, for each of $d = 1, \dots, 5$, the restriction $\TTT\vert_{C(d)}$ of $\TTT$ to a cylinder set $C(d)$ is a bijection $\TTT\vert_{C(d)}: C(d) \longrightarrow \XXX$.

Moreover, $\TTT\vert_{C(d)}$ is ``order-preserving'' (with respect to $\preceq$) when $d = 1, 3, 5$ and is ``order-reversing'' when $d = 2, 4$.
The proof of the proposition follows from induction on $k$.
\end{proof}
\begin{proposition}\label{PropInvariance}
Suppose $P \in \XXX$ with
\[
P = [d_1, \dots, d_k, \dots]_{\XXX}
\]
and let $P' = T^k(P)$ for some $k\ge 1$.
Let $\mathbf{p}$ and $\mathbf{p}'$ be normalized vectors representing $P$ and $P'$, that is, $\mathbf{p} = (P, 1)$ and $\mathbf{p}' = (P', 1)$. Then for any $\mathbf{z}_1$ and $\mathbf{z}_2$ in $\mathbb{R}^2$ we have
\[
\langle \mathbf{p}, \mathbf{z}_1 \rangle
\langle \mathbf{p}', M_{d_k}^{-1} \cdots M_{d_1}^{-1}\mathbf{z}_2 \rangle
=
\langle \mathbf{p}, \mathbf{z}_2 \rangle
\langle \mathbf{p}', 
M_{d_k}^{-1} \cdots M_{d_1}^{-1}
\mathbf{z}_1 \rangle
.
\]
\end{proposition}
\begin{proof}
Let $M  = M_{d_1} \cdots M_{d_k}$. 
By applying Proposition~\ref{PropRomiksystemBerggrenTrees} successively $k$ times, we conclude that
$\mathbf{p}' = \lambda M^{-1}\mathbf{p}$ for some positive scalar $\lambda$.
The left-hand side of the statement in the proposition then becomes
\[
\langle \mathbf{p}, \mathbf{z}_1 \rangle
\langle \mathbf{p}', M^{-1} \mathbf{z}_2 \rangle
=
\langle \mathbf{p}, \mathbf{z}_1 \rangle
\langle 
\lambda M^{-1}\mathbf{p}
, M^{-1} \mathbf{z}_2 \rangle
=
\lambda \langle \mathbf{p}, \mathbf{z}_1 \rangle
\langle 
\mathbf{p}
, \mathbf{z}_2 \rangle
\]
because of the orthogonality of $M$.
Likewise, the right-hand side is
\[
\langle \mathbf{p}, \mathbf{z}_2 \rangle
\langle \mathbf{p}', M^{-1} \mathbf{z}_1 \rangle
=
\lambda \langle \mathbf{p}, \mathbf{z}_2 \rangle
\langle 
\mathbf{p}
, \mathbf{z}_1 \rangle.
\]
This proves the proposition.
\end{proof}

\subsection{Heights, approximation constants and best approximants}
The metric space $\XXX$ has a countable dense subset 
\[
\ZZZ = \XXX \cap \mathbb{Q}^2.
\]
For each $Z\in \ZZZ$, there exists a unique nonnegative and coprime triple $(a, b, c)$ of integers satisfying $a^2 + ab + b^2 = c^2$ (that is, an \emph{Eisenstein triple}), so that $Z = (\frac ac, \frac bc)\in \ZZZ$.
We define the \emph{height} $\Ht(Z)$ of $Z$ to be
\[
\Ht(Z) = c.
\]
We can express $\Ht(Z)$ using the bilinear pairing in the following way. 
Let $\mathbf{z} = (a, b, c)$ be the primitive integral vector representing $Z = (\frac ac, \frac bc)$ as above and let
\[
	\mathbf{v}_Q = \begin{pmatrix}0  \\ 0 \\  1
	\end{pmatrix}.
\]
Then
\begin{equation}\label{EqHtVQ}
\Ht(Z) = 
-\langle \mathbf{z}, \mathbf{v}_Q \rangle.
\end{equation}
Finally, we define the \emph{approximation constant} $\delta(P)$ of $P \in \XXX - \ZZZ$ to be
\[
\delta(P) = \liminf_{\substack{Z\in \ZZZ \\ \Ht(Z) \to \infty}}
\Ht(Z)\dd(P, Z).
\]
Recall that the metric $\dd(\cdot, \cdot)$ on $\XXX$ comes from the inner product \eqref{EqDefInnerProduct2}.

\begin{definition}\label{DefDeltaBasic}
Let $P \in \XXX$ and $Z = (\frac ac, \frac bc)\in \ZZZ$.
Write $\mathbf{p}$ for the normalized vector representing $P$, that is, $\mathbf{p} = (P, 1)$ and $\mathbf{z} = (a, b, c)$.
We define $\delta(P; Z)$ to be the positive real number satisfying
\[
\delta^2(P; Z) = -2c
\langle
\mathbf{p}, \mathbf{z}
\rangle.
\]
Notice from Lemma~\ref{LemNegativeOnCQ} that $\langle \mathbf{p}, \mathbf{z} \rangle$ is always negative.
If we apply Proposition~\ref{PropBilinearPairingInnerProductRelation}, this becomes equivalent to
	\begin{equation}\label{EqDeltaCsin}
\delta(P; Z) = 2c
\left\vert
\sin
\left(
\frac{\theta(P, Z)}{2}
\right)
\right\vert.
	\end{equation}
\end{definition}
It is easy to see that, as $Z\to P$, we have
\[
\frac{\dd(P, Z)}{2\lvert\sin(\theta(P, Z)/2)\rvert } \to 1.
\]
It follows from this that
\begin{equation}\label{EqDeltaEqualsPairing}
\delta(P) = \liminf_{\substack{Z \in \ZZZ \\ \Ht(Z)\to\infty}}
\delta(P; Z).
\end{equation}

For the rest of this subsection, our focus is to prove the following statements:
\begin{itemize}
	\item The heights of boundary points of a cylinder set are less than or equal to those of interior points (Theorem~\ref{PropInteriorHeight}).
	\item Fix $P = [d_1, d_2, \dots, d_k, \dots]\in \XXX - \ZZZ$. Then the set of all boundary points of cylinder sets $\{ C(d_1, \dots, d_k) \}_{k=1}^{\infty}$ will contain all the \emph{best approximants} of $P$. (Theorem~\ref{ThmBestApproximants}) 
\end{itemize}

\begin{notation}\label{NotYj}
	Fix a finite Romik sequence $\{ d_1, \dots, d_k \}$.
First, we define
\[
	\begin{gathered}
		\mathbf{u}_1 = (1, 0, 1), \quad \mathbf{u}_2 = (5, 3, 7), \quad \mathbf{u}_3 = (8, 7, 13), \\
		\mathbf{u}_4 = (7, 8, 13), \quad \mathbf{u}_5 = (3, 5, 7), \quad \mathbf{u}_6 = (0, 1, 1)
	\end{gathered}
\]
and $\mathbf{y}_j = M_{d_1} \cdots M_{d_k}\mathbf{u}_j$.
Let $Y_j$ be the point represented by $\mathbf{y}_j$ for $j = 1, \dots, 6$.
As a result, $Y_1$ and $Y_6$ are the boundary points of the cylinder set $C(d_1, \dots, d_k)$. Further, $Y_d$ and $Y_{d+1}$ are the boundary points of the cylinder set $C(d_1, \dots, d_k, d)$ for $d = 1, \dots, 5$.
Also, Proposition~\ref{PropMonotonicity} shows that
\[
	Y_1 \preceq \cdots \preceq Y_6
	\quad
	\text{or}
	\quad
	Y_6 \preceq \cdots \preceq Y_1,
\]
depending on $\sign(d_1, \dots, d_k) =1$ or $-1$.
Figure~\ref{FigY1Y6} shows the former case. If the latter holds, we can relabel $Y_j$ and $C(d_1, \dots, d_k, d)$ in reverse order.
\end{notation}
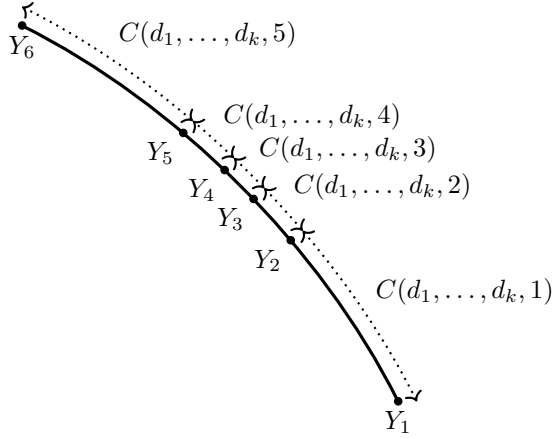
\begin{figure}
\centering
\begin{tikzpicture}[scale=2.5]
      
        \tikzset{
    partial ellipse/.style args={#1:#2:#3}{
        insert path={+ (#1:#3) arc (#1:#2:#3)}
    }
}
                                                     
\draw[very thick, rotate=-45] (0,0) [partial ellipse=60:120:2.828cm and 1.632cm] node[below]{$Y_6$};
        \draw[fill] (6/7, 10/7) circle (0.02) node[below left] {$Y_5$};
        \draw[fill] (14/13, 16/13) circle (0.02) node[below left] {$Y_4$};
        \draw[fill] (16/13, 14/13) circle (0.02) node[below left] {$Y_3$};
	\draw[fill] (10/7, 6/7) circle (0.02) node[below left]{$Y_2$};
        \draw[fill] (0, 2) circle (0.02);
	\draw[fill] (2, 0) circle (0.02) node[below]{$Y_1$};
        
        \draw[dotted, thick, <->, rotate=-45] (0, 0) [partial ellipse=60:82:2.9694cm and 1.7136cm]
	node[midway, above right] {$\footnotesize{C(d_1, \dots, d_k, 1)}$};
        \draw[dotted, thick, <->, rotate=-45] (0, 0) [partial ellipse=82:87.9:2.9694cm and 1.7136cm]
	node[midway, above right] {$\footnotesize{C(d_1, \dots, d_k, 2)}$};
        \draw[dotted, thick, <->, rotate=-45] (0, 0) [partial ellipse=87.9:92.3:2.9694cm and 1.7136cm]
	node[midway, above right] {$\footnotesize{C(d_1, \dots, d_k, 3)}$};
        \draw[dotted, thick, <->, rotate=-45] (0, 0) [partial ellipse=92.3:98:2.9694cm and 1.7136cm]
	node[midway, above right] {$\footnotesize{C(d_1, \dots, d_k, 4)}$};
        \draw[dotted, thick, <->, rotate=-45] (0, 0) [partial ellipse=98:120:2.9694cm and 1.7136cm]
	node[midway, above right] {$\footnotesize{C(d_1, \dots, d_k, 5)}$};
        
\end{tikzpicture}
\caption{The cylinder set $C(d_1,\dots, d_k)$ and rational points $Y_1, \dots, Y_6$, which are boundary points of 
	$C(d_1, \dots, d_k,1)$, \dots, 
	$C(d_1, \dots, d_k,5)$.
	This shows the case $Y_1\preceq \cdots \preceq Y_6$.
}
\label{FigY1Y6}
\end{figure}

\begin{proposition}\label{PropInteriorHeight2}
	With Notation~\ref{NotYj}, we have
\[
\max\{\Ht(Y_2), \Ht(Y_5)\} \le \min\{\Ht(Y_3), \Ht(Y_4)\}.
\]
\end{proposition}
\begin{proof}
	First, we prove $\Ht(Y_2) \le \Ht(Y_3)$. 
	Using Lemma~\ref{Lem324} and the orthogonality of $H$ and $U_d$ with respect to the bilinear pairing $\langle \cdot\, , \cdot\, \rangle$, we obtain
\begin{align*}
    \Ht(Y_3) - \Ht(Y_2) &=
    -\langle 
    H U_{d_1} \cdots H U_{d_k} (\mathbf{u}_3 - \mathbf{u}_2), \mathbf{v}_Q
\rangle \\
&=
    -\langle 
    \mathbf{u}_3 - \mathbf{u}_2, 
    U_{\hat{d}_k} H \cdots H U_{\hat{d}_1} \cdot H\mathbf{v}_Q
\rangle \\
&=
    -\langle 
    H(\mathbf{u}_3 - \mathbf{u}_2), 
    M_{\hat{d}_{k}}\cdots M_{\hat{d}_1} \cdot  H\mathbf{v}_Q
\rangle.
\end{align*}
An easy calculation shows that
\begin{equation}\label{EqU3minusU2}
	H(\mathbf{u}_3 - \mathbf{u}_2) = -(0, 1, 0)^T.
\end{equation}
Notice that $H\mathbf{v}_Q = (4, 4, 7)^T$ represents a point in the set $\Omega$ in \S\ref{SubsecPrelim}.
Therefore, if we let 
\[
	\mathbf{y} = (y_1,y_2, y_3) = M_{\hat{d}_{k}}\cdots M_{\hat{d}_1} \cdot  H\mathbf{v}_Q
= M_{\hat{d}_{k}}\cdots M_{\hat{d}_1}(4, 4, 7)^T,
\]
then we can apply Lemma~\ref{LemPreservationOfShade} repeatedly and see that $\mathbf{y}$ represents a point in $\Omega$ as well.
In particular, $y_1, y_2\ge 0$ from the property ($\Omega$-IV). 
So, 
\[
    \Ht(Y_3) - \Ht(Y_2) = -\langle H(\mathbf{u}_3 - \mathbf{u}_2), \mathbf{y}
\rangle = 
    \langle (0, 1, 0)^T, \mathbf{y} \rangle
=
    \frac{y_1}2 + y_2 \ge 0.
\]
Similarly, in order to prove $\Ht(Y_5) \ge \Ht(Y_4)$, we use
\begin{equation}\label{EqU4minusU5}
	H(\mathbf{u}_4 - \mathbf{u}_5) = -(1, 0, 0)^T
\end{equation}
in place of \eqref{EqU3minusU2} and obtain
\[
    \Ht(Y_4) - \Ht(Y_5) = 
    \langle (1, 0, 0)^T, \mathbf{y} \rangle
    = y_1 + \frac{y_2}2\ge 0.
\]

Next, we prove that $\Ht(Y_4) \ge \Ht(Y_2)$ and $\Ht(Y_3) \ge \Ht(Y_5)$.
To do the former, we compute
\[
    \Ht(Y_4) - \Ht(Y_2) = 
    \langle (-1, 2, 0)^T, \mathbf{y} \rangle
    =  \frac{3y_2}2\ge 0.
\]
For the latter, 
\[
    \Ht(Y_3) - \Ht(Y_5) = 
    \langle (2, -1, 0)^T, \mathbf{y} \rangle
    =  \frac{3y_1}2\ge 0.
\]
The proof of the proposition is now complete. 
\end{proof}
\begin{proposition}\label{PropInteriorHeight3}
	With Notation~\ref{NotYj}, we have
\[
3\Ht(Y_1) \le \Ht(Y_2)
\quad
	\text{and}
\quad
3\Ht(Y_6) \le \Ht(Y_5).
\]
\end{proposition}
\begin{proof}
	We prove the two inequalities in a similar way as in the proof of Proposition~\ref{PropInteriorHeight2}.
	First, note that
	\[
		H (\mathbf{u}_2 - 3\mathbf{u}_1 ) = 
	-\begin{pmatrix} 1 \\ 2 \\ 2 \\
	\end{pmatrix}
	=
	-\frac23\begin{pmatrix} 2 \\ 2 \\ 3 \\
	\end{pmatrix}
	-
	\frac13 \begin{pmatrix} -1 \\ 2 \\ 0 \\
	\end{pmatrix}.
	\]
	Also, $\mathbf{y}$ represents a point in $\Omega$ and therefore 
	$\langle (2, 2, 3 )^T, \mathbf{y} \rangle \ge 0$ from the property ($\Omega$-III) in \S\ref{SubsecPrelim}. So,
	\begin{align*}
3\Ht(Y_1) - \Ht(Y_2)
&=
    \frac23\langle (2, 2, 3 )^T, \mathbf{y} \rangle
    +
    \frac13\langle (-1, 2, 0)^T, \mathbf{y} \rangle \\
    &=  
    \frac23\langle (2, 2, 3 )^T, \mathbf{y} \rangle
    +
    \frac{y_1}2\ge 0.
    \end{align*}
    Likewise, the other inequality is proven in the same way. We omit the detail.
\end{proof}
\begin{theorem}\label{PropInteriorHeight}
	Let $Z^{(1, 0)}$ and $Z^{(0, 1)}$ be the boundary points of a cylinder set $C(d_1, \dots, d_k)$ (see \eqref{EqDefZk}) and let $Z$ be a rational point in the interior of the same cylinder set $C(d_1, \dots, d_k)$.
Then we have 
\[
	\max\{ \Ht(Z^{(1,0)}) , \Ht(Z^{(0, 1)}) \} \le \Ht(Z).
\]
\end{theorem}
\begin{proof}
	First, we prove that $\Ht(Y_1) \le \Ht(Y_5)$ following the same strategy.
	That is,
\[
    \Ht(Y_5) - \Ht(Y_1) = 
    \langle (-1, 2, 0)^T, \mathbf{y} \rangle
    = \frac{3y_2}2\ge 0.
\]
Since Propositions~\ref{PropInteriorHeight2} and \ref{PropInteriorHeight3} together imply 
	\[
		\Ht(Y_1) \le 
		\min\{ \Ht(Y_2), \Ht(Y_3), \Ht(Y_4) \},
	\]
	we have shown that
		\begin{equation}\label{EqBoundary1}
		\Ht(Y_1) \le 
		\min\{ \Ht(Y_2), \Ht(Y_3), \Ht(Y_4), \Ht(Y_5)\}.
		\end{equation}
Likewise, 
		\begin{equation}\label{EqBoundary2}
		\Ht(Y_6) \le 
		\min\{ \Ht(Y_2), \Ht(Y_3), \Ht(Y_4), \Ht(Y_5)\}.
		\end{equation}
	The inequalities \eqref{EqBoundary1} and \eqref{EqBoundary2} can be summarized by saying that, if $Z_1$ is a boundary point of $C(d_1, \dots, d_k)$ and $Z_2$ is a boundary point of $C(d_1, \dots, d_k, d_{k+1})$, then 
	\[
		\Ht(Z_1) \le \Ht(Z_2).
		\]
	The proof of Theorem~\ref{PropInteriorHeight} then follows by applying this argument repeatedly.
\end{proof}

\begin{proposition}
\label{prop:diam_zero}
Suppose $P\in \XXX - \ZZZ$ with $P = [d_1, d_2, \dots]_{\XXX}.$
Then the diameter of $C(d_1, \dots, d_k)$ (with respect to the metric defined by the inner product \eqref{EqDefInnerProduct2})
tends to zero as $k\to\infty$.
Also,  $Z_k^{(1, 0)}(P)$ and $Z_k^{(0, 1)}(P)$ tend to $P$ as $k\to\infty$.
\end{proposition}
\begin{proof}
The second statement in the proposition follows from the first because 
 $\dd(Z_k^{(1, 0)},P)$ and 
 $\dd(Z_k^{(0, 1)},P)$ are both bounded by the diameter 
 $
  \mathrm{diam}(
  C(d_1, \dots, d_k)
  )
  $
  of the cylinder set $C(d_1, \dots, d_k)$.
  Because $C(d_1, \dots, d_k)$ is a (closed) subarc of $\XXX$, 
  $
  \mathrm{diam}(
  C(d_1, \dots, d_k)
  )
  $
  is given by the distance between 
  $Z_k^{(1, 0)}(P)$ 
  and
  $Z_k^{(0, 1)}(P)$. 
  Using the notations introduced in Definition~\ref{DefZk}, we have
  \begin{align*}
  \mathrm{diam}(
  C(d_1, \dots, d_k)
  )^2
  &=
  \left( 
  \frac{a_k^{(1, 0)}}{c_k^{(1, 0)}}
  -
  \frac{a_k^{(0, 1)}}{c_k^{(0, 1)}}
  \right)^2 
  +
  \left( 
  \frac{b_k^{(1, 0)}}{c_k^{(1, 0)}}
  -
  \frac{b_k^{(0, 1)}}{c_k^{(0, 1)}}
  \right)^2 
  \\
  +
  &\left( 
  \frac{a_k^{(1, 0)}}{c_k^{(1, 0)}}
  -
  \frac{a_k^{(0, 1)}}{c_k^{(0, 1)}}
  \right) 
  \left( 
  \frac{b_k^{(1, 0)}}{c_k^{(1, 0)}}
  -
  \frac{b_k^{(0, 1)}}{c_k^{(0, 1)}}
  \right)
  \\
  &=
  \frac{
  -2
  \langle 
  \mathbf{z}_k^{(1, 0)}, \mathbf{z}_k^{(0, 1)}\rangle
  }{
    c_k^{(1, 0)}c_k^{(0, 1)}
  }
  =
  \frac{
  -2
  \langle 
  \mathbf{u}^{(1, 0)}, \mathbf{u}^{(0, 1)}\rangle
  }{
    c_k^{(1, 0)}c_k^{(0, 1)}
  }
  =
  \frac{
  2
  }{
    c_k^{(1, 0)}c_k^{(0, 1)}
  }.
  \end{align*}
On the other hand, utilizing Notation~\ref{NotYj}, we have (cf.~Figure~\ref{FigY1Y6})
\[
c_k^{(1, 0)}c_k^{(0, 1)} = 
\Ht(Z_k^{(1, 0)}(P))
\Ht(Z_k^{(0, 1)}(P))
=\Ht(Y_1)\Ht(Y_6)
\] 
and
\[
  c_{k+1}^{(1, 0)}c_{k+1}^{(0, 1)} = 
  \begin{cases}
    \Ht(Y_1)\Ht(Y_2) & \text{ if } d_{k+1} = 1, \\
    \Ht(Y_2)\Ht(Y_3) & \text{ if } d_{k+1} = 2, \\
    \Ht(Y_3)\Ht(Y_4) & \text{ if } d_{k+1} = 3, \\
    \Ht(Y_4)\Ht(Y_5) & \text{ if } d_{k+1} = 4, \\
    \Ht(Y_5)\Ht(Y_6) & \text{ if } d_{k+1} = 5. \\
  \end{cases}
\]
From Theorem~\ref{PropInteriorHeight}, we have
\[
\max\{ \Ht(Y_2), \Ht(Y_3), \Ht(Y_4), \Ht(Y_5) \}
<
\min\{ \Ht(Y_1), \Ht(Y_6) \},
\]
so we conclude that 
$c_k^{(1, 0)}c_k^{(0, 1)} < c_{k+1}^{(1, 0)}c_{k+1}^{(0, 1)}$.
This proves that $\mathrm{diam}(C(d_1, \dots, d_k))$ tends to zero as $k \to \infty$.
\end{proof}

\begin{theorem}\label{ThmBestApproximants}
Let $P \in \XXX - \ZZZ$ and $Z\in \ZZZ$.
Then there exists a $k\ge1$ such that
\begin{equation*}
	\min\{\delta(P;Z_k^{(1,0)}(P)),   \delta(P;Z_k^{(0,1)}(P))\} \le \delta(P;Z).  
\end{equation*}
(See Definition~\ref{DefZk}.)
\end{theorem}
\begin{remark}
Write $P = [d_1, \dots ]_{\XXX}$ and let $k$ be the greatest integer such that $Z\in C(d_1, \dots, d_k)$ but $Z\not\in C(d_1, \dots, d_{k+1})$.
Then our proof will reveal that $\delta(P; Z)$ is greater than or equal to one of the following four values:
\[
	\delta(P;Z_k^{(1,0)}(P)),
	\quad
	\delta(P;Z_k^{(0,1)}(P)),
	\quad
	\delta(P;Z_{k+1}^{(1,0)}(P)),
	\quad
	\delta(P;Z_{k+1}^{(0,1)}(P)).
\]
\end{remark}
The rest of this subsection is devoted to giving a proof of Theorem~\ref{ThmBestApproximants}, which will consist of several propositions and lemmas.

\begin{proposition}\label{LemPairingIntegral}
Let $\mathbf{z}_1$ and $\mathbf{z}_2$ be distinct primitive integral vectors representing $Z_1, Z_2\in \ZZZ$, respectively.
Then
\[
\langle
\mathbf{z}_1, \mathbf{z}_2
\rangle
\le -\frac12
\]
with the equality holding if $Z_1$ and $Z_2$ are the boundary points of a common cylinder set.
Also, using Notation~\ref{NotYj}, we have
\[
\langle
	\mathbf{y}_j, \mathbf{y}_{j+2}
\rangle
= -\frac32
\]
for $j = 1, \dots, 4$.
\end{proposition}
\begin{proof}
From the definition \eqref{EqDefPairing} of the pairing, $\langle \mathbf{z}_1, \mathbf{z}_2 \rangle$ is always a half integer whenever $\mathbf{z}_1$ and $\mathbf{z}_2$ are integral vectors.
If we combine this with Lemma~\ref{LemNegativeOnCQ}, we obtain the desired inequality.
Now, if $Z_1$ and $Z_2$ are the boundary points of a common cylinder, 
	then we may assume $\mathbf{z}_1 = \mathbf{z}^{(1, 0)}$ and $\mathbf{z}_2 = \mathbf{z}^{(0, 1)}$ in \eqref{EqDefZ10andZ01}, so that
\[
\langle
\mathbf{z}_1, \mathbf{z}_2
\rangle
=
\langle
	M_{d_1} \cdots M_{d_k}\mathbf{u}^{(1,0)}, M_{d_1} \cdots M_{d_k}\mathbf{u}^{(0,1)}
\rangle
=
\langle
	\mathbf{u}^{(1,0)}, \mathbf{u}^{(0, 1)}
\rangle
=
-\frac12.
\]
The last assertion is proven similarly, with the easy observations 
$
\langle
	\mathbf{u}_j, \mathbf{u}_{j+2}
\rangle
= -3/2$ 
	for $j = 1, \dots, 4$.
\end{proof}

\begin{lemma}\label{LemMonotonePairing}
	Let $P_1, P_2, P_3\in\XXX_0$ and let $\mathbf{p}_1, \mathbf{p}_2,\mathbf{p}_3$ be normalized vectors representing them, that is, $\mathbf{p}_j = (P_j, 1)$ for $j = 1,2,3$.
	If $P_1 \preceq P_2 \preceq  P_3$ then 
	\[
		\langle \mathbf{p}_1, \mathbf{p}_3 \rangle \le
		\langle \mathbf{p}_1, \mathbf{p}_2 \rangle.
	\]
	The same holds if the condition 
	$P_1 \preceq P_2 \preceq  P_3$ 
	is replaced by
	$P_3 \preceq P_2 \preceq  P_1$.
\end{lemma}
\begin{proof}
	From 
Proposition~\ref{PropBilinearPairingInnerProductRelation}
	and \eqref{EqDefAngleP1P2}, we have
	\[
		\langle \mathbf{p}_1, \mathbf{p}_3 \rangle  = -2\sin^2\left(
		\frac{\theta(P_3) - \theta(P_1)}2
		\right)
	\]
		and
	\[
		\langle \mathbf{p}_1, \mathbf{p}_2 \rangle  = -2\sin^2\left(
		\frac{\theta(P_2) - \theta(P_1)}2
		\right).
	\]
	The conclusion follows immediately from this because the condition in the lemma implies
	\[\theta(P_2) - \theta(P_1) \le \theta(P_3) - \theta(P_1).\]
	When $P_1$ and $P_3$ are exchanged, the conclusion remains unchanged because the bilinear pairing is symmetic.
\end{proof}

In what follows, whenever $Z, Z_1, Z_2, \dots$ denote elements in $\ZZZ$, the corresponding small bold-faced letters $\mathbf{z}, \mathbf{z}_1, \mathbf{z}_2, \dots$ will mean the primitive integral vectors representing $Z, Z_1, Z_2,$ etc,.

\begin{lemma}[Type I]\label{LemTypeI}
	Let $P\in \XXX$ and $Z, Z_1\in \ZZZ$.
	Assume that the following conditions hold:
\begin{enumerate}[font=\upshape, label=(\Alph*)]
	\item $\Ht(Z_1) \le \Ht(Z)$
	\item Either $Z \preceq Z_1 \preceq P$ or $P \preceq Z_1 \preceq Z$.
\end{enumerate}
	Then $\delta(P; Z_1) \le \delta(P; Z)$.
\end{lemma}
\begin{proof}
	Write $c_1 = \Ht(Z_1)$ and $c(Z) = \Ht(Z)$. 
	Also, let $\mathbf{p}$ be the normalized vector representing $P$.
	With these notations, we can use Definition~\ref{DefDeltaBasic} to express $\delta(P; Z_1)$ and $\delta(P; Z)$ using the bilinear pairing.
	That is, 
	\begin{equation}\label{EqType1Inequality}
		\frac{\delta^2(P; Z_1)}{\delta^2(P; Z)} =
		\frac{-2c_1\langle \mathbf{p}, \mathbf{z}_1 \rangle}{%
		-2c(Z)\langle \mathbf{p}, \mathbf{z} \rangle}
		=
		\frac{c_1\langle \mathbf{p}, \mathbf{z}_1 \rangle}{%
		c(Z)\langle \mathbf{p}, \mathbf{z} \rangle}.
	\end{equation}
	On the other hand, we can apply Lemma~\ref{LemMonotonePairing} with the condition (B) to obtain 
	\begin{equation}\label{EqType1Monotone}
		\frac{\langle \mathbf{p}, \mathbf{z} \rangle}{c(Z)}
		\le
		\frac{\langle \mathbf{p}, \mathbf{z}_1 \rangle}{c_1}.
	\end{equation}
	Now combine \eqref{EqType1Monotone} with \eqref{EqType1Inequality} to obtain the following:
	(Warning: Notice from Lemma~\ref{LemNegativeOnCQ} that the values of the above pairings are all negative and one needs care handling the directions of these inequalities.)
	\[
		\frac{\delta^2(P; Z_1)}{\delta^2(P; Z)}  \le
		\frac{c_1^2}{c(Z)^2}
		\le 1
	\]
	where the last inequality is from (A).
	This proves the lemma.
\end{proof}

\begin{lemma}[Type II]\label{LemTypeII}
	Let $P\in \XXX$ and $Z, Z_1,Z_2\in \ZZZ$.
	Assume that the following conditions hold:
\begin{enumerate}[font=\upshape, label=(\Alph*)]
	\item $\Ht(Z_1) \le \Ht(Z)$
	\item Either $Z_1 \preceq P \preceq Z_2 \preceq Z$ or 
	$Z \preceq Z_2 \preceq P \preceq Z_1$,
	\item $\langle \mathbf{z}, \mathbf{z}_2 \rangle \le
		\langle \mathbf{z}_1, \mathbf{z}_2 \rangle.$
\end{enumerate}
	Then $\delta(P; Z_1) \le \delta(P; Z)$.
\end{lemma}
\begin{proof}
	Again we write $c_1 = \Ht(Z_1)$, $c_2=\Ht(Z_2)$, and $c(Z) = \Ht(Z)$ and proceed similarly. That is,
	\begin{equation}\label{EqType2Inequality}
		\frac{\delta^2(P; Z_1)}{\delta^2(P; Z)} =
		\frac{c_1\langle \mathbf{p}, \mathbf{z}_1 \rangle}{%
		c(Z)\langle \mathbf{p}, \mathbf{z} \rangle}
		\le
		\frac{\langle \mathbf{p}, \mathbf{z}_1 \rangle}{%
		\langle \mathbf{p}, \mathbf{z} \rangle},
	\end{equation}
	where the last inequality is from the condition (A) in the statement of the lemma.
	This time, Lemma~\ref{LemMonotonePairing}, when applied with the condition (B), gives
	\begin{equation}\label{EqType2Monotone}
		\frac{\langle \mathbf{z}_1, \mathbf{z}_2 \rangle}{c_1c_2}
		\le
		\frac{\langle \mathbf{p}, \mathbf{z}_1 \rangle}{c_1}
		\quad
		\text{and}
		\quad
		\frac{\langle \mathbf{p}, \mathbf{z} \rangle}{c(Z)}
		\le
		\frac{\langle \mathbf{z}, \mathbf{z}_2 \rangle}{c(Z)c_2}.
	\end{equation}
	Therefore,
	\[
		\frac{\delta^2(P; Z_1)}{\delta^2(P; Z)}  \le
		\frac{\langle \mathbf{p}, \mathbf{z}_1 \rangle}{%
		\langle \mathbf{p}, \mathbf{z} \rangle}
		\le
		\frac{\langle \mathbf{z}_1, \mathbf{z}_2 \rangle}{c_2}
		\frac{c_2}{\langle \mathbf{z}, \mathbf{z}_2 \rangle}
		=
		\frac{\langle \mathbf{z}_1, \mathbf{z}_2 \rangle}{%
		\langle \mathbf{z}, \mathbf{z}_2 \rangle}
		\le 1,
	\]
	where the last inequality is from (C).
	This proves the lemma.
\end{proof}
\begin{proposition}\label{PropBestApproxRational}
	We continue to use Notation~\ref{NotYj}.
	Suppose that $P$ is in the interior of a cylinder set $C(d_1,\dots, d_k, d_{k+1})$.
\begin{enumerate}[font=\upshape, label=(\alph*)]
	\item If $d_{k+1} = 1$, then
		$
	\delta(P; Y_1) \le \min\{\delta(P; Y_3), \delta(P; Y_4), \delta(P; Y_5) \}.
	$
\item If $d_{k+1} = 2$, then
$
	\delta(P; Y_2) \le \delta(P; Y_4) 
		$
		and
		$
	\delta(P; Y_1) \le \delta(P; Y_5).
$
\item If $d_{k+1} = 3$, then
$
	\delta(P; Y_1) \le \delta(P; Y_2)
$
and
$
	\delta(P; Y_6) \le \delta(P; Y_5).
$
\item If $d_{k+1} = 4$, then
$
	\delta(P; Y_5) \le \delta(P; Y_3) 
		$
		and
		$
	\delta(P; Y_6) \le \delta(P; Y_2).
$
	\item If $d_{k+1} = 5$, then
		$
	\delta(P; Y_6) \le \min\{\delta(P; Y_2), \delta(P; Y_3), \delta(P; Y_4) \}.
	$
\end{enumerate}
\end{proposition}
\begin{proof}
	For (a), 
	we will prove $\delta(P; Y_1) \le \delta(P; Y_3)$ first.
	Let $Z_1 = Y_1$, $Z_2 = Y_2$, and $Z=Y_3$.
	We will apply Lemma~\ref{LemTypeII} (Type II) to obtain the desired inequality.
	To do so, we verify the conditions (A)---(C) in Lemma~\ref{LemTypeII}.
The point $Z_1$ is a boundary point of $C(d_1, \dots, d_k)$ and $Z$ is in its interior.
	So Proposition~\ref{PropInteriorHeight} gives (A).
	For (B), it is clear by definition that either $Z_1 \preceq P \preceq Z_2 \preceq Z$ or $Z \preceq Z_2 \preceq P \preceq Z_1$.
In addition, since $Z$ is in the interior of 
$C(d_1, \dots, d_k)$, 
we can apply Proposition~\ref{PropInteriorHeight} to $C(d_1, \dots, d_k)$ and conclude $\Ht(Z_1) \le \Ht(Z)$.
Finally, Proposition~\ref{LemPairingIntegral} gives
$\langle \mathbf{z}_1 ,\mathbf{z}_2\rangle = -1/2$, 
which is the maximum value of the pairing with integral vectors.
So the condition (C) in Lemma~\ref{LemTypeII} is satisfied. 
As a result, we obtain from Lemma~\ref{LemTypeII} that $\delta(P; Y_1) \le \delta(P; Y_3)$.
The same argument with $Z = Y_4$ and $Z=Y_5$ proves $\delta(P; Y_1) \le \delta(P; Y_4)$ and
$\delta(P; Y_1) \le \delta(P; Y_5)$.

	The proof of (b) is similar.  
	To prove $\delta(P; Y_2) \le \delta(P; Y_4)$, we let
	$Z_1 = Y_2$, $Z_2 = Y_3$, and $Z=Y_4$.
	In this case, the condition (A) of Lemma~\ref{LemTypeII} comes from Proposition~\ref{PropInteriorHeight2}. The other conditions are verified similarly as before. 
	For $\delta(P; Y_1) \le \delta(P; Y_5)$, we apply Lemma~\ref{LemTypeII} again with
	$Z_1 = Y_1$, $Z_2 = Y_3$, and $Z=Y_5$.
	Proposition~\ref{LemPairingIntegral} says $
	\langle \mathbf{y}_1 , \mathbf{y}_3 \rangle
	=
	\langle \mathbf{y}_3 , \mathbf{y}_5 \rangle
	=-3/2
	$, which gives (C).

	It remains to prove (c) in the proposition.
	Let $c_1 =\Ht(Y_1)$ and $c_2 = \Ht(Y_2)$. 
	First, we apply Proposition~\ref{PropInteriorHeight3} to get
	\begin{equation}\label{C1}
		\frac{\delta(P; Y_1)}{\delta(P; Y_2)} =
		\frac{c_1 \langle \mathbf{p} , \mathbf{y}_1 \rangle}{%
		c_2 \langle \mathbf{p} , \mathbf{y}_2 \rangle}
		\le
		\frac13
		\frac{\langle \mathbf{p} , \mathbf{y}_1 \rangle}{%
		 \langle \mathbf{p} , \mathbf{y}_2 \rangle}.
	\end{equation}
	Next, let $P' = \TTT^k(P)$ and $\mathbf{p}' = (P', 1)$ and apply Proposition~\ref{PropInvariance}:
	\begin{equation}\label{C2}
		\frac{\langle \mathbf{p} , \mathbf{y}_1 \rangle}{%
		\langle \mathbf{p} , \mathbf{y}_2 \rangle}
		=
		\frac{\langle \mathbf{p}' , \mathbf{u}_1 \rangle}{%
		\langle \mathbf{p}' , \mathbf{u}_2 \rangle}.
	\end{equation}
	To simplify this expression further, we note that $d(P') = d_{k+1} = 3$ and
	\[
	\mathbf{u}_1' := 
	M_3^{-1} 
	\mathbf{u}_1 = 
		\begin{pmatrix} 0 \\ -1 \\ 1 \end{pmatrix}
			\quad
			\text{and}
	\quad
	\mathbf{u}_2' := 
	M_3^{-1} \mathbf{u}_2 = 
		\begin{pmatrix} 1 \\ -1 \\ 1 \end{pmatrix}.
	\]
Let $P'' = \TTT(P')$ and $\mathbf{p}'' = (P'', 1)$, and apply Proposition~\ref{PropInvariance} once again to obtain
	\begin{equation}\label{C4}
		\frac{\langle \mathbf{p}' , \mathbf{u}_1 \rangle}{%
		\langle \mathbf{p}' , \mathbf{u}_2 \rangle} 
		=
		\frac{\langle \mathbf{p}'' , \mathbf{u}_1' \rangle}{%
			\langle \mathbf{p}'' , \mathbf{u}_2' \rangle} 
		=
		\frac{\sin^2\left(\frac{\theta(P'')-\theta(U_1')}2\right)}{
			\sin^2\left(\frac{\theta(P'')-\theta(U_2')}2\right)},
	\end{equation}
	where $U_1' = (0, -1)$ and $U_2' = (1, -1)$.
	Clearly, $\theta(U_1') = 4\pi/3$ and $\theta(U_2') = 5\pi/3$ (see Figure~\ref{FigLocation}).
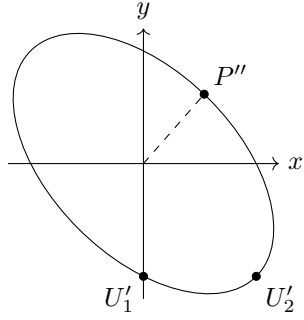
\begin{figure}
\centering
\begin{tikzpicture}[scale=1.5]
        %
        %
        %
        %
      
        \tikzset{
    partial ellipse/.style args={#1:#2:#3}{
        insert path={+ (#1:#3) arc (#1:#2:#3)}
    }
}
        \draw[->] (-1.2, 0) -- (1.2, 0) node[right]{$x$};
        \draw[->] (0, -1.2) -- (0, 1.2) node[above]{$y$};

        \draw[rotate=-45] (0,0) ellipse (1.414cm and 0.816cm); 
        \draw[dashed] (0, 0) -- (7/13, 8/13) node[above right]{$P''$};
        \filldraw (7/13, 8/13) circle (1.0pt);
        \filldraw (0, -1) circle (1.0pt) node[below left]{$U_1'$};
        \filldraw (1, -1) circle (1.0pt) node[below right]{$U_2'$};

    
    
\end{tikzpicture}
    \caption{Locations of $U_1', U_2'$ and $P''$ on $\XXX_0$.}
    \label{FigLocation}
\end{figure}
	Some elementary calculus shows that
		\[
		\frac{\sin^2\left(\frac{\theta(P'')}2 - \frac{4\pi}3\right)}{
			\sin^2\left(\frac{\theta(P'')-\frac{5\pi}3}2\right)}
		=
			\left(
		\frac{%
		\frac{\sqrt3}2\cot\frac{\theta(P'')}2 + \frac12
		}{%
		\frac12\cot\frac{\theta(P'')}2 + \frac{\sqrt{3}}2
		}
			\right)^2 \le 3
		\]
	when $0\le \theta(P'') \le \pi/3$.
	Combining this with \eqref{C2} and \eqref{C4}, we obtain
	\begin{equation}\label{C5}
		\frac{\langle \mathbf{p} , \mathbf{y}_1 \rangle}{%
		\langle \mathbf{p} , \mathbf{y}_2 \rangle}
		\le 3.
	\end{equation}
	Then we conclude from \eqref{C1} and \eqref{C5} that
		\[
		\frac{\delta(P; Y_1)}{\delta(P; Y_2)} =
		\frac{c_1 \langle \mathbf{p} , \mathbf{y}_1 \rangle}{%
		c_2 \langle \mathbf{p} , \mathbf{y}_2 \rangle}
		\le
		1,
		\]
		which gives (c) of the proposition.
		The second inequality of (c) is proven by the same method. 

		The cases (d) and (e) immediately follow from (b) and (a) by symmetry.
		This completes the proof of the proposition.
\end{proof}
\begin{proof}[Proof of Theorem~\ref{ThmBestApproximants}]
As before, we write $P = [d_1, \dots, d_k, \dots]_{\XXX}$.
If $Z$ is one of the boundary points of the cylinder sets $\{C(d_1,\dots, d_k)\}_{k=1}^{\infty}$ containing $P$, then the conclusion of Theorem~\ref{ThmBestApproximants} is obviously true and there is nothing to prove.
So we will assume that $Z$ is not equal to any boundary point of $C(d_1, \dots, d_k)$ for every $k\ge1$.

Fix $k$ to be the largest integer such that $Z$ is in $C(d_1,\dots, d_k)$
but not in $C(d_1,\dots, d_k,d_{k+1}).$
Then the Romik digit expansion of $Z$ is given by
\[
Z= [d_1, \dots, d_k, d_Z, \dots]_{\XXX}
\]
with $d_Z\neq d_{k+1}$. 
In particular, $Z$ is not equal to any of 
the four points
$Z_k^{(1, 0)}(P)$,
$Z_k^{(0, 1)}(P)$, 
$Z_{k+1}^{(1, 0)}(P)$,
and
$Z_{k+1}^{(0, 1)}(P)$.

We first handle the case when $Z\in \{ Y_1, \dots, Y_6\}$. However, recall that $Z$ is assumed to be not equal to boundary points of any cylinder set containing $P$.
	Since $Y_1$ and $Y_6$ are the boundary points of $C(d_1, \dots, d_k)$ we have 
	\[
		Z\in \{ Y_2, Y_3, Y_4, Y_5 \} -
		\{
		Z_{k+1}^{(1, 0)}(P),
		Z_{k+1}^{(0, 1)}(P)
	\}.
	\]
	Now, we use Proposition~\ref{PropBestApproxRational} to show that the conclusion of the theorem is true.
For example, 
if $d_{k+1} = 1$, then $Y_2 = Z_{k+1}^{(0, 1)}(P)$ and therefore $Z \in\{ Y_3, Y_4, Y_5\}$ and
Proposition~\ref{PropBestApproxRational} gives
\[
\delta(P; Z_k^{(1, 0)}(P)) \le \delta(P; Z).
\]
Similarly for all remaining values of $d_{k+1}$, Proposition~\ref{PropBestApproxRational} gives
either
\[
\min\{
\delta(P; Z_k^{(1, 0)}(P)),
\delta(P; Z_k^{(0, 1)}(P))
\} 
\le \delta(P; Z),
\]
or 
\[
\min\{
\delta(P; Z_{k+1}^{(1, 0)}(P))
,
\delta(P; Z_{k+1}^{(0, 1)}(P))
\}
\le \delta(P; Z).
\]
To finish the proof of Theorem~\ref{ThmBestApproximants}, we now suppose that $Z$ is in the interior of $C(d_1, \dots, d_k,d_Z)$.
We handle the \emph{adjacent interior} case, namely, $d_Z= d_{k+1} +1$ or $d_Z= d_{k+1}-1$.
Say the former holds.
We let $Z_1$ be the common boundary of 
$C(d_1,\dots, d_k, d_{k+1}+1)$ and  
$C(d_1,\dots, d_k, d_{k+1})$.
Then we have $Z \preceq Z_1\preceq P$ or $P \preceq Z_1\preceq Z$.
We apply Proposition~\ref{PropInteriorHeight} to the cylinder set $C(d_1, \dots, d_k, d_k+1)$ to conclude that $\Ht(Z_1) \le \Ht(Z)$.
	Therefore the conditions (A) and (B) for Lemma~\ref{LemTypeI} (Type I) are satisfied and we conclude that $\delta(P; Z_1)\le \delta(P;Z)$.
This proves the statement in the conclusion of Theorem~\ref{ThmBestApproximants}.

The only remaining case is when $Z$ is in the interior of 
$C(d_1, \dots, d_k, d_Z)$,
which is not adjacent to 
$C(d_1, \dots, d_k, d_{k+1})$.
Then one of the boundary points, say, $Y$, of 
$C(d_1, \dots, d_k, d_Z)$,
satisfies either 
$Z\preceq Y\preceq P$ or
$P\preceq Y\preceq Z$,
while $Y$ itself is not a boundary of 
$C(d_1, \dots, d_k)$ or
$C(d_1, \dots, d_k, d_{k+1})$.
Then we can apply Lemma~\ref{LemTypeI} as before to obtain
$\delta(P; Y)\le \delta(P;Z)$.
However, we have already proved  that 
\[
\min\{
\delta(P; Z_k^{(1, 0)}(P)),
\delta(P; Z_k^{(0, 1)}(P))
\} 
\le \delta(P; Y)
\]
or
\[
\min\{
\delta(P; Z_{k+1}^{(1, 0)}(P))
,
\delta(P; Z_{k+1}^{(0, 1)}(P))
\}
\le \delta(P; Y)
\]
for any such $Y$.
The proof of Theorem~\ref{ThmBestApproximants} is now completed.
\end{proof}

\section{Perron's Formula}\label{SecPerron}
The goal of this section is to state and prove a version of \emph{Perron's formula} (Theorem~\ref{ThmPerronFormula}).
\subsection{Preliminary definitions and the statement of Perron's formula}
Define a \emph{stereographic projection} $\| \cdot \|: \XXX \longrightarrow [0, \infty]$ to be
\begin{equation}\label{EqDefNorm}
\| (\alpha, \beta) \| = \frac{\sqrt{3}(\alpha + \beta - 1)}{-2\alpha -\beta+2}.
\end{equation}
The above stereographic projection is chosen so that our version of Perron's formula (Theorem~\ref{ThmPerronFormula}) becomes similar to the classical Perron formula 
in the theory of continued fractions.
Also, see Remark~\ref{rem:stereo_motivation}.

The stereographic projection is \emph{order-reversing} in the sense that $P_1\preceq P_2$ if and only if $\| P_1 \| \ge \| P_2 \|.$
Also, we define $(\cdot)^{\vee}: \XXX \longrightarrow \XXX$ to be
\[
 (\alpha, \beta)^{\vee}  = (\beta, \alpha).
\]
The map $(\cdot)^{\vee}$ is compatible with vectors in the following sense.
If $P$ is represented by the normalization of $\mathbf{p}= (p_1, p_2, p_3)$, then $P^{\vee}$ is represented by the normalization of $\mathbf{p}^{\vee} := (p_2, p_1, p_3)$.
\begin{lemma}
\label{lem:digitwise_check}
For $d\in \{1, 2, 3, 4, 5 \}$, define
\[
d^{\vee}
=
\begin{cases}
5 & \text{ if } d = 1,\\
4 & \text{ if } d = 2,\\
3 & \text{ if } d = 3,\\
2 & \text{ if } d = 4,\\
1 & \text{ if } d = 5.\\
\end{cases}
\]
Suppose that $P \in \XXX$ with $P = [d_1, d_2, \dots]_{\XXX}$.
Then $P^{\vee} = [d_1^{\vee}, d_2^{\vee}, \dots]_{\XXX}$.
\end{lemma}
\begin{proof}
It is obvious that the first digit of $P$ is $d_1$ if and only if the first digit of $P^{\vee}$ is $d_1^{\vee}$ (see Figure~\ref{FigRomikDigits}).
To see that the remaining digits of $P^{\vee}$ are given as in the lemma, we write $P = (x, y)$ and $\mathbf{p} = (x, y, 1)$. 
And we let
\[
S = \begin{pmatrix}
  0 & 1 & 0 \\
  1 & 0 & 0 \\
  0 & 0 & 1
\end{pmatrix},
\]
so that $S\mathbf{p} = \mathbf{p}^{\vee}$.
It is easy to verify that $S M^{-1}_d = M^{-1}_{d^{\vee}}S$ for $d = 1, \dots, 5$.
Now, for any $k>1$, 
\[
M_{d_{k-1}^{\vee}}^{-1} 
\cdots
M_{d_1^{\vee}}^{-1} 
\mathbf{p}^{\vee}
=
S 
\cdot
M_{d_{k-1}}^{-1} 
\cdots
M_{d_1}^{-1} \mathbf{p}.
\]
Combining this with Proposition~\ref{PropRomiksystemBerggrenTrees}, we see that the $k$-th digit of $P^{\vee}$ is the same as $d_k^{\vee}$.
\end{proof}

Lastly, we define another map $(\cdot)^{\hat{}}:\XXX \longrightarrow \XXX$ using Romik digit expansion.
If $P = [d_1, d_2, \dots]_{\XXX} \in \XXX$, then define
\[
\hat{P} = [\hat{d}_1, \hat{d}_2, \dots]_{\XXX}
\]
where $\hat{d}$ is defined as in Lemma~\ref{Lem324}.

\begin{theorem}[Perron's formula]\label{ThmPerronFormula}
Let $P\in \XXX-\ZZZ$ with 
\[
P = [d_1, d_2, \dots, d_k, d_{k+1},\dots]_{\XXX}.
\]
We use the following notations: for each $k\ge 1$,
\begin{enumerate}[font=\upshape, label=(\roman*)]
\item  $\mathbf{u} = \mathbf{u}^{(1, 0)} = (1, 0, 1)$,
\item  $\mathbf{z}_k =  \mathbf{z}_k^{(1, 0)}(P)$
and $Z_k = Z_k^{(1, 0)}(P)$ as defined in Definition~\ref{DefZk},
\item $P_k' = \TTT^k(P) = [d_{k+1}, d_{k+2}, \dots]_{\XXX},$
\item $P_k'' = [d_k, d_{k-1}, \dots, d_2, d_1, 1^{\infty}]_{\XXX},$ and
\item $\epsilon_k(P) = \left\vert\frac{\sin(\theta(P)/2)}{\sin(\theta(Z_k)/2)}\right\vert.$
\end{enumerate}
Then we have
\[
\delta(P; Z_k) = \frac{\epsilon_k(P)}{\| (\hat{P}''_k)^{\vee} \| + \| P_k' \|}.
\]
\end{theorem}
\begin{corollary}\label{CorPerronFormula}
The notations are the same as in Theorem~\ref{ThmPerronFormula}. Then
\[
\delta(P) = 
\liminf_{k\to\infty}
\min
\left\{
\frac1
{\| (\hat{P}''_k)^{\vee} \| + \| P_k' \|},
\frac1
{\| \hat{P}''_k \| + \| (P_k')^{\vee} \|}
\right\}.
\]
\end{corollary}
\begin{proof}
First of all, note that $\epsilon_k(P)\to1$ as $k\to\infty$ because $Z_k\to P$ (cf.~Proposition~\ref{prop:diam_zero}).
Next, we see from Theorem~\ref{ThmBestApproximants} that it is sufficient to approximate $P$ by 
$Z_k^{(1, 0)}(P)$ 
and
$Z_k^{(0, 1)}(P)$.
So,
\[
\delta(P) = 
\liminf_{k\to\infty}
\min
\left\{
\delta(P; Z_k^{(1, 0)}(P)),
\delta(P; Z_k^{(0, 1)}(P))
\right\}.
\]
However, one can show that
\[
\delta(P; Z_k^{(0, 1)}(P)) = \delta(P^{\vee}; Z_k^{(1, 0)}(P^{\vee})).
\]
For instance, the proof of Proposition~2.12 in \cite{CK19} applies identically to our situation.
The corollary then follows from Theorem~\ref{ThmPerronFormula}.
\end{proof}
\subsection{Proof of Perron's formula}
\begin{proposition}\label{PropStereographicProjection}
Let $P \in \XXX$.
\begin{enumerate}[font=\upshape, label=(\alph*)]
\item $\| P \|\cdot \| P^{\vee} \| = 1.$
(When $\| P \| = \infty$ or $0$, we interpret this as $\infty\cdot 0 = 1$.)
\item We have
\[
\cot
\left(
\frac{\theta(P)}2
\right)
= 
2\| P \| 
+ \sqrt3.
\]
\end{enumerate}
\end{proposition}
\begin{proof}
The identity (a) is proven by some straightforward but tedious calculation from the definition of $\| P \|$, together with the condition that $\alpha^2 + \alpha\beta+\beta^2 = 1$.
Also, (b) comes from the definition \eqref{EqArgumentRelations} of $\theta(P)$, combined with some elementary trigonometry. We omit the details.
\end{proof}
\begin{proof}[Proof of Theorem~\ref{ThmPerronFormula}]
We use the notations that are given in the statements of Theorem~\ref{ThmPerronFormula}.
In addition, we define
\begin{equation}\label{EqDefWk}
   \mathbf{w}_k = (M_{d_1} \cdots M_{d_k})^{-1} \mathbf{u}. 
\end{equation}
Also, we denote by  
$\mathbf{p}$ and $\mathbf{p}_k'$ 
the normalized vectors representing $P$ and $P_k'$, that is,
$\mathbf{p} = (P, 1)$ and $\mathbf{p}_k' = (P_k', 1)$.
With these notations, we prove Theorem~\ref{ThmPerronFormula} in the following steps:
\[
\text{Step 1:}\quad
\delta^2(P; Z_k) = 
-2\Ht(Z_k)
\frac{
\langle \mathbf{p}, \mathbf{u} \rangle
}
{
\langle \mathbf{z}_k, \mathbf{u} \rangle
}
\frac{
\langle \mathbf{w}_k, \mathbf{u} \rangle
\langle \mathbf{p}_k', \mathbf{u} \rangle
}{
\langle \mathbf{p}_k', \mathbf{w}_k \rangle
}.
\]
To prove Step 1, 
we begin with Definition~\ref{DefDeltaBasic}, which gives
\begin{equation}\label{EqDeltaBasic2}
\delta^2(P; Z_k) = 
-2\Ht(Z_k)
\langle \mathbf{p}, \mathbf{z}_k \rangle.
\end{equation}
Then we apply Proposition~\ref{PropInvariance} (with $\mathbf{z}_1 = \mathbf{z}_k$ and $\mathbf{z}_2 = \mathbf{u}$)
to obtain
\begin{equation}\label{EqInvariance}
\langle 
\mathbf{p},
\mathbf{z}_k
\rangle 
=
\langle 
\mathbf{p},
\mathbf{u}
\rangle 
\frac{
\langle 
\mathbf{p}_k',
\mathbf{u}
\rangle 
}{
\langle 
\mathbf{p}_k',
\mathbf{w}_k
\rangle
}.
\end{equation}
On the other hand, we note 
from the orthogonality of $M_{d_1}, \dots, M_{d_k}$ that
\begin{equation}\label{EqWkEqualsZk}
\langle 
\mathbf{w}_k,
\mathbf{u}
\rangle 
=
\langle 
(M_{d_1} \cdots M_{d_k})^{-1} \mathbf{u},
\mathbf{u}
\rangle
=
\langle 
\mathbf{u},
\mathbf{z}_k 
\rangle.
\end{equation}
Now, we combine \eqref{EqDeltaBasic2}, \eqref{EqInvariance}, and \eqref{EqWkEqualsZk} to obtain Step 1. 
\[
\text{Step 2:}\quad
\delta^2(P; Z_k) =
\epsilon^2(P, Z_k)
\left(
\frac2{
\cot\left(
\frac{\theta(P_k')}2
\right)
-
\cot\left(
\frac{\theta(W_k)}2
\right)
}
\right)^2.
\]
Here, $W_k$ is the point represented by the normalization of $\mathbf{w}_k$.
In other words, if we write $\mathbf{w}_k = (w_1, w_2, w_3)$, then $W_k = (w_1/w_3, w_2/w_3)\in \XXX_0$ (see \eqref{EqUnitCircle}).
To prove Step 2, we apply 
Proposition~\ref{PropBilinearPairingInnerProductRelation}
repeatedly to the pairings in Step 1. 
That is,
\begin{equation}\label{EqTemp1}
\Ht(Z_k)
\frac{
\langle \mathbf{p}, \mathbf{u} \rangle
}
{
\langle \mathbf{z}_k, \mathbf{u} \rangle
}
=
\frac{
\langle \mathbf{p}, \mathbf{u} \rangle
}
{
\langle \mathbf{z}_k/\Ht(Z_k), \mathbf{u} \rangle
}
=\frac{
\sin^2(\theta(P)/2)
}{
\sin^2(\theta(Z_k)/2)
}
=
\epsilon^2(P, Z_k),
\end{equation}
and
\begin{equation}\label{EqTemp}
\begin{gathered}
\frac{
\langle \mathbf{w}_k, \mathbf{u} \rangle
\langle \mathbf{p}_k', \mathbf{u} \rangle
}{
\langle \mathbf{p}_k', \mathbf{w}_k \rangle
}
=-2\frac{
\sin^2\left(
\frac{\theta(W_k)}{2}
\right)
\sin^2\left(
\frac{\theta(P_k')}{2}
\right)
}{
\sin^2\left(
\frac{\theta(P_k', W_k)}{2}
\right)
} \\
=\frac{-2}
{
\left(
\cot\left(
\frac{\theta(P_k')}2
\right)
-
\cot\left(
\frac{\theta(W_k)}2
\right)
\right)^2
},
\end{gathered}
\end{equation}
where the last line is justified by some easy trigonometric identities.
Then Step 2 follows from Step 1, together with \eqref{EqTemp1} and \eqref{EqTemp}.

\[
\text{Step 3:}\quad
\delta^2(P; Z_k) =
\epsilon^2(P, Z_k)
\left(
\frac2{
\cot\left(
\frac{\theta(P_k')}2
\right)
+
\cot\left(
\frac{\theta(\hat{P}_k'')^{\vee}
}2
\right)
-2\sqrt3
}
\right)^2.
\]
To obtain the equality in Step 3 from Step 2, it suffices to show
\begin{equation}\label{EqWkPk}
\cot\left(
\frac{\theta(W_k)}2
\right)
+
\cot\left(
\frac{\theta(\hat{P}_k'')^{\vee}}2
\right)
=2\sqrt3.
\end{equation}
First, write $W_k = (w_1, w_2, w_3)$ as before and use Proposition~\ref{PropStereographicProjection} to get
\begin{equation}\label{EqWkCot}
\cot\left(
\frac{\theta(W_k)}2
\right)
=
\frac{
2\sqrt3(w_1 + w_2 - w_3)
}{
-2w_1 - w_2 + 2w_3}
+ \sqrt3
.
\end{equation}
For the other cotangent, we define 
\[
\mathbf{p}_k'' = M_{d_k}\cdots M_{d_1} \mathbf{u}.
\]
By definition, its normalization represents $P_k''=[d_k, \dots, d_1, 1^{\infty}]_{\XXX}$.
Therefore, $\hat{P}_k''$ is represented by the normalization of
\[
\hat{\mathbf{p}}_k'' :=
 M_{\hat{d}_k}\cdots M_{\hat{d}_1} \hat{\mathbf{u}},
\]
where $\hat{\mathbf{u}}$ is defined to be
\[
\hat{\mathbf{u}}
=
\begin{pmatrix}
0 \\ 1 \\ 1
\end{pmatrix}.
\]
A quick calculation using Definition~\ref{DefHandUd} shows that $\hat{\mathbf{u}}= H\mathbf{u}$.
We use Lemma~\ref{Lem324} and \eqref{EqDefWk} to see that
\[
\hat{\mathbf{p}}_k'' = 
M_{\hat{d}_k}\cdots M_{\hat{d}_1} H\mathbf{u}
=
H
U_{\hat{d}_k}H \cdots U_{\hat{d}_1}H\mathbf{u}
=
H\mathbf{w}_k.
\]
This lets us express the entries of $\hat{\mathbf{p}}_k''$ directly using $\mathbf{w}_k = (w_1, w_2, w_3)$.
Namely,
\[
\hat{\mathbf{p}}_k'' = 
H\mathbf{w}_k
=
\begin{pmatrix}
  -4 & -3 & 4 \\
  -3 & -4 & 4 \\
  -6 & -6 & 7 
\end{pmatrix}
\begin{pmatrix}
w_1 \\ w_2 \\ w_3 
\end{pmatrix}
=
\begin{pmatrix}
  -4w_1  -3w_2 + 4w_3 \\
  -3w_1  -4w_2 + 4w_3 \\
  -6w_1  -6w_2 + 7w_3 
\end{pmatrix}
.
\]
Therefore, we see that $(\hat{P}_k'')^{\vee} = (\alpha, \beta)$ with
\[
\alpha =
\frac{
  -3w_1  -4w_2 + 4w_3 
}{
  -6w_1  -6w_2 + 7w_3 
}
\quad
\text{and}
\quad
\beta =
\frac{
  -4w_1  -3w_2 + 4w_3 
}{
  -6w_1  -6w_2 + 7w_3 
}
.
\]
Apply this to Proposition~\ref{PropStereographicProjection} again and, after some simplification, we get
\[
\cot\left(
\frac{\theta(
(\hat{P}_k'')^{\vee}
)}2
\right)
=
\frac{
2\sqrt3(-w_1 - w_2 + w_3)
}{
-2w_1 - w_2 + 2w_3}
+ \sqrt3.
\]
Combine this with \eqref{EqWkCot} to obtain the proof of \eqref{EqWkPk}. 
This completes the proof of Step 3.
Finally, we apply Proposition~\ref{PropStereographicProjection} one more time to observe that
\[
\cot\left(
\frac{\theta(W_k)}2
\right)
+
\cot\left(
\frac{\theta(\hat{P}_k'')^{\vee}}2
\right)
-2\sqrt3
=
2(\| P_k' \| + \| (\hat{P}_k'')^{\vee} \| ).
\]
This concludes the proof of Theorem~\ref{ThmPerronFormula}.
\end{proof}

\section{Admissible sequences and their Lagrange numbers}\label{SecAdmissible}

\subsection{Doubly infinite Romik sequences and Lagrange numbers}\label{SubsecDoublyInfinite}
By a \emph{Romik sequence}, we mean an element of $\{1, 2, 3, 4, 5 \}^{\mathbb{N}}$.
By a slight abuse of language, we think of a Romik sequence $P$ as a point of $\XXX$ whose Romik digit expansion is $P$. 
When $P =[d_1, d_2, \dots ]$, we define 
$P^{\vee}$ and $\hat{P}$ to be 
$
P^{\vee} = [d_1^{\vee}, d_2^{\vee}, \dots]
$
and
$\hat{P}=[\hat{d}_1, \hat{d}_2, \dots]$ 
(cf.~Lemmas~\ref{Lem324} and \ref{lem:digitwise_check}).

A \emph{reversed} Romik sequence is an element of $\{1, 2, 3, 4, 5\}^{\mathbb{Z}_{<0}}$.
For a Romik sequence $P =[d_1, d_2, \dots ]$, we denote by $P^*$ the reversed Romik sequence obtained by reading the digits of $P$ backwards. Namely,
\[
P^* = [\dots, d_2, d_1].
\]

A \emph{doubly infinite Romik sequence} is an equivalence class of $\{1, 2, 3, 4, 5\}^{\mathbb{Z}}$ where two elements in $\{1, 2, 3, 4, 5\}^{\mathbb{Z}}$ are defined to be equivalent whenever they are \emph{shifts} of one another.
A \emph{section} of a doubly infinite Romik sequence $T$ is simply an element in (the equivalence class) $T$.
Informally, we think of a doubly infinite Romik sequence to be an ``unmarked'' or ``base point free'' doubly infinite sequence in $\{1, 2, 3, 4, 5\}$, while its section is obtained by fixing a ``base point''.  

Suppose that two Romik sequences $P =[p_1, p_2, \dots]$ and $Q=[q_1, q_2, \dots]$ are given.
We define $P^*\vert Q$ to be an element of $\{1, 2, 3, 4, 5\}^{\mathbb{Z}}$ given by 
\[
P^*\vert Q = [\dots, p_2, p_1, q_1, q_2, \dots].
\]
Informally, we think of $P^*\vert Q$ as a section of a doubly infinite Romik sequence with its base point fixed at $q_1$.
For a doubly infinite Romik sequence $T$, 
we define $\hat{T}$, $T^{\vee}$ and $T^*$ in an obvious way;
$\hat{T}$ and $T^{\vee}$ are the doubly infinite Romik sequences obtained by applying $\hat{(\cdot)}$ and $(\cdot)^{\vee}$ to each digit of $T$, and $T^*$ is obtained from $T$ by reading the digits backwards.
Finally, any sequence (finite or infinite) with values in $\{1, 2, 3, 4, 5\}$ is called \emph{reduced} if it contains neither 1 nor 5.

For $P\in \XXX - \ZZZ$, we define \emph{the Lagrange number} $L(P)$ of $P$ to be $L(P) = 1/\delta(P)$.
Corollary~\ref{CorPerronFormula} shows that
\begin{equation}
\label{eq:LPLimsupMax}
L(P) = \limsup_{k\to\infty} 
\max
\left\{
\| (\hat{P}''_k)^{\vee} \| + \| P_k' \|,
\| \hat{P}''_k \| + \| (P_k')^{\vee} \|
\right\}.
\end{equation}
Modeling after this, we define a Lagrange number of a doubly infinite Romik sequence.
\begin{definition}[Lagrange number of a doubly infinite Romik sequence]\label{DefLagrangeSequence}
Let $T$ be a doubly infinite Romik sequence. 
If $P^*\vert Q$ is section of $T$, we define
\[
L(P^* \vert  Q) = 
\| \hat{P}^{\vee} \| + \| Q \|.
\]
	Also, the \emph{Lagrange number of} $T$ is defined to be
\[
L(T) =
\sup_{P^*\vert Q}
\left\{
\max\left(
L(P^*\vert Q),
L((P^{\vee})^*\vert Q^{\vee})
\right)
\right\}
\]
Here, the supremum is taken over all sections $\{ P^*\vert Q \}$ of $T$.
\end{definition}

\begin{proposition}\label{PropBasicLProperty}
	Let $T$ be a doubly infinite Romik sequence.
	\begin{enumerate}[font=\upshape, label=(\alph*)]
		\item $ L(T) = L(T^{\vee}) = L(\hat{T}^*) = L((\hat{T}^*)^{\vee}).$
		\item If $T$ is reduced (that is, $T$ contains neither 1 nor 5), then $L(T) = L(T^{\vee}) = L(T^*) = L((T^*)^{\vee}).$
		\item If $T$ is reduced and $T = T^*$ then
\[
L(T) =
\sup_{P^*\vert Q}
\left\{
L(P^*\vert Q)
\right\}
\]
where the supremum is taken over all sections $\{ P^*\vert Q \}$ of $T$.
	\end{enumerate}
\end{proposition}
\begin{proof}
	Whenever $P^*\vert Q$ is a section of $T$, $(P^{\vee})^*\vert Q^{\vee}$ is a section of $T^{\vee}$. 
	It follows from this that $L(T) = L(T^{\vee})$.
Note that, if $P^*\vert Q$ is a section of $T$, then $\hat{Q}^*\vert \hat{P}$ is a section of $\hat{T}^*$.
Therefore,
\[
L(\hat{Q}^*\vert \hat{P}) = \| Q^{\vee} \| + \|\hat{P} \| =  
L((P^{\vee})^*\vert Q^{\vee}),
\]
	which shows that $L(\hat{T}^*) = L(T^{\vee})$. 
	This proves all the equalities in (a).
	If $T$ is reduced, then $\hat{T} = T$ and all the equalities in (b) follow from (a).
	For (c), assume $T$ is reduced and $T^*=T$. 
      If $P^*\vert Q$ is a section of $T$, then $Q^*\vert P$ is also a section of $T$. 
	Since $T$ is also assumed to be reduced we have
	\[
	L((P^{\vee})^*\vert Q^{\vee}) = \| P \| + \| Q^{\vee} \| = L(Q^*\vert P).
	\]
	This proves (c).
\end{proof}

Following Bombieri (\S2 in \cite{Bom07}), we say that a doubly infinite Romik sequence $T$ is \emph{extremal for} a point $P$ in $\XXX$ if $L(P) = L(T)$.
\begin{lemma}[Bombieri's trick]\label{LemBombieri}
For any $P\in \XXX-\ZZZ$, there exists an extremal doubly infinite Romik sequence $T$.
\end{lemma}
\begin{proof}
Our proof is based on Bombieri's argument in page 191 of \cite{Bom07}.
We will frequently rely on the fact that the function $\| \cdot \| : \{1, 2, 3, 4, 5 \}^{\mathbb{N}} \longrightarrow [0, \infty]$ is continuous when $\{1, 2, 3, 4, 5 \}^{\mathbb{N}}$ is equipped with the product topology.
This is an immediate corollary of Proposition~\ref{prop:diam_zero}.

Let $\mathcal{P} \in \mathcal{X} - \mathcal{Z}$.\footnote{During this proof, we use a calligraphic font $\mathcal{P}$ to denote an element in $\mathcal{X} - \mathcal{Z}$ in order to avoid the conflict with the $P$ in a section $P^*\vert Q$ of $T$.}
It is not difficult to prove that the lim sup and max in \eqref{eq:LPLimsupMax} commute, 
that is,
\begin{align*}
  L(\mathcal{P}) &= \limsup_{k\to\infty}
  \max\left\{
    \| (\hat{P}_k'')^{\vee} \| + \| P_k' \|
    ,
    \| (\hat{P}_k'')\| + \| P_k'^{\vee}  \|
  \right\} \\
  &=
  \max\left\{
  \limsup_{k\to\infty}
    \| (\hat{P}_k'')^{\vee} \| + \| P_k' \|
    ,
  \limsup_{k\to\infty}
    \| \hat{P}_k''\| + \| P_k'^{\vee}  \|
  \right\}.
\end{align*}
First, let us assume that
\begin{equation}
  L(\mathcal{P}) =
  \limsup_{k\to\infty}
  \left(
    \| (\hat{P}_k'')^{\vee} \| + \| P_k' \|
  \right)
   \ge 
  \limsup_{k\to\infty}
  \left(
  \| \hat{P}_k''\| + \| P_k'^{\vee}  \|
\right),
  \label{eq:assumption_lim_sup}
\end{equation}
and we shall find a doubly infinite sequence $T$ satisfying 
\begin{equation}
  L(\mathcal{P}) = L(T).
  \label{eq:Bombieri_trick}
\end{equation}

Choose a sequence (of indices) $k_1 < k_2 < \cdots $ such that
\[
  L(\mathcal{P}) = 
  \lim_{j\to\infty}
  \left( 
    \| (\hat{P}_{k_j}'')^{\vee} \| + \| P_{k_j}' \|
  \right).
\]
According to Tychonoff's theorem, $\{ 1, 2, 3, 4, 5 \}^{\mathbb{Z}}$ is a compact set.
Therefore the sequence 
$\{ (P_{k_j}'')^*\vert P_{k_j}' \}_{j=1}^{\infty}$ 
has a limit point, say, $P_0^*\vert Q_0$ in $\{ 1, 2, 3, 4, 5 \}^{\mathbb{Z}}$.
Let $T$ be the equivalence class of $P_0^*\vert Q_0$ and we will show that \eqref{eq:Bombieri_trick} is satisfied for this $T$.

By replacing $\{k_j\}_{j=1}^{\infty}$ with its subsequence, if necessary, we may assume without loss of generality that
$(P_{k_j}'')^*\vert P_{k_j}' \to P_0^*\vert Q_0$
with respect to the product topology in $\{ 1, 2, 3, 4, 5 \}^{\mathbb{Z}}$.
Consequently, we have $P_{k_j}'' \to P_0$ and $P_{k_j}' \to Q_0$ as elements of $\{1, 2, 3, 4, 5\}^{\mathbb{N}}$ with respect to the product topology.
In particular, $(\hat{P}_{k_j}'')^{\vee} \to \hat{P}_0^{\vee}$ because the digit-wise substitutions $(\hat{\cdot})$ and $(\cdot)^{\vee}$ are continuous.
Therefore,
\[
    \| (\hat{P}_{k_j}'')^{\vee} \| + \| P_{k_j}' \|
    \to
    \| \hat{P}_0^{\vee} \| + \| Q_0 \|
  =L(P_0^*\vert Q_0),
\]
which proves $L(\mathcal{P}) = L(P_0^*\vert Q_0)$.

Next, we prove that $L(P_0^*\vert Q_0) =L(T)= \sup_{P^*\vert Q}\max\{ L(P^*\vert Q), L(P^{\vee})^*\vert Q^{\vee} \}$ where $\{P^*\vert Q\}$ runs over all sections of $T$. 
Let $P^*\vert Q$ be a $t$-shift of $P_0^*\vert Q_0$  for some $t\in \mathbb{Z}$.
Then it is obvious that 
$P_{k_j+t}'' \to P$ and $P_{k_j+t}' \to Q$,
and therefore
\[
    \| (\hat{P}_{k_j+t}'')^{\vee} \| + \| P_{k_j+t}' \|
    \to
    \| \hat{P}^{\vee} \| + \| Q \|
    =
  L(P^*\vert Q).
\]
This shows that $L(P^*\vert Q)$ is a subsequential limit of 
$\{ 
    \| (\hat{P}_k'')^{\vee} \| + \| P_k' \|
\}$,
thus $\le L(\mathcal{P})$.
Since $L(\mathcal{P}) = L(P_0^*\vert Q_0)$,  we conclude that $L(P^*\vert Q) \le L(P_0^*\vert Q_0)$.
Additionally, we deduce from $P_{k_j+t}'' \to P$ and $(P_{k_j+t}')^{\vee} \to Q^{\vee}$ that
\[
  \| \hat{P}_{k_j+t}'' \| + \| (P_{k_j+t}')^{\vee} \|
    \to
    \| \hat{P} \| + \| Q^{\vee} \|
    =
  L((P^{\vee})^*\vert Q^{\vee}),
\]
showing that $L((P^{\vee})^*\vert Q^{\vee})$ is a subsequential limit of
$
\{ 
  \| \hat{P}_k'' \| + \| (P_k')^{\vee} \|
\}$.
From the assumption we made in \eqref{eq:assumption_lim_sup}, we conclude that $L((P^{\vee})^*\vert Q^{\vee}) \le L(\mathcal{P}) = L(P_0^*\vert Q_0)$.
This completes the proof of $L(P_0^*\vert Q_0) = L(T)$, thus of \eqref{eq:Bombieri_trick}.

Next, assume that
\begin{equation*}
  \limsup_{k\to\infty}
  \left(
    \| (\hat{P}_k'')^{\vee} \| + \| P_k' \|
  \right)
   \le 
  \limsup_{k\to\infty}
  \left(
  \| \hat{P}_k''\| + \| P_k'^{\vee}  \|
\right)
=L(\mathcal{P}).
\end{equation*}
In this case, we let $P_0^*\vert Q_0$ be a limit point of 
$\{ ({P_{k_j}''}^{\vee})^*\vert P_{k_j}'^{\vee} \}$ in $\{1, 2, 3, 4, 5 \}^{\mathbb{Z}}$
and let $T$ be the equivalent class of $P_0^*\vert Q_0$. 
The proof of $L(T) = L(\mathcal{P})$ can be done in a similar way as before, so we omit it.
\end{proof}

Thanks to this lemma, it suffices to characterize all doubly infinite Romik sequences $T$ with $L(T)< 4/\sqrt3$.

\subsection{Romik system on the real line}
Let $\XXX_{[0, \infty]} = [0, \infty]$ and define a map $\TTT_{[0, \infty]}: \XXX_{[0, \infty]}\longrightarrow \XXX_{[0, \infty]}$ to be
\[
\TTT_{[0, \infty]}(t)=
\begin{cases}
t-\sqrt3  & \text{ if } t\ge \sqrt3 ,\\
\dfrac{-t+\sqrt3}{\sqrt3t - 2}  & \text{ if } 2/\sqrt{3}\le t\le \sqrt3,\\
\dfrac{2t-\sqrt3}{-\sqrt3t + 2}  & \text{ if } \sqrt3/2\le t\le 2/\sqrt3,\\
\dfrac{-2t +\sqrt3}{\sqrt3t - 1}  & \text{ if } 1/\sqrt{3}\le t\le \sqrt3/2,\\
\dfrac{t}{-\sqrt3t + 1}  & \text{ if } 0\le t\le 1/\sqrt3.\\
\end{cases}
\]
Recall from \eqref{EqDefNorm} that we have defined a stereographic projection
$\| \cdot \|: \XXX \longrightarrow \XXX_{[0, \infty]}$.
Some straightforward but tedious calculation shows that the dynamical system $(\XXX_{[0, \infty]}, \TTT_{[0, \infty]})$ is conjugate to the Romik system $(\XXX, \TTT)$ with the conjugation map being the stereographic projection. 
In other words, the diagram 
\begin{equation}\label{EqCommutativeDiagram}
\begin{tikzcd}
\XXX \arrow[r, "\| \cdot \|"] \arrow[d, "\TTT"]
& \XXX_{[0, \infty]} \arrow[d, "\TTT_{[0, \infty]}"] \\
\XXX \arrow[r, "\| \cdot \|"]
& \XXX_{[0, \infty]}
\end{tikzcd}
\end{equation}
commutes.
Also, the Romik digit of $P$ is determined by its stereographic image:
\[
d(P) = 
\begin{cases}
1 & \text{ if } \| P \| \ge \sqrt3, \\
2 & \text{ if } 2/\sqrt3 \le \| P \| \ge \sqrt3, \\
3 & \text{ if } \sqrt3/2 \le \| P \| \ge 2/\sqrt3, \\
4 & \text{ if } 1/\sqrt3 \le \| P \| \ge \sqrt3/2, \\
5 & \text{ if } 0 \le \| P \| \ge 1/\sqrt3. \\
\end{cases}
\]
Figure~\ref{FigCylinderSetsLine} shows how the cylinder sets of $\XXX$ are mapped under the stereographic projection (cf.~Figure~\ref{FigCylinderSetsLengthOne}). 
By abuse of language, a cylinder set of $\XXX_{[0, \infty]}$ will mean the image of the corresponding cylinder set of $\XXX$.
\begin{figure}
\begin{center}
	\begin{tikzpicture}[scale=2.3]
		\draw (0, 0) node[below]{0} -- (1, 0) 
		node[midway, above]{$d = 5$}
		node[below]{$\frac1{\sqrt3}$} 
		-- (1.5, 0) 
		node[midway, above]{$d = 4$}
		node[below]{$\frac{\sqrt3}2$} 
		-- (2, 0) 
		node[midway, above]{$d = 3$}
		node[below]{$\frac2{\sqrt3}$} 
		-- (3, 0) 
		node[midway, above]{$d = 2$}
		node[below]{$\sqrt3$} 
		-- (5, 0) 
		node[midway, above]{$d = 1$}
		node[below]{$\infty$} ;
		\foreach \a in {0, 1, 1.5, 2, 3, 5}
		\filldraw (\a, 0) circle(0.5pt);
\end{tikzpicture}
\end{center}
    \caption{Images of cylinder sets in $\XXX_{[0, \infty]}$}
    \label{FigCylinderSetsLine}
\end{figure}
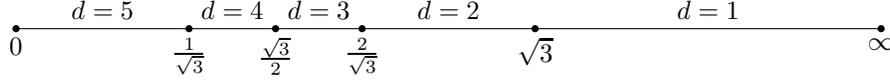

How do the actions of $M_1, \dots, M_5$ on $\XXX$ translate to $\XXX_{[0, \infty]}$?
To answer this, we let $\mathrm{GL}_2(\mathbb{R})$ act on $\mathbb{R}\cup \{ \infty\}$ via the fractional linear transformation:
\begin{equation}\label{EqDefFractionalLinearAction}
\begin{pmatrix}
a &  b \\ c & d \\
\end{pmatrix}\cdot x
=\frac{ax+ b}{cx + d}.
\end{equation}
(If $x = \infty$ then  $\begin{psmallmatrix}
a & b \\ c & d
\end{psmallmatrix}\cdot \infty = a/c$ by definition.)
And we define the matrices $N_1, \dots, N_5$ to be
\begin{equation}\label{EqNd}
\begin{gathered}
N_1 = 
\begin{pmatrix}
1 &  \sqrt{3} \\ 0 & 1 \\
\end{pmatrix},
\quad
N_2 = 
\begin{pmatrix}
2 & \sqrt{3} \\ \sqrt{3} & 1 \\
\end{pmatrix},
\quad
N_3 = 
\begin{pmatrix}
2 & \sqrt{3} \\ \sqrt{3} & 2 \\
\end{pmatrix},
\\
N_4 = 
\begin{pmatrix}
1 & \sqrt{3} \\ \sqrt{3} & 2 \\
\end{pmatrix},
\quad
N_5 = 
\begin{pmatrix}
1 & 0 \\ \sqrt{3} & 1 \\
\end{pmatrix}.
\end{gathered}
\end{equation}
Then the following proposition shows that the actions of $N_d$ on $\XXX_{[0, \infty]}$ and those of $M_d$ on $\XXX$ are the same
(cf.~Proposition~\ref{PropRomiksystemBerggrenTrees}).
\begin{proposition}\label{PropActionNd}
Let $P\in \XXX$. Then we have
\[
\| [d, P ]\| = N_d \cdot \| P \|
\]
for $d = 1, \dots, 5.$
Here, the dot between $N_d$ and $\| P \|$ indicates the $\mathrm{GL}_2(\mathbb{R})$-action of $N_d$ on $\XXX_{[0, \infty]}$ as in \eqref{EqDefFractionalLinearAction}.

\end{proposition}
\begin{proof}
For each $d = 1, \dots, 5$, write
$
P = (\alpha, \beta) = [d_1, d_2, \dots]
$
and
$
P_d' = (\alpha_d', \beta_d') = [d, d_1, d_2, \dots].
$
Then  $P = \TTT(P_d')$.
If we let $t_d = \| P_d' \|$ and 
$t = \| P \|$ then the commutativity of 
\eqref{EqCommutativeDiagram}
implies that 
\begin{equation} \label{EqPTPd}
t = \TTT_{[0,\infty]}(t_d).
\end{equation}
On the other hand, \eqref{EqNd} gives 
\begin{gather*}
N_1^{-1} = 
\begin{pmatrix}
1 &  -\sqrt{3} \\ 0 & 1 \\
\end{pmatrix},
\quad
N_2^{-1} = 
\begin{pmatrix}
-1 & \sqrt{3} \\ \sqrt{3} & -2 \\
\end{pmatrix},
\quad
N_3^{-1} = 
\begin{pmatrix}
2 & -\sqrt{3} \\ -\sqrt{3} & 2 \\
\end{pmatrix},
\\
N_4^{-1} = 
\begin{pmatrix}
-2 & \sqrt{3} \\ \sqrt{3} & -1 \\
\end{pmatrix},
\quad
N_5^{-1} = 
\begin{pmatrix}
1 & 0 \\ -\sqrt{3} & 1 \\
\end{pmatrix}.
\end{gather*}
By comparing this with \eqref{EqPTPd} (and the definition of $\TTT_{[0, \infty]}$), 
we complete the proof of the proposition.
\end{proof}

For a fixed finite sequence $w = [d_1, \dots, d_k]$ of Romik digits, we define
\[
N_w = 
\begin{pmatrix}
a_w & b_w \\ c_w & d_w\\
\end{pmatrix}
=
N_{d_1}N_{d_2}\cdots N_{d_k}.
\]
We denote by $w^*$ and $w^{\vee}$ the corresponding sequences of Romik digits obtained from $w$ by applying $*$ and $\vee$,
that is,
\[
w^* = [d_k, d_{k-1}, \dots,d_1] \text{ and }
w^{\vee} = [d_1^{\vee}, d_2^{\vee}, \dots, d_k^{\vee}].
\]
\begin{proposition}\label{PropCylinderMatrix}
Let $w = [d_1, \dots, d_k]$ and suppose $P\in C(d_1, \dots, d_k)$.
Then 
\[
\begin{cases}
b_w/d_w \le \| P \| \le a_w/c_w & \text{ if } \det(N_w) = 1, \\
a_w/c_w \le \| P \| \le b_w/d_w & \text{ if } \det(N_w) = -1. \\
\end{cases}
\]
Also,
\[
\| Z^{(1, 0)}_k(P) \| = a_w/c_w
\quad
\text{and}
\quad
\| Z^{(0, 1)}_k(P) \| = b_w/d_w.
\]
\end{proposition}
\begin{proof}
Proposition~\ref{PropActionNd} implies  
\[
\| Z_k^{(1, 0)}(P) \| = 
\| [d_1, \dots, d_k, 1^{\infty}] \| 
=
N_{d_1}\cdots N_{d_k} \cdot \| (1, 0) \| = N_w \cdot \infty = \frac{a_w}{c_w}.
\]
The equality for $Z^{(0, 1)}_k(P)$ is similarly proven.
Now, the inequalities for $P\in C(d_1, \dots, d_k)$ follow from Proposition~\ref{PropMonotonicity}.
\end{proof}

For instance, let 
$[2, d_1, d_2, \dots] \in C(2)$.
Since $N_2 = \begin{psmallmatrix}
2 & \sqrt3 \\ \sqrt3 & 1\\
\end{psmallmatrix}$ and its determinant is $-1$ we have
\[
\frac{2}{\sqrt3}
\le
\| [2, d_1, d_2, \dots] \| 
\le \frac{\sqrt3}{1}
\]
for any $P = [d_1, d_2, \dots] \in\XXX$.
To simplify notations,
we will write the above inequality as
\begin{equation}\label{EqCylinder4}
\frac{2}{\sqrt3}
\le
\| 2\cdots \| 
\le 
\sqrt3.
\end{equation}
In other words, the ellipses ($\cdots$) will be used whenever the inequality holds true regardless of the choice of the remaining Romik digits.

\begin{proposition}\label{CorCommonPrefixComparison}
Suppose that $w$ is a common prefix (possibly empty) of $P_1$ and $P_2$, that is,
\[
\begin{cases}
P_1 = [w, d_1, \dots] \\
P_2 = [w, d_2, \dots],
\end{cases}
\]
with $d_1\neq d_2$.
If $\det(N_w) = 1$, then 
\[
\| P_1 \| \le \| P_2 \| 
\Longleftrightarrow
d_1>d_2.
\]
If $\det(N_w) = -1$, then 
\[
\| P_1 \| \le \| P_2 \|
\Longleftrightarrow
d_1<d_2.
\]
\end{proposition}
\begin{proof}
This follows from Proposition~\ref{PropMonotonicity} and the fact that the stereographic projection is order-reversing.
\end{proof}

\begin{lemma}\label{LemMatrixTranspose}
For a fixed sequence $w = [d_1, \dots, d_k]$ of Romik digits, write 
\[
N_w
=
\begin{pmatrix}
a_w & b_w \\ c_w & d_w\\
\end{pmatrix}.
\]
Then we have
\[
N_{w^{\vee}} = 
\begin{pmatrix}
d_w & c_w \\ b_w & a_w\\
\end{pmatrix}.
\]
Further, if $w$ is reduced (that is, $w$ contains neither 1 nor 5), then
\[
N_{w^*} = 
\begin{pmatrix}
a_w & c_w \\ b_w & d_w\\
\end{pmatrix},
\]
and that
\[
N_{(w^*)^{\vee}} = 
\begin{pmatrix}
d_w & b_w \\ c_w & a_w\\
\end{pmatrix}.
\]
\end{lemma}
\begin{proof}
Let 
\[
J 
=
\begin{pmatrix}
0 & 1 \\ 1 & 0\\
\end{pmatrix}.
\]
Then it is easy to see $J^2 = I_2$ (the $2\times 2$ identity matrix) and that
\[
J
\begin{pmatrix}
a & b \\ c & d\\
\end{pmatrix}
J
=
\begin{pmatrix}
d & c \\ b & a\\
\end{pmatrix}
\]
for any $a, b, c, d\in \mathbb{R}$.
A quick inspection of \eqref{EqNd} reveals
\[
JN_dJ = N_{d^{\vee}}
\]
for $d = 1, \dots, 5$.
From this observation, we have
\begin{align*}
N_{w^{\vee}} &=
N_{d_1^{\vee}} N_{d_2^{\vee}}  \cdots N_{d_k^{\vee}} \\
&= (JN_{d_1}J)(J N_{d_2} J) \cdots (JN_{d_k}J) = 
JN_wJ \\
&=
\begin{pmatrix}
d_w & c_w \\ b_w & a_w\\
\end{pmatrix}.
\end{align*}
For the second claim, note that $N_2, N_3, N_4$ are symmetric and
\[
N_{w^*} = N_{d_k}\cdots N_{d_1} =
N_{d_k}^T\cdots N_{d_1}^T= (N_w)^T.
\]
The third comes from combining the previous two.
\end{proof}

\begin{remark}
\label{rem:stereo_motivation}
It is a consequence of our choice of the stereographic projection in \eqref{EqDefNorm} that $JN_dJ = N_{d^{\vee}}$ and $N_2, N_3, N_4$ are symmetric.
This is one of the reasons why we have selected this particular stereographic projection in \eqref{EqDefNorm}.
\end{remark}

Given a finite Romik sequence $w = [d_1, \dots, d_k]$, we denote by ${}^{\infty}w^{\infty}$ the doubly infinite Romik sequence
\[
{}^{\infty}w^{\infty}
=
\cdots www\cdots.
\]
Notice that, if $P_w = w^{\infty}$ then 
$(P_{w^*})^*\vert P_w$ 
is a section of ${}^{\infty}w^{\infty}$, that is,
\[
(P_{w^*})^*\vert P_w
=
\cdots w\vert w\cdots.
\]
\begin{proposition}\label{PropPnormAndLP}
Given $w = [d_1, \dots, d_k]$, let $P_w = w^{\infty}$ and write
\[
N_w
=
\begin{pmatrix}
a_w & b_w \\ c_w & d_w\\
\end{pmatrix}.
\]
Also, let $\Delta_w = \Tr(N_w)^2 - 4\det(N_w)$.
Then,
\[
\| P_w \| =
\frac{a_w-d_w+\sqrt{\Delta_w}}{2c_w}.
\]
Further, if $w= [d_1,\dots d_k]$ is reduced, then
\[
L((P_{w^*})^*\vert P_w) = \frac{\sqrt{\Delta_w}}{c_w}.
\]
\end{proposition}
\begin{proof}
Since $\mathcal{T}_{[0, \infty]}^k(P_w) = P_w$ we have from Proposition~\ref{PropActionNd} that $N_w^{-1} \cdot \| P_w \| = \| P_w\|$.
On the other hand, it is easy to see from \eqref{EqDefFractionalLinearAction} that, 
for any $N\in \mathrm{GL}_2(\mathbb{R})$, we have 
$N\cdot x = x$ whenever the (column) vector $(x, 1)$ is an eigenvector of $N$.
Now a direct calculation shows that the vector
\[
\begin{pmatrix}
\frac{a_w-d_w+\sqrt{\Delta_w}}{2c_w} \\
1
\end{pmatrix}
\]
is an eigenvector of $N_w$.
This establishes the first equation in the proposition. For the second, combine this with Lemma~\ref{LemMatrixTranspose}.
\end{proof}
\subsection{Examples}\label{SecLExamples}
\subsubsection{Constant sequences} Since $ (3^{\infty})^{\vee}=3^{\infty}$ Proposition~\ref{PropStereographicProjection} gives $\| 3^{\infty} \| =1 $. So, 
\[
L({}^{\infty}3^{\infty})
=
L({}^{\infty}3\vert 3^{\infty}) = \| P^{\vee} \| + \| P \|= 1 + 1 = 2.
\]
For ${}^{\infty} 2^{\infty}$, we use Proposition~\ref{PropPnormAndLP} with 
$N_2 = \begin{psmallmatrix} 2  & \sqrt{3}\\ \sqrt{3} & 1
\end{psmallmatrix}$
to obtain
    \[
    L({}^{\infty}2\vert 2^{\infty}) = 
    \frac{\sqrt{3^2 + 4}}{\sqrt{3}} =
    \sqrt{\frac{13}{3}}.
    \]
Likewise, using
$N_4 = \begin{psmallmatrix} 1  & \sqrt{3}\\ \sqrt{3} & 2 \end{psmallmatrix}$,
\[
    L({}^{\infty}4|4^{\infty}) =\sqrt{\frac{13}{3}},
\]
which proves
\[
    L({}^{\infty}2^{\infty}) =
    L({}^{\infty}4^{\infty}) =
    \sqrt{\frac{13}{3}}.
\]
\subsubsection{}
Let $w = [2, 4]$ and $P_w = w^{\infty}$. Since
    \[
    N_2N_4 = 
    \begin{pmatrix} 5  & 4\sqrt{3}\\ 2\sqrt{3} & 5\end{pmatrix},
    \]
Proposition~\ref{PropPnormAndLP} gives
    \[
    L((P_{w^*})^*\vert P_w) = \frac{\sqrt{10^2- 4}}{2\sqrt{3}}=2\sqrt2.
    \]
    
  Next, we calculate $L(\dots 1515\vert 1515\dots)$. 
This sequence is \emph{not} reduced, so we cannot apply the second equation in Proposition~\ref{PropPnormAndLP} directly. 
Let
    $P=(51)^{\infty}$ and $Q = (15)^{\infty}$. 
    We begin with
    \[
    N_5N_1 = 
    \begin{pmatrix} 1  & \sqrt{3}\\ \sqrt{3} & 4\end{pmatrix}.
    \]
    So the first equation in Proposition~\ref{PropPnormAndLP} gives
    \[
    \| P \| 
    =  \frac{-3+ \sqrt{21}}{2\sqrt{3}}.
    \]
    Apply Proposition~\ref{PropStereographicProjection} to get
    \[
    \| Q \| 
    = 
    \| P^{\vee} \|
    = 
    \| P \|^{-1}
    = \frac{3+ \sqrt{21}}{2\sqrt{3}}.
    \]
    So,
    \begin{equation}\label{Eq15}
    \begin{aligned}
    L(\dots 1515\vert 1515\dots) &= L(P^*\vert Q) = \| (\hat{P})^{\vee} \| + \| Q \| \\
    &=
    \| P \| + \| Q \| 
    = \frac{\sqrt{21}}{\sqrt{3}}\\
    &= \sqrt7.
    \end{aligned}
    \end{equation}

\subsubsection{The case $T = {}^{\infty}2\, 3\, 2^{\infty}$}
We claim that $L(T)  = 4/\sqrt3$.
From (c) of Proposition~\ref{PropBasicLProperty}, it is enough for us to compute $L(P^*\vert Q)$ for every section $P^*\vert Q$ of $T$.
First, we compute 
\[
L( {}^{\infty}2\, 3|  2^{\infty}) =
\| 3\, 4^{\infty} \| + \| 2^{\infty} \|.
\]
For $\| 2^{\infty} \|$, we apply Proposition~\ref{PropPnormAndLP} with $P = 2^{\infty}$ to obtain
\[
	\| 2^{\infty} \|
	= 
	\frac{\sqrt{13} + 1}{2\sqrt3}.
\]
Then use Proposition~\ref{PropStereographicProjection} to get
\[
	\| 4^{\infty} \|
	=
	\| 2^{\infty} \|^{-1}
	= 
	\frac{\sqrt{13} - 1}{2\sqrt3}.
\]
From Proposition~\ref{PropActionNd} and the fact that 
$N_3 = \begin{psmallmatrix} 2 & \sqrt3 \\ \sqrt3 & 2 \end{psmallmatrix}$, we have
\[
	\| 3\, 4^{\infty} \| = 
	\frac{2 \| 4^{\infty} \| + \sqrt3}{\sqrt3\| 4^{\infty} \| + 2}
	=
	\frac{\sqrt3 (\sqrt{13} + 5)}{2(\sqrt{13}+4)}.
\]
Some easy simplification gives 
\[
L( {}^{\infty}2\, 3\vert   2^{\infty}) =
	\frac{\sqrt3 (\sqrt{13} + 5)}{2(\sqrt{13}+4)}
	+
	\frac{\sqrt{13} + 1}{2\sqrt3}
=\frac{4}{\sqrt3}.
\]
Next, we show that $L(P^*\vert Q)\le 4/\sqrt3$ for any section $P^*\vert Q \neq {}^{\infty}2\, 3\vert 2^{\infty}$ of $T$.
We consider ${}^{\infty}2\,3\, 2^k\vert 2^{\infty}$ for some $k\ge1$.
Notice from Proposition~\ref{CorCommonPrefixComparison} that 
$\| 4^k3\,2^{\infty}\| < \| 3\, 2^{\infty} \|$.
So,
\[
L({}^{\infty}2\,3\, 2^k\vert 2^{\infty} )
	= 
\| 4^k3\,2^{\infty}\| + \| 2^{\infty} \| <
 \| 3\, 2^{\infty} \| + \| 2^{\infty} \| 
=\frac4{\sqrt3}.
\]
The remaining sections to be considered are of type ${}^{\infty}2\vert 2^k 3\,2^{\infty}$ for some $k\ge1$.
We use Proposition~\ref{CorCommonPrefixComparison} again to get
$\| 2^k 3\, 2^{\infty} \| \le \| 2 \, 3\, 2^{\infty} \|$ for any $k\ge1$.
Then apply Proposition~\ref{PropCylinderMatrix}
with $N_2N_3 =
\begin{psmallmatrix}
7 & 4 \sqrt{3} \\
3 \sqrt{3} & 5
\end{psmallmatrix}
$
to obtain
$\| 2\, 3 \, 2^{\infty} \| \le \frac{4\sqrt3}5$.
So,
\begin{align*}
L({}^{\infty}2\vert 2^k 3\,2^{\infty})
	&=
	\|
	4^{\infty}
	\| 
	+
	\|
2^k 3\, 2^{\infty} 
\| 
	\\
	&
\le	
	\frac{\sqrt{13} - 1}{2\sqrt3}
	+
	\frac{4\sqrt3}5 =
	\frac{1}{30} \sqrt{3} {\left(5 \sqrt{13} + 19\right)} \\
	&
	< \frac4{\sqrt3}.
\end{align*}

\subsection{Admissible Sequences}

A doubly infinite Romik sequence $T$ is \emph{admissible} if
\[
L(T) \le \frac{4}{\sqrt3},
\]
and is \emph{strongly admissible} if 
\[
L(T) < \frac{4}{\sqrt3}.
\]
Note that $T$ is admissible (or strongly admissible) if and only if any one of the sequences $\{ T, T^{\vee}, \hat{T}^*, (\hat{T}^*)^{\vee}\}$ is admissible (or strongly admissible).

\begin{proposition}\label{PropFobidden15}
The digits 1 and 5 are forbidden in an admissible $T$.
\end{proposition}
\begin{proof}
It is enough to show that $1$ is forbidden.
First, we prove that $14$ is forbidden. 
If $T$ does contain $14$, we choose a section $P^*\vert Q$ of $T$ with $Q = 14\cdots.$
Apply Proposition~\ref{PropCylinderMatrix} with $w = [1, 4]$
and 
$N_1N_4 =
\begin{psmallmatrix}
4 & 3 \sqrt{3} \\
\sqrt{3} & 2
\end{psmallmatrix}
$ to obtain $\| Q \| \ge \frac{4}{\sqrt3}.$
As a consequence,
\[
L(P^*\vert Q) = \| \hat{P}^{\vee} \| + \| Q \| > \frac4{\sqrt3}.
\]
Similar arguments show that 11, 12, 13 are all forbidden in an admissible $T$.
Therefore, if an admissible $T$ contains $1$, it must extend to the right as 15.
This implies that $T^{\vee}$ must contain 51. 
So the sequence 51 in $T^{\vee}$ must extend as 515. Repeating this argument, we show that $T$ must admit a section
\[
P^*\vert (15)^{\infty}.
\]
By choosing another section of $T$ (cutting far to the right), we conclude from \eqref{Eq15} that
\[
L(T) \ge L(\dots 1515\vert 1515\dots) = \sqrt{7}> \frac{4}{\sqrt3},
\]
which shows that $T$ cannot be admissible.
\end{proof}
One easy consequence of the above proposition is that an admissible $T$ cannot terminate with $1^\infty$ or $5^{\infty}$ in either direction. In particular, when we apply Proposition~\ref{PropCylinderMatrix} to obtain a bound of a cylinder set, the inequalities are always strict.
\begin{proposition}\label{Prop24Forbidden}
The sequences $24$ and $42$ are forbidden in an admissible $T$.
\end{proposition}
\begin{proof}
As before, it is enough to show that $42$ is forbidden.
Assume that $T$ contains $42$. Choose a section $P^*4\vert 2Q$. Then the bound \eqref{EqCylinder4} gives 
\[
L(P^*4\vert 2Q) = \| 2\hat{P}^{\vee}\| + \|2Q \|
> \frac{2}{\sqrt3} + \frac{2}{\sqrt3} = \frac{4}{\sqrt3}.
\]
This shows that $T$ is not admissible.
\end{proof}

\begin{proposition}\label{Prop234Forbidden}
The sequences $234$ and $432$ are forbidden in an admissible $T$.
\end{proposition}
\begin{proof}
We will derive a contradiction by assuming that an admissible $T$ contains $432$.
Consider a section $P^*43\vert 2Q$ of $T$.
We have
\[
L(P^*43\vert 2Q) = \| 32\hat{P}^{\vee} \| + \| 2Q\|.
\]
To find a lower bound for $\| 32\hat{P}^{\vee} \|$, we invoke Proposition~\ref{PropFobidden15} with $w = [3, 2]$ and
$N_3N_2 = \begin{psmallmatrix}
7 & 3 \sqrt{3} \\
4 \sqrt{3} & 5
\end{psmallmatrix}$ to obtain 
\[
\| 32\hat{P}^{\vee} \| > \frac7{4\sqrt3}.
\]
For $\| 2Q \|$, we need a better bound than \eqref{EqCylinder4}.
Note from Proposition~\ref{PropFobidden15} that $P$ and $Q$ are \emph{reduced}, that is, they contain no 1 and 5. 
Therefore, $Q$ cannot begin with 1.
In particular, $\| 2Q \| \ge \| 22\cdots \|$.
Apply Proposition~\ref{PropCylinderMatrix} with 
$N_2N_2 = \begin{psmallmatrix}
7 & 3 \sqrt{3} \\
3 \sqrt{3} & 4\end{psmallmatrix}$
to obtain
\begin{equation}\label{EqCylinderBound4}
\| 2Q \| \ge \| 22\cdots \| >\frac{3\sqrt3}{4}.
\end{equation}
To sum up, we have
\[
L(P^*43\vert 2Q) > 
\frac{7}{4\sqrt3}
+ \frac{3\sqrt3}{4}
 = \frac{4}{\sqrt3},
\]
which contracts the assumption that $T$ is admissible.
\end{proof}
\begin{proposition}\label{Prop33Forbidden}
Suppose that $T \neq {}^{\infty}3^{\infty}$ and that $T$ is admissible.
Then the sequence $33$ is forbidden in $T$.
\end{proposition}
\begin{proof}
Since $T$ is reduced we may assume, by replacing $T$ with $T^{\vee}$ and/or $T^*$ if necessary, that $T$ contains a section $P^*33\vert 2Q$.
As before, we use
$N_3N_3 = 
\begin{psmallmatrix}
7 & 4 \sqrt{3} \\
4 \sqrt{3} & 7
\end{psmallmatrix}$ to yield
\[
\| 33P^{\vee} \| > \frac{4\sqrt3}{7}.
\]
For $\| 2Q \|$ in this case, we need a bit sharper bound than \eqref{EqCylinderBound4}, which came from the fact that $\| 2Q \|$ is bounded below by the lower bound of $C(2, 2)$.
To improve this bound, we note that, because $Q$ is reduced, $2Q$ cannot belong to the cylinder set $C(2, 2, 5)$.
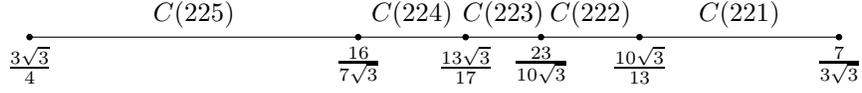
\begin{figure}
\begin{center}
	\begin{tikzpicture}[scale=1.7]
		\draw (226, 0) node[below]{$\frac{3\sqrt3}{4}$}-- (228.57, 0) 
		node[midway, above]{$C(225)$}
		node[below]{$\frac{16}{7\sqrt3}$} 
		-- (229.41, 0) 
		node[midway, above]{$C(224)$}
		node[below]{$\frac{13\sqrt3}{17}$} 
		-- (230, 0) 
		node[midway, above]{$C(223)$}
		node[below]{$\frac{23}{10\sqrt3}$} 
		-- (230.77, 0) 
		node[midway, above]{$C(222)$}
		node[below]{$\frac{10\sqrt3}{13}$} 
		-- (232.33, 0) 
		node[midway, above]{$C(221)$}
		node[below]{$\frac{7}{3\sqrt3}$} ;
		\foreach \a in {226.0, 228.57, 229.41, 230, 230.77, 232.33}
		\filldraw (\a, 0) circle(0.6pt);
\end{tikzpicture}
\end{center}
    \caption{Images of $C(2,2)$ and of its subsets $C(2,2,1)$, $\dots$, $C(2,2,5)$. 
    The commas in the cylinder sets in the figure are suppressed for a typographical reason.}
    \label{FigCylinderSetsLine3}
\end{figure}
Therefore, $\| 2Q \|$ is bounded below by the lower bound of 
$C(2, 2, 4)$.
See Figure~\ref{FigCylinderSetsLine3}.

Since
$N_2N_2N_4 = 
\begin{psmallmatrix}
16 & 13 \sqrt{3} \\
7 \sqrt{3} & 17
\end{psmallmatrix}$
the lower bound of $C(2, 2, 4)$ is $\frac{16}{7\sqrt3}$
and therefore
\begin{equation}\label{EqCylinderBound442}
\| 2Q \| 
\ge \frac{16}{7\sqrt3}.
\end{equation}
So,
\[
L(P^*33\vert 2Q) = \| 33P^{\vee} \|+\| 2Q \| > 
\frac{4\sqrt3}{7}
+\frac{16}{7\sqrt3}
=\frac{4}{\sqrt3}.
\]
\end{proof}

\subsection{Lagrange numbers of admissible sequences}
We characterize all doubly infinite admissible sequences in the following theorem.
\begin{theorem}\label{ThmCharacterizationOfAdmissibleSeq}
Suppose that $T$ is admissible. Then either $T$ or $T^{\vee}$ is equal to
\begin{enumerate}[font=\upshape, label=(\alph*)]
    \item ${}^{\infty}2^{\infty}$, ${}^{\infty}3^{\infty}$,
    \item ${}^{\infty}2 \, 3 \, 2^{\infty}$, or
    \item $
    \cdots 3\, 2^{2k}\, 3 \, 2^{2k}\, 3 \cdots,
    $
    for a fixed positive integer $k$.
\end{enumerate}
\end{theorem}
We have already proven in \S\ref{SecLExamples} that 
${}^{\infty}2^{\infty}$,
${}^{\infty}3^{\infty}$,
and
${}^{\infty}4^{\infty}$
are strongly admissible
and that 
    ${}^{\infty}2 \, 3 \, 2^{\infty}$
is admissible.
We will now assume that $T$ is non-constant (meaning that $T$ is not equal to any of
${}^{\infty}2^{\infty}$,
${}^{\infty}3^{\infty}$,
${}^{\infty}4^{\infty}$)
and admissible, therefore reduced. Also, by replacing $T$ with $T^{\vee}$ if necessary, we may assume that $T$ contains 2 as one of its digits.

Pick a digit 2 in $T$ and continue reading the subsequent digits to the right until a consecutive sequence of 2 is broken. Either $T$ terminates (to the right) with $2^\infty$ or $T$ must have as a section
\[
P^*\vert Q = P^*\vert 2^k 3 \cdots,
\]
because 24 is forbidden by Proposition~\ref{Prop24Forbidden}.
Since 33 and 234 are also forbidden by Propositions~\ref{Prop33Forbidden} and \ref{Prop234Forbidden} we see that $Q$ must be of the form $Q = 2^k3\, 2 \cdots$.
By repeating the same argument indefinitely we conclude that
\[
Q = 2^{k_1}\,3\,2^{k_2}\,3\,2^{k_3}3\cdots
\]
where $k_1, k_2, \dots$ are positive integers or $k_j = \infty$ for some $j$
(which means that $Q$ terminates with $2^{\infty}$ at that place).
Apply the same argument to $T^*$ to get
\begin{equation}\label{EqPossibleT}
T = \cdots2^{k_{-1}}\,3\,2^{k_0}\,3\,2^{k_1}3\cdots.
\end{equation}
To complete the proof of Theorem~\ref{ThmCharacterizationOfAdmissibleSeq}, we must
show that all $k_j$ are equal to the same positive even integer or two consecutive $k$'s are both equal to $\infty$. 
Pick two consecutive exponents of $2$ and call them $m$ and $k$, so that
\begin{equation}\label{EqTwithm}
T = \cdots 3 \, 2^m \,3\, 2^k \,3 \cdots.
\end{equation}
If $k=m=\infty$, we are done. So, we will assume from now on that at least one of them is $<\infty$.
We prove a series of propositions below (Propositions ~\ref{PropMgreaterThanKodd}---\ref{PropTheEnd}), which will collectively show that $k$ and $m$ must be equal to a (common) even integer.

An important technical lemma in our proof is to give explicit expressions for $N_2^kN_3$ and $N_3N_4^k$ for $k\ge0$.
\begin{proposition}\label{PropN4k}
For each $k\ge0$, we have
\[
N_2^kN_3 = 
\begin{pmatrix}
4c_k - d_k & \sqrt3(3c_k - d_k) \\
\sqrt3c_k & d_k \\
\end{pmatrix},
\]
and
\[
N_3N_4^k = 
\begin{pmatrix}
  d_k & \sqrt3(3c_k - d_k) \\
\sqrt3c_k &4c_k - d_k\\
\end{pmatrix}.
\]
Here, the sequences $\{c_k\}$ and $\{d_k\}$ are given by
\[
c_k = \frac{1}{\sqrt{13}}
\left(
\lambda^{k+1}
-
\overline{\lambda}^{k+1}
\right),
\]
and
\[
d_k = \frac{1}{\sqrt{13}}
\left[
\left(
\frac{7 - \sqrt{13}}{2}
\right)
\lambda^{k+1}
-
\left(
\frac{7 + \sqrt{13}}{2}
\right)
\overline{\lambda}^{k+1}
\right]
\]
where
\[
\lambda=\frac{3 + \sqrt{13}}{2},
\quad
\text{and}
\quad
\overline{\lambda}=\frac{3 - \sqrt{13}}{2}.
\]
Further,
\[
9c_k^2 - 7c_kd_k + d_k^2 = (-1)^{k+1}.
\]
\end{proposition}
\begin{proof}
The formula for $N_2^k N_3$ can be proven by induction as follows.
The case for $k=0$ is easy.
The induction hypothesis shows that
 \[
 \begin{pmatrix}
 c_{k+1} \\
 d_{k+1} \\
 \end{pmatrix}=
 \begin{pmatrix}
 5 & -1 \\ 9 & -2 \\
 \end{pmatrix}
 \begin{pmatrix}
 c_k \\ d_k
 \end{pmatrix}
 \]
with $c_0 = 1$ and $d_0 = 2.$
Then we show that the given expressions for $c_k$ and $d_k$ satisfy the above recursive formula.
We omit the detail.

For $N_3N_4^k$, we use Lemma~\ref{LemMatrixTranspose} with the fact
\[
((2^k 3)^*)^{\vee}  = 3\, 4^k.
\]
The last equation is equivalent to the fact that $\det(N_2^kN_3) = (-1)^k$.
\end{proof}

\begin{proposition}\label{PropMgreaterThanKodd}
In the expression \eqref{EqTwithm}, if $k$ is odd and if $k\le m \le \infty$ then $T$ is not admissible.
\end{proposition}
\begin{proof}
Write
$
P = 3\, 2^m \, 3\cdots
$
and
$Q = 2^k \,3\cdots,$
so that
\[
P^*\vert Q = 
\cdots 3 \, 2^m \,3\vert  2^k \,3 \cdots 
\]
is a section of $T$.
Then,
\[
L(P^*\vert Q) = \|P^{\vee}\| + \| Q \|= \| 3\, 4^m\cdots\| + \| 2^k 3\cdots\|.
\]
We find the lower bounds of the two terms above using Proposition~\ref{PropN4k}. 
Since $k$ is odd we have 
$\det(N_2^kN_3) = \det(N_3N_4^k) = -1$ and
\begin{equation}\label{Eq4kLowerBound}
\| 2^k 3\cdots\| > \frac{4c_k - d_k}{\sqrt3 c_k},    
\end{equation}
and
\[
\| 3\, 4^{m}\cdots\| =
\| 3\, 4^{k} \, 4^{m-k} \cdots\| =
\| 3\, 4^{k}\cdots\|  >\frac{d_k}{\sqrt3 c_k}.
\]
So,
\[
L(P^*\vert Q)>
\frac{4c_k - d_k}{\sqrt3 c_k} + 
\frac{d_k}{\sqrt3 c_k} = \frac{4}{\sqrt3}.
\]
This shows that $T$ is not admissible.
\end{proof}
\begin{proposition}\label{PropMlessEvenThanKodd}
In the expression \eqref{EqTwithm}, if $k$ is odd and if $m$ is even then $T$ is not admissible.
\end{proposition}
\begin{proof}
Proposition~\ref{PropMgreaterThanKodd} implies that $m>k$ is impossible for an admissible $T$. So, we will assume that $m < k$.

First, Proposition~\ref{CorCommonPrefixComparison} says that
\[
\| 3\, 4^{m}\,3 \cdots \| \ge
\| 3\, 4^{m} 4^{(k-1)-m} \,3 \cdots \|=
\| 3\,4^{k-1} \,3 \cdots \|.
\]
To find a lower bound for $\| 3\,4^{k-1} 3\cdots\|$,
we use Proposition~\ref{PropN4k} to get
\begin{align*}
N_3N_4^{k-1}N_3 &=(N_3 N_4^k)(N_4^{-1} N_3) =
\begin{pmatrix}
  d_k & \sqrt3(3c_k - d_k) \\
\sqrt3c_k &4c_k - d_k\\
\end{pmatrix}
\begin{pmatrix}
  -1& 0 \\
\sqrt3 &1\\
\end{pmatrix} \\
&=
\begin{pmatrix}
  9c_k - 4d_k & \sqrt3(3c_k - d_k) \\
\sqrt3(3c_k-d_k) &4c_k - d_k\\
\end{pmatrix}.
\end{align*}
Since $\det(N_3N_4^{k-1}N_3) = 1$ we have
\begin{equation}\label{Eq32k_13kLowerBound}
\| 3\,4^{k-1} 3\cdots\| > \frac{\sqrt{3}(3c_k-d_k)}{4c_k - d_k}.
\end{equation}
On the other hand, 
combining \eqref{Eq4kLowerBound} and \eqref{Eq32k_13kLowerBound}, we have
\begin{align*}
L(P^*\vert Q)-\frac{4}{\sqrt3} &= \| 3\,4^{m} 3\cdots\| + \| 2^k 3\cdots\| -\frac{4}{\sqrt3}\\
&> \frac{\sqrt{3}(3c_k-d_k)}{4c_k - d_k} +
\frac{4c_k - d_k}{\sqrt3 c_k} 
-\frac{4}{\sqrt3} \\
&= \frac{9c_k^2 - 7c_kd_k + d_k^2}{\sqrt{3}(4c_k-d_k)}
= \frac1{\sqrt{3}(4c_k-d_k)}>0.
    \end{align*}
Here, the last equality is obtained from Proposition~\ref{PropN4k}, together with the assumption that $k$ is odd.
\end{proof}
\begin{proposition}\label{PropKodd}
In the expression \eqref{EqTwithm}, if $k$ is odd then $T$ is not admissible.
\end{proposition}
\begin{proof}
The only remaining case (after Propositions~\ref{PropMgreaterThanKodd} and \ref{PropMlessEvenThanKodd}) is when both $k$ and $m$ are odd and $m < k$.
But, in this case, $T^*$ would contain a pattern which is forbidden by Proposition ~\ref{PropMgreaterThanKodd}.
\end{proof}
Now we consider the case when both $k$ and $m$ are even and distinct. Replacing $T$ by $T^*$ if necessary, we may assume without loss of generality that $m<k$. The next proposition completes the proof of Theorem~\ref{ThmCharacterizationOfAdmissibleSeq}.
\begin{proposition}\label{PropTheEnd}
In the expression \eqref{EqTwithm}, if both $k$ and $m$ are even and if $m<k$ (possibly $k=\infty$) then $T$ is not admissible.
\end{proposition}
\begin{proof}
We handle the case $k<\infty$ first. 
Since $m\le k-2$, Proposition~\ref{CorCommonPrefixComparison} implies that
\[
\| 3\,4^m\, 3\cdots \| 
\ge
\| 3\,4^m \,4^{(k-2)-m}\, 3\cdots \|
=
\| 3\,4^{k-2}\, 3\cdots \|.
\]
We need to find a lower bound of $\| 3\,2^{k-2}\, 3\cdots \|.$
To do so,
\begin{align*}
N_3N_4^{k-2}N_3 &=(N_3 N_4^k)(N_4^{-2} N_3) =
\begin{pmatrix}
  d_k & \sqrt3(3c_k - d_k) \\
\sqrt3c_k &4c_k - d_k\\
\end{pmatrix}
\begin{pmatrix}
  5 & 3 \\
-2\sqrt3 & -1\\
\end{pmatrix} \\
&=
\begin{pmatrix}
  -18c_k + 11d_k & \sqrt3(-3c_k +2 d_k) \\
\sqrt3(-3c_k+2d_k) &-c_k + d_k\\
\end{pmatrix}.
\end{align*}
Since $\det(N_3N_4^{k-2}N_3) = 1$, we have
\begin{equation}\label{Eq32k_23kLowerBound}
\| 3\,4^{k-2} 3\cdots\| > \frac{\sqrt{3}(-3c_k+2d_k)}{-c_k + d_k}.
\end{equation}
For $\|2^k\,3\cdots\|$, we note that $\det(N_2^kN_3) = 1$. Proposition~\ref{PropN4k} gives
\begin{equation}\label{EqLowerBoundN4kKeven}
    \|2^k\,3\cdots\| > \frac{\sqrt3(3c_k - d_k)}{d_k}.
\end{equation}
Let $P = 3\, 2^{m}\, 3\cdots$ and $Q = 2^k\,3\cdots$.
We use \eqref{Eq32k_23kLowerBound} and \eqref{EqLowerBoundN4kKeven} to obtain
\begin{align*}
L(P^*\vert Q) -\frac{4}{\sqrt3} &= \| 3\, 4^{m}\, 3\cdots \| + \| 2^k\,3\cdots \|
-\frac{4}{\sqrt3} \\
& > 
\frac{\sqrt3(-3c_k +2 d_k)}{-c_k + d_k}
+
\frac{\sqrt3(3c_k - d_k)}{d_k} 
-\frac{4}{\sqrt3} \\
&=
\frac{-9c_k^2 +7c_kd_k - d_k^2}{\sqrt{3}d_k(d_k - c_k)} \\
&=
\frac1{\sqrt{3}d_k(d_k - c_k)}>0.
\end{align*}

Suppose that $m< k = \infty$. Let $Q_{2j} = 2^{2j}3 \cdots$ (with an arbitrarily chosen tail). Then Proposition~\ref{CorCommonPrefixComparison} implies that $\{\| Q_{2j} \|\}$ is a monotonically increasing sequence (regardless of the choice of tail in $Q_{2j}$), converging to  $\| 2^{\infty} \|$.
So, for any given $\epsilon>0$,
\begin{align*}
L(P^*\vert 2^{\infty}) - \frac{4}{\sqrt3}- \epsilon 
&=
\| P^{\vee} \| + \| 2^{\infty} \| - \epsilon - \frac{4}{\sqrt3} \\
& \ge
\| P^{\vee} \| + \| Q_{2j} \|
-\frac{4}{\sqrt3} \\
&= L(P^*\vert Q_{2j}) - \frac{4}{\sqrt3} \\
&>0
\end{align*}
for all sufficiently large $j$.
\end{proof}


\begin{theorem}[Main Theorem]\label{ThmMainTheorem}
A doubly infinite Romik sequence $T$ is admissible if and only if $T$ or $T^{\vee}$ is equal to
one of the sequences in Theorem~\ref{ThmCharacterizationOfAdmissibleSeq}.
Their Lagrange numbers are
\begin{enumerate}[font=\upshape, label=(\alph*)]
    \item $L({}^{\infty}3^{\infty}) = 2$, and $L({}^{\infty}2^{\infty}) = L({}^{\infty}4^{\infty}) = \sqrt{13/3}$.
    \item $L({}^{\infty}2 \, 3 \, 2^{\infty}) = 4/\sqrt{3}$.
    \item For each $k>0$,
    \[
    L(\cdots 3\, 2^{2k}\, 3 \, 2^{2k}\, 3 \cdots) =
    \frac{2}{\sqrt{3}}
\sqrt{4 - \frac{13}%
{\left(\lambda^{2k+1} - \overline{\lambda}^{2k+1}\right)^2}}.
    \]
\end{enumerate}
Here, $\lambda$ and $\overline{\lambda}$ are the constants defined in Proposition~\ref{PropN4k}.
\end{theorem}
\begin{proof}
We have already computed in \S\ref{SecLExamples} the Lagrange numbers of the sequences in (a) and (b).
So, we need to find the Lagrange number of 
\[
T_{2k} = \cdots 3\, 2^{2k}\, 3 \, 2^{2k}\, 3 \cdots.
\]
Since 
$T_{2k}$ is reduced and
$T_{2k} = T_{2k}^*$,
Proposition~\ref{PropBasicLProperty} says that
we only need to compute $L(P^*\vert Q)$ for the following types of $P$ and $Q$:
\begin{enumerate}[font=\upshape, label=(\Alph*)]
\item $P_A = (3\, 2^{2k})^{\infty}$ and $Q_A = (2^{2k}3)^{\infty}$,
\item $P_B = (2^{2k}3)^{\infty}$ and $Q_B = (3\,2^{2k})^{\infty}$,
\item $P_C = (2^{2k-m}\, 3 \, 2^m)^{\infty}$ and $Q_C=(2^m\, 3 \,2^{2k-m})^{\infty}$ with $0< m< 2k$.
\end{enumerate}
For the case (A),
we deduce from Propositions~\ref{PropPnormAndLP} and \ref{PropN4k} that
\begin{equation}\label{EqLValue}
\begin{aligned}
L(P_A^*\vert Q_A)
&=
\frac{\sqrt{(4c_{2k})^2 - 4}}{\sqrt3 c_{2k}} \\
&=
\frac{2}{\sqrt{3}}
\sqrt{4 - \frac{1}{c_{2k}^2}} \\
&=
\frac{2}{\sqrt{3}}
\sqrt{4 - \frac{13}%
{\left(\lambda^{2k+1} - \overline{\lambda}^{2k+1}\right)^2}},
\end{aligned}
\end{equation}
which is the value stated in the theorem.
So, in order to complete the proof, we will show that 
\begin{equation}\label{EqTwoInequalities}
L(P_A^*\vert Q_A) \ge L(P_B^*\vert Q_B) \quad \text{and} \quad
L(P_A^*\vert Q_A) \ge L(P_C^*\vert Q_C).
\end{equation}
Notice from \eqref{EqNd} (or simply from Figure~\ref{FigCylinderSetsLine}) that
\[
\| 3\cdots \| \le \frac2{\sqrt3} \quad \text{and} \quad
\| 4\cdots \| \le \frac{\sqrt3}2.
\]
So, we have 
	\begin{align*}
	L(P_B^* \vert Q_B) &=
\| 4^{2k}3 \cdots \| + \| 3\, 2^{2k} \cdots \| 
\le 
\frac2{\sqrt3}+
	\frac{\sqrt3}2
=
\frac7{2\sqrt3} = 2.0207\cdots  \\
		&<\frac{\sqrt{133}}5
		\le
	      L(P_A^*\vert Q_A),
	\end{align*}
which establishes the first inequality in \eqref{EqTwoInequalities}. 
For the second, 
\begin{equation}\label{EqLAminusLC}
\begin{aligned}
L(P_A^*\vert Q_A) &- L(P_C^*\vert Q_C) = \| P_A^{\vee}\| - \| P_C^{\vee}\| + \| Q_A \| - \| Q_C \| \\
&=
(\| 34\cdots \| - \| 4^{2k-m}3\cdots\|)
+
(\| 22\cdots \| - \| 2^m3\cdots\|).
\end{aligned}
\end{equation}
Observe that
\[
\begin{aligned}
N_2N_2 &=
\begin{pmatrix}
  7 & 3\sqrt3 \\
3\sqrt3 & 4\\
\end{pmatrix},
&
\frac{3\sqrt3}{4} < \| 22\cdots \| < \frac{7}{3\sqrt3}, \\
N_3N_4 &=
\begin{pmatrix}
  5 & 4\sqrt3 \\
3\sqrt3 & 7\\
\end{pmatrix},
&
\frac{5}{3\sqrt3} < \| 34\cdots \| < \frac{4\sqrt3}7, \\
N_4N_3 &=
\begin{pmatrix}
  5 & 3\sqrt3 \\
4\sqrt3 & 7\\
\end{pmatrix},
&
\frac5{4\sqrt3} < \| 43\cdots \| < \frac{3\sqrt3}7. \\
\end{aligned}
\]
Therefore, 
\[
\| 34 \cdots \| 
- \| 4^{2k-m}3\cdots\| \ge 
\| 34 \cdots \| - \| 43\cdots \| \ge 
\frac{5}{3\sqrt3}
-
\frac{3\sqrt3}{7}
=
\frac{8\sqrt3}{63}
\]
and
\[
\| 22 \cdots \| 
- \| 2^{m}3\cdots\| \ge 
\| 22 \cdots \| - \| 22\cdots \| \ge 
\frac{3\sqrt3}{4}
-
\frac{7}{3\sqrt3}
=
-\frac{\sqrt3}{36}.
\]
So, \eqref{EqLAminusLC} becomes
\[
L(P_A^*\vert Q_A) - L(P_C^*\vert Q_C) 
\ge
\frac{8\sqrt3}{63}
-\frac{\sqrt3}{36}>0.
\]
This completes proving \eqref{EqTwoInequalities}, thus the theorem is now proven.
\end{proof}

\end{document}